\documentclass[12pt,reqno,a4paper]{amsart}
\usepackage{amsmath,amssymb,amsfonts,amscd}
\usepackage[mathscr]{eucal}
\usepackage{yfonts}
\usepackage{tensor}
\usepackage{tikz}
\usepackage{tikz-cd}
\usepackage{arydshln}
\usepackage{hyperref,xcolor} 
\usepackage{comment}

\usetikzlibrary{shapes,decorations,arrows,calc,arrows.meta,fit,positioning}
\usetikzlibrary{snakes}
\tikzset{
    -Latex,auto,node distance =1 cm and 1 cm,thick,
    state/.style ={rectangle, draw, minimum width = 0.7 cm},
    point/.style = {rectangle, draw, inner sep=0.04cm,fill,node contents={}},
    odd/.style={Latex-Latex,draw=blue},
    even/.style={Latex-Latex,draw=black},
    el/.style = {inner sep=2pt, align=left, sloped}
}

\def\co{\colon\thinspace}
\newcommand{\cat}[1] {\textcolor{red}{#1}}

\newtheorem{theorem}{Theorem}
\newtheorem{lemma}{Lemma}
\newtheorem{proposition}{Proposition}
\newtheorem{corollary}{Corollary}

\theoremstyle{definition}
\newtheorem{example}{Example}
\newtheorem{remark}{Remark}

\newcommand{\Z}{\mathbb{Z}}
\newcommand{\N}{\mathbb{N}}

\newcommand{\ZZ}{{\mathbb Z}_2}

\newcommand{\la}{{\lambda}}

\newcommand{\p}{\partial}

\renewcommand{\Im}{{\mathrm{Im}}}

\newcommand{\ach}{{(\underline{a})}}
\newcommand{\bch}{{(\underline{b})}}

\newcommand{\aund}{{\underline{a}}}

\newcommand{\bund}{{\underline{b}}}

\newcommand{\hmund}{{\underline{\hat \mu}}}
\newcommand{\hnund}{{\underline{\hat \nu}}}

\newcommand{\der}[2]{{\frac{\partial {#1}}{\partial {#2}}}}

\DeclareMathOperator{\cof}{cofactor}
\DeclareMathOperator{\id}{id}
\newcommand{\St}{\widetilde{\mathop{St}}}

\DeclareMathOperator{\Ber}{Ber}

\DeclareMathOperator{\spa}{span}
\DeclareMathOperator{\GL}{GL}

\DeclareMathOperator{\cl}{cl}

\newcommand{\lder}[2]{{\partial {#1}/\partial {#2}}}
\newcommand{\dder}[3]{{\frac{\partial^2 {#1}}{\partial {#2}\partial {#3}}}}

\newcommand{\ldder}[3]{{{\partial^2 {#1}}/{\partial {#2}\partial {#3}}}}

\newcommand{\R}[1]{{\mathbb R}^{#1}}
\newcommand{\RR}{\mathbb R}
\newcommand{\CC}{\mathbb C}
\newcommand{\fun}{C^{\infty}}

\newcommand{\al}{{\alpha}}
\newcommand{\be}{{\beta}}

\renewcommand{\d}{\delta}
\newcommand{\h}{{\eta}}
\newcommand{\e}{{\varepsilon}}
\renewcommand{\t}{{\theta}}

\newcommand{\x}{{\xi}}

\newcommand{\wed}{\wedge}

\renewcommand{\L}{{\Lambda}}

\newcommand{\itt}{{\tilde\imath}}
\newcommand{\jt}{{\tilde\jmath}}

\newcommand{\at}{{\tilde a}}
\newcommand{\bt}{{\tilde b}}
\newcommand{\ct}{{\tilde c}}
\newcommand{\dt}{{\tilde d}}
\newcommand{\pt}{{\tilde p}}

\newcommand{\vt}{{\tilde v}}
\newcommand{\wt}{{\tilde w}}

\DeclareMathOperator{\pl}{\text{\textswab{pl\"uck}}}

\DeclareMathOperator{\Pl}{\text{\textswab{Pl\"uck}}}

\newcommand{\xp}{\boldsymbol{x}}
\newcommand{\up}{\boldsymbol{u}}
\renewcommand{\wp}{\boldsymbol{w}}
\newcommand{\ep}{{e}}

\newcommand{\hnu}{\hat\nu}
\newcommand{\hmu}{\hat\mu}
\newcommand{\hb}{\hat\beta}

\newcommand{\hla}{\hat\lambda}

\newcommand{\pfin}{P_{1,-1}^{\text{\textrm{fin}}}\left(\L^{r|s}(V)\oplus \L^{s|r}(\Pi V)\right)}
\newcommand{\pfinth}{P_{1,-1}^{\text{\emph{fin}}}\left(\L^{r|s}(V)\oplus \L^{s|r}(\Pi V)\right)}
\newcommand{\pfinoo}{P_{1,-1}^{\text{\textrm{fin}}}\left(\L^{1|1}(V)\oplus \L^{1|1}(\Pi V)\right)}
\newcommand{\pess}{P^{\text{\textrm{ess}}}\left(\L^{r}(V)\right)}
\newcommand{\pessth}{P^{\text{\emph{ess}}}\left(\L^{r}(V)\right)}

\newcommand{\ocot}
{\begin{tikzpicture}
\path (90:1cm) coordinate (P1);
\path (162:1cm) coordinate (P2);
\path (234:1cm) coordinate (P3);
\path (306:1cm) coordinate (P4);
\path (18:1cm) coordinate (P5);
\draw (P1) -- (P2) -- (P3) -- (P4) -- (P5) -- cycle;
\node [anchor=south] at (P1)   {$1$};
\node [anchor=east] at (P2)   {$2$};
\node [anchor=east] at (P3)   {$3$};
\node [anchor=west] at (P4)   {$4$};
\node [anchor=west] at (P5)   {$5$};
\path [draw, very thick] (P1) -- (P4) -- cycle;
\path [draw, very thick] (P1) -- (P3) -- cycle;
\filldraw[red] (90:1cm) circle (2pt);
\filldraw[red] (306:1cm) circle (2pt);
\end{tikzpicture}}

\newcommand{\otoc}
{\begin{tikzpicture}
\path (90:1cm) coordinate (P1);
\path (162:1cm) coordinate (P2);
\path (234:1cm) coordinate (P3);
\path (306:1cm) coordinate (P4);
\path (18:1cm) coordinate (P5);
\draw (P1) -- (P2) -- (P3) -- (P4) -- (P5) -- cycle;
\node [anchor=south] at (P1)   {$1$};
\node [anchor=east] at (P2)   {$2$};
\node [anchor=east] at (P3)   {$3$};
\node [anchor=west] at (P4)   {$4$};
\node [anchor=west] at (P5)   {$5$};
\path [draw, very thick] (P1) -- (P4) -- cycle;
\path [draw, very thick] (P1) -- (P3) -- cycle;
\filldraw[red] (90:1cm) circle (2pt); 
\filldraw[red] (234:1cm) circle (2pt);
\end{tikzpicture}}

\newcommand{\tpot}
{\begin{tikzpicture}
\path (90:1cm) coordinate (P1);
\path (162:1cm) coordinate (P2);
\path (234:1cm) coordinate (P3);
\path (306:1cm) coordinate (P4);
\path (18:1cm) coordinate (P5);
\draw (P1) -- (P2) -- (P3) -- (P4) -- (P5) -- cycle;
\node [anchor=south] at (P1)   {$1$};
\node [anchor=east] at (P2)   {$2$};
\node [anchor=east] at (P3)   {$3$};
\node [anchor=west] at (P4)   {$4$};
\node [anchor=west] at (P5)   {$5$};
\path [draw, very thick] (P3) -- (P5) -- cycle;
\path [draw, very thick] (P1) -- (P3) -- cycle;
\filldraw[red] (234:1cm) circle (2pt); 
\filldraw[red] (18:1cm) circle (2pt); 
\end{tikzpicture}}

\newcommand{\ottp}
{\begin{tikzpicture}
\path (90:1cm) coordinate (P1);
\path (162:1cm) coordinate (P2);
\path (234:1cm) coordinate (P3);
\path (306:1cm) coordinate (P4);
\path (18:1cm) coordinate (P5);
\draw (P1) -- (P2) -- (P3) -- (P4) -- (P5) -- cycle;
\node [anchor=south] at (P1)   {$1$};
\node [anchor=east] at (P2)   {$2$};
\node [anchor=east] at (P3)   {$3$};
\node [anchor=west] at (P4)   {$4$};
\node [anchor=west] at (P5)   {$5$};
\path [draw, very thick] (P3) -- (P5) -- cycle;
\path [draw, very thick] (P1) -- (P3) -- cycle;
\filldraw[red] (234:1cm) circle (2pt); 
\filldraw[red] (90:1cm) circle (2pt); 
\end{tikzpicture}}

\newcommand{\dptp}
{\begin{tikzpicture}
\path (90:1cm) coordinate (P1);
\path (162:1cm) coordinate (P2);
\path (234:1cm) coordinate (P3);
\path (306:1cm) coordinate (P4);
\path (18:1cm) coordinate (P5);
\draw (P1) -- (P2) -- (P3) -- (P4) -- (P5) -- cycle;
\node [anchor=south] at (P1)   {$1$};
\node [anchor=east] at (P2)   {$2$};
\node [anchor=east] at (P3)   {$3$};
\node [anchor=west] at (P4)   {$4$};
\node [anchor=west] at (P5)   {$5$};
\path [draw, very thick] (P2) -- (P5) -- cycle;
\path [draw, very thick] (P3) -- (P5) -- cycle;
\filldraw[red] (162:1cm) circle (2pt); 
\filldraw[red] (18:1cm) circle (2pt); 
\end{tikzpicture}}

\newcommand{\tpdp}
{\begin{tikzpicture}
\path (90:1cm) coordinate (P1);
\path (162:1cm) coordinate (P2);
\path (234:1cm) coordinate (P3);
\path (306:1cm) coordinate (P4);
\path (18:1cm) coordinate (P5);
\draw (P1) -- (P2) -- (P3) -- (P4) -- (P5) -- cycle;
\node [anchor=south] at (P1)   {$1$};
\node [anchor=east] at (P2)   {$2$};
\node [anchor=east] at (P3)   {$3$};
\node [anchor=west] at (P4)   {$4$};
\node [anchor=west] at (P5)   {$5$};
\path [draw, very thick] (P2) -- (P5) -- cycle;
\path [draw, very thick] (P3) -- (P5) -- cycle;
\filldraw[red] (234:1cm) circle (2pt); 
\filldraw[red] (18:1cm) circle (2pt); 
\end{tikzpicture}}

\newcommand{\dcdp}
{\begin{tikzpicture}
\path (90:1cm) coordinate (P1);
\path (162:1cm) coordinate (P2);
\path (234:1cm) coordinate (P3);
\path (306:1cm) coordinate (P4);
\path (18:1cm) coordinate (P5);
\draw (P1) -- (P2) -- (P3) -- (P4) -- (P5) -- cycle;
\node [anchor=south] at (P1)   {$1$};
\node [anchor=east] at (P2)   {$2$};
\node [anchor=east] at (P3)   {$3$};
\node [anchor=west] at (P4)   {$4$};
\node [anchor=west] at (P5)   {$5$};
\path [draw, very thick] (P2) -- (P5) -- cycle;
\path [draw, very thick] (P2) -- (P4) -- cycle;
\filldraw[red] (162:1cm) circle (2pt); 
\filldraw[red] (306:1cm) circle (2pt); 
\end{tikzpicture}}

\newcommand{\dpdc}
{\begin{tikzpicture}
\path (90:1cm) coordinate (P1);
\path (162:1cm) coordinate (P2);
\path (234:1cm) coordinate (P3);
\path (306:1cm) coordinate (P4);
\path (18:1cm) coordinate (P5);
\draw (P1) -- (P2) -- (P3) -- (P4) -- (P5) -- cycle;
\node [anchor=south] at (P1)   {$1$};
\node [anchor=east] at (P2)   {$2$};
\node [anchor=east] at (P3)   {$3$};
\node [anchor=west] at (P4)   {$4$};
\node [anchor=west] at (P5)   {$5$};
\path [draw, very thick] (P2) -- (P5) -- cycle;
\path [draw, very thick] (P2) -- (P4) -- cycle;
\filldraw[red] (162:1cm) circle (2pt); 
\filldraw[red] (18:1cm) circle (2pt); 
\end{tikzpicture}}

\newcommand{\ocdc}
{\begin{tikzpicture}
\path (90:1cm) coordinate (P1);
\path (162:1cm) coordinate (P2);
\path (234:1cm) coordinate (P3);
\path (306:1cm) coordinate (P4);
\path (18:1cm) coordinate (P5);
\draw (P1) -- (P2) -- (P3) -- (P4) -- (P5) -- cycle;
\node [anchor=south] at (P1)   {$1$};
\node [anchor=east] at (P2)   {$2$};
\node [anchor=east] at (P3)   {$3$};
\node [anchor=west] at (P4)   {$4$};
\node [anchor=west] at (P5)   {$5$};
\path [draw, very thick] (P1) -- (P4) -- cycle;
\path [draw, very thick] (P2) -- (P4) -- cycle;
\filldraw[red] (90:1cm) circle (2pt); 
\filldraw[red] (306:1cm) circle (2pt); 
\end{tikzpicture}}

\newcommand{\dcoc}
{\begin{tikzpicture}
\path (90:1cm) coordinate (P1);
\path (162:1cm) coordinate (P2);
\path (234:1cm) coordinate (P3);
\path (306:1cm) coordinate (P4);
\path (18:1cm) coordinate (P5);
\draw (P1) -- (P2) -- (P3) -- (P4) -- (P5) -- cycle;
\node [anchor=south] at (P1)   {$1$};
\node [anchor=east] at (P2)   {$2$};
\node [anchor=east] at (P3)   {$3$};
\node [anchor=west] at (P4)   {$4$};
\node [anchor=west] at (P5)   {$5$};
\path [draw, very thick] (P1) -- (P4) -- cycle;
\path [draw, very thick] (P2) -- (P4) -- cycle;
\filldraw[red] (162:1cm) circle (2pt); 
\filldraw[red] (306:1cm) circle (2pt); 
\end{tikzpicture}}

\newcommand{\sqot}
{\begin{tikzpicture}
\path (45:1cm) coordinate (P1);
\path (135:1cm) coordinate (P2);
\path (225:1cm) coordinate (P3);
\path (315:1cm) coordinate (P4);
\draw (P1) -- (P2) -- (P3) -- (P4) -- cycle;
\node [anchor=west] at (P1)   {$1$};
\node [anchor=east] at (P2)   {$2$};
\node [anchor=east] at (P3)   {$3$};
\node [anchor=west] at (P4)   {$4$};
\path [draw, very thick] (P1) -- (P3) -- cycle;
\filldraw[red] (45:1cm) circle (2pt); 
\filldraw[red] (225:1cm) circle (2pt); 
\end{tikzpicture}}

\newcommand{\sqdc}
{\begin{tikzpicture}
\path (45:1cm) coordinate (P1);
\path (135:1cm) coordinate (P2);
\path (225:1cm) coordinate (P3);
\path (315:1cm) coordinate (P4);
\draw (P1) -- (P2) -- (P3) -- (P4) -- cycle;
\node [anchor=west] at (P1)   {$1$};
\node [anchor=east] at (P2)   {$2$};
\node [anchor=east] at (P3)   {$3$};
\node [anchor=west] at (P4)   {$4$};
\path [draw, very thick] (P2) -- (P4) -- cycle;
\filldraw[red] (135:1cm) circle (2pt); 
\filldraw[red] (315:1cm) circle (2pt); 
\end{tikzpicture}}

\newcommand{\pentapic}{%
\begin{tikzpicture}
\node (A) {\ocot};
\node (B) [below left =of A] {\tpot};
\node (C) [below  =of B] {\ottp};
\node (D) [below  =of C] {\dptp};
\node (E) [below right =of D] {\tpdp};
\node (F) [right =of E] {\dcdp};
\node (G) [above right =of F] {\dpdc};
\node (H) [above =of G] {\ocdc};
\node (I) [above =of H] {\dcoc};
\node (K) [right =of A] {\otoc};
\path[draw] (A) -- (K) [dashed, color=blue, very thick] node[midway, color=black] {$\t^3,\t^4$} -- cycle ;
\path[draw] (I) -- (H) [dashed, color=blue, very thick] node[midway, color=black] {$\t^1,\t^2$} -- cycle ;
\path[draw] (F) -- (G) [dashed, color=blue, very thick] node[midway, color=black] {$\t^4,\t^5$} -- cycle ;
\path[draw] (D) -- (E) [dashed, color=blue, very thick] node[midway, color=black] {$\t^2,\t^3$} -- cycle ;
\path[draw] (B) -- (C) [dashed, color=blue, very thick] node[midway, color=black] {$\t^1,\t^5$} -- cycle ;
\path [->] let \p1=($(A)-(B)$),\n1={atan2(\y1,\x1)},\n2={180+\n1} in
     ($ (B.\n1)!2pt!90:(A.\n2) $) edge node {$T^{35}$} ($ (A.\n2)!2pt!-90:(B.\n1) $);
    \path [->] let \p1=($(B)-(A)$),\n1={atan2(\y1,\x1)},\n2={180+\n1} in
    ($ (A.\n1)!2pt!90:(B.\n2) $) edge node {$T^{14}$} ($ (B.\n2)!2pt!-90:(A.\n1) $);
\path [->] let \p1=($(C)-(D)$),\n1={atan2(\y1,\x1)},\n2={180+\n1} in
     ($ (D.\n1)!2pt!90:(C.\n2) $) edge node {$T^{25}$} ($ (C.\n2)!2pt!-90:(D.\n1) $);
    \path [->] let \p1=($(D)-(C)$),\n1={atan2(\y1,\x1)},\n2={180+\n1} in
    ($ (C.\n1)!2pt!90:(D.\n2) $) edge node {$T^{13}$} ($ (D.\n2)!2pt!-90:(C.\n1) $);
\path [->] let \p1=($(E)-(F)$),\n1={atan2(\y1,\x1)},\n2={180+\n1} in
     ($ (F.\n1)!2pt!90:(E.\n2) $) edge node {$T^{24}$} ($ (E.\n2)!2pt!-90:(F.\n1) $);
    \path [->] let \p1=($(F)-(E)$),\n1={atan2(\y1,\x1)},\n2={180+\n1} in
    ($ (E.\n1)!2pt!90:(F.\n2) $) edge node {$T^{35}$} ($ (F.\n2)!2pt!-90:(E.\n1) $);
\path [->] let \p1=($(G)-(H)$),\n1={atan2(\y1,\x1)},\n2={180+\n1} in
     ($ (H.\n1)!2pt!90:(G.\n2) $) edge node {$T^{14}$} ($ (G.\n2)!2pt!-90:(H.\n1) $);
    \path [->] let \p1=($(H)-(G)$),\n1={atan2(\y1,\x1)},\n2={180+\n1} in
    ($ (G.\n1)!2pt!90:(H.\n2) $) edge node {$T^{25}$} ($ (H.\n2)!2pt!-90:(G.\n1) $);
\path [->] let \p1=($(I)-(K)$),\n1={atan2(\y1,\x1)},\n2={180+\n1} in
     ($ (K.\n1)!2pt!90:(I.\n2) $) edge node {$T^{13}$} ($ (I.\n2)!2pt!-90:(K.\n1) $);
    \path [->] let \p1=($(K)-(I)$),\n1={atan2(\y1,\x1)},\n2={180+\n1} in
    ($ (I.\n1)!2pt!90:(K.\n2) $) edge node {$T^{24}$} ($ (K.\n2)!2pt!-90:(I.\n1) $);
\end{tikzpicture}}

\title[Super Pl\"{u}cker embedding and cluster algebras]{On   super Pl\"{u}cker embedding and
cluster algebras}
\author{Ekaterina~Shemyakova}
\address{Department of Mathematics,  University of Toledo, Toledo,  Ohio, 43606, USA}
\email{ekaterina.shemyakova@utoledo.edu}
\author{Theodore~Voronov}
\address{Department of Mathematics,  University of Manchester, Manchester, M13 9PL,  UK  and
Faculty of Physics, Tomsk State University, Tomsk, 634050, Russia}
\email{theodore.voronov@manchester.ac.uk}

\thanks{Research of the first author was partially supported by NSF under grant  1708033. Research of the second author was partially supported by LMS grants.}

\begin{document}
\begin{abstract}
We define a super analog of the classical Pl\"{u}cker embedding of the Grassmannian into a projective space. One of the  difficulties of the problem is rooted in the fact that   super exterior powers $\L^{r|s}(V)$  are not a simple  generalization from   the completely even case (this works only for $r|0$ when it is possible to use $\L^r(V)$). To construct the embedding we need to non-trivially combine a super vector space $V$ and its parity-reversion $\Pi V$. Our ``super Pl\"{u}cker map''    takes the  Grassmann supermanifold $G_{r|s}(V)$ to a   ``weighted projective space'' $P\left(\L^{r|s}(V)\oplus \L^{s|r}(\Pi V)\right)$ with weights $+1,-1$.
A simpler map $G_{r|0}(V)\to P(\L^r(V))$ works for the case $s=0$.
We  construct  a super analog of Pl\"{u}cker coordinates, prove that our map is an embedding, and obtain   ``super  Pl\"{u}cker relations''. We analyze another type  of  relations (due to Khudaverdian) and show their equivalence with the super Pl\"{u}cker relations for $r|s=2|0$.  We    discuss     application to  much sought-after super cluster algebras  and  construct  a  super cluster structure  for $G_2(\R{4|1})$ and $G_2(\R{5|1})$.
\end{abstract}

\maketitle
\tableofcontents

\section{Introduction}\label{sec.intro}
The classical Pl\"ucker map
is a map from the Grassmann manifold $G_k(V)$ of $k$-planes in the $n$-space $V$ to the projective space $P(\L^k(V))$ associated with  the space of multivectors of degree $k$ in $V$.
Pl\"ucker map is an embedding and its image in $P(\L^k(V))$ is specified by quadric equations known as the Pl\"ucker relations. This gives a realization of the Grassmann manifold as a projective algebraic variety.
The classical Pl{\"u}cker relations serve as a prototypal example for the definition of cluster algebras --- the notion that has been attracting great attention in recent years.

An analog for super Grassmannians of the Pl{\"u}cker map has been unknown
(except for one example in the recent work by Cervantes--Fioresi--Lled\'{o}~\cite{cervantes:quantum-chiral-2011}). It was   even expected not to exist in general, see Manin~\cite{manin:gaugeeng}. In the present paper we construct such an analog.

There are two cases  to distinguish. One is ``general'', i.e
the case of  $r|s$-planes in an $n|m$-space.
The main difficulty, in this case,  is  correct identification of the target of the map (which turns to be different from a straightforward analog of the classical situation).
For this case, we  found a correct  target space;   we  also   identified variables that serve as   analogs of the classical Pl\"{u}cker coordinates   and  we   proved that the  ``super Pl\"ucker map'' that we constructed is an embedding.
Another case,  which  we call ``algebraic'', is that of $r$-, i.e. $r|0$-planes in an $n|m$-space.
It
can be treated by
algebraic tools that are closer to those used in the classical theory. We investigated both cases   and  in particular obtained   the   ``super  Pl\"ucker'' relations. (For the algebraic case they can be obtained by two different methods and we compare these approaches.)

Super  Pl\"ucker relations  can have   direct   application to the much sought-after  ``super cluster algebras''.
A general notion of super cluster algebras algebras remains yet conjectural; different approaches have been put forward   by
Ovsienko~\cite{ovsienko-supercluster:2015},  Li--Mixco--Ransingh--Srivastava~\cite{srivastava:tosupercluster-2017}, and Ovsienko--Shapiro~\cite{ovsienko-shapiro:2018}. We have established a ``super cluster structure'' for the examples of super Grassmannians   $G_2(\R{4|1})$ and $G_2(\R{5|1})$. We believe it generalizes to arbitrary $G_r(\R{n|m})$.

We made an adaptation of the definition for $r|s$ superforms given in~\cite{tv:compl,tv:pdf,tv:cohom-1988}, \cite{tv:dual,tv:cartan1}, \cite{tv:git}. An $r|s$-vector is a function $F$ of $r$ even and $s$ odd covectors satisfying the two conditions:
\begin{equation}\label{eq.bercond0}
    F(P\cdot g)=F(P)\cdot\Ber g\,,
\end{equation}
for all $g\in \GL(r|s)$, and
\begin{equation}\label{eq.fundeqcond}
    \dder{F}{{p_a{}^i}}{{p_b{}^j}}  + (-1)^{\itt\,\jt +\at(\itt+\jt)} \dder{F}{p_a{}^j}{p_b{}^i} = 0\,.
\end{equation}
Here $P=(p_a{}^i)$ is the matrix of components of covectors. (By $\itt,\jt$, etc. we denote parities.)

The following ``non-linear analog'' of wedge product
is an example of such function $F$:
\begin{equation}\label{eq.nonlinwed1}
    [\up_1,\ldots,\up_r|\up_{r+1},\ldots,\up_{r+s}](p^1,\ldots,p^r|p^{r+1},\ldots,p^{r+s})=\Ber \Bigl(\langle \up_i,p^j\rangle\Bigr)\,,
\end{equation}
a rational function of vectors $\up_i \in V$ (as well as covectors $p^j$).
For $s=0$, it gives\,---\,up to a normalization factor\,---\,exactly the value of $\up_1\wed \ldots\wed \up_k$ on $k$ covector arguments.
Formula~\eqref{eq.nonlinwed1}
appeared in~\cite{hov:gayduk}  in the context of integration on supermanifolds as an example of an  $r|s$-density in the sense of A.~S.~Schwarz,  and linear combinations of such objects were called there   ``densities of type $\Ber$''. This formula was later considered  in the context of Voronov--Zorich superforms by A.~Belopolsky~\cite{bel:strings, bel:pco}, under the name  of  ``Pl\"ucker form'' (pointing at  its  relevance for a super  Pl\"ucker mapping\,---\, but this was no further developed).

One may try to define a ``super Pl\"ucker map'' by assigning
\begin{equation}
    L=\spa(\up_1,\ldots,\up_r|\up_{r+1},\ldots,\up_{r+s}) \mapsto F=[\up_1,\ldots,\up_r|\up_{r+1},\ldots,\up_{r+s}]\in \L^{r|s}(V) \,.
\end{equation}
If the vectors $\up_i$ undergo a non-degenerate linear transformation by an element   $g\in \GL(r|s)$, then $F$ is multiplied by an invertible factor $\Ber g$, so it may seem natural to take as the codomain the 
projectivization $P(\L^{r|s}(V))$, which would be completely analogous to the classical case. However, unlike the classical prototype, such a map would not be invertible.

It turns out that to be able to reconstruct $L$, one needs to consider in parallel also the parity-reversed plane $\Pi L\subset \Pi V$ and the corresponding mapping  $L\mapsto F^*\in \L^{s|r}(\Pi V)$.
In the general supercase, to be able to find $L$ from $F$ and $F^*$, one has to extend   them to different domains, so they really have to go in parallel, which leads to finally identifying the \emph{super Pl\"ucker map} as a   mapping
\begin{equation}
   G_{r|s}(V) \to P_{1,-1}\left(\L^{r|s}(V)\oplus \L^{s|r}(\Pi V)\right)\,,
\end{equation}
where at  the right-hand side $P_{1,-1}$ stands for   equivalence classes $(F,G)\sim (\la F, \la^{-1} G)$, which is a version of a weighted projected space,  with weights $+1$, $-1$.

We show that the so constructed super Pl\"ucker map is indeed an embedding.
Unlike the classical case, it is given by rational formulas. (Unless $s=0$ when we can redefine it back to  be  $G_{r|0}(V)\to P(\L^r(V))$ and the formulas are polynomial.)
The facts that in the general case the ambient space of the embedding is not a usual projective space and that the appearance of rational functions are  no wonder, of course, as it has been known (Manin~\cite{manin:gaugeeng}, Penkov~\cite{penkov:BWB1990}) that the general super Grassmannian is not a  projective algebraic variety.

This brings us to the question about an analog of the Pl\"ucker relations. We are able to obtain the complete result both in the algebraic case of $r|0$-planes and in the general case of $r|s$-planes in $n|m$-space. In the algebraic case $s=0$,  where two different approaches are possible,   we develop  and compare both.

A key step both for proving that the super Pl\"ucker map is  an embedding and for establishing a superanalog of Pl\"ucker relations was identification of the   correct analog of Pl\"ucker variables.
``Super Pl\"ucker  coordinates'' of an $r|s$-plane in an $n|m$-space consist of two sets of variables. The  even Pl\"ucker  coordinates are just various $r|s$-minors (the Berezinians of $r|s$-submatrices) of the    matrix whose rows span the plane.
The odd Pl\"ucker  coordinates are
``wrong'' minors, which are the above usual minors modified by replacing one even column by an odd matrix column. And the same is done for the parity-reversed matrix.
~\footnote{Use of such ``wrong'' minors as well the need to employ in parallel the parity reversed matrix is fundamentally related with the super Cramer  rule, see~\cite{rabin:duke} and \cite{tv:ber}; we also met this for ``super Wronskians''~\cite{tv:superdif}.}  In the algebraic case ($s=0$),
we in principle
obtain a larger set of variables as the components of the corresponding multivector $T\in \L^r(V)$ (there is no need to use $\Pi V$ in this case). However, part of the variables can be excluded by using (part of) algebraically obtained super Pl\"ucker relations and the remaining part (we call the ``essential'') coincide with those given by the general method as described above.

There is an interesting connection between super Pl\"ucker relations that we obtained and
{relations}
of Khudaverdian (around 2010, unpublished), who
noticed that
if $T$ is a simple multivector, then $1/T$ will also be a (simple) multivector of the parity-reversed degree in the space with reversed parity. Applying
the definition of exterior powers by differential equations\cat{~\eqref{eq.fundeqcond}}, one can obtain from here certain quadratic relations for such a multivector. It is indeed tempting to think that they will be exactly the Pl\"ucker relations in the classical case (and thus providing a way for the sought-for super analog). This turns out only partly true. Namely, we show here that for $k$-planes in $n|m$-space,  the Khudaverdian relations are equivalent to the (super) Pl\"ucker relations if $k=2$, but  starting from $k=3$,  they are only a corollary of the super Pl\"ucker relations.

The structure of the paper is as follows. In Section~\ref{sec.extern}, we recall some 
constructions of
supergeometry. In Section~\ref{sec.simpl}, we develop the general setup for the super Pl\"ucker map and study the simplest super case $r|s=k|0$. For this case, we obtain a complete description of the super Pl\"ucker relations by algebraic method.   In Section~\ref{sec.general}, we study the general case: introduce the super Pl\"ucker map for the Grassmannian $G_{r|s}(V)$ and prove that it is an embedding. We develop general method  for obtaining super Pl\"ucker relations.
We also
analyze there ``Khudaverdian's relations''.  In Section~\ref{sec.clust}, we discuss   application  of the super   Pl\"ucker relations  to super cluster algebras (which so far remain only partly known and, to an extent, conjectural). We consider examples where we are able to introduce a ``super cluster structure''.

Results of this paper, at various early stages, were reported at the special session of the AMS meeting in Madison in Fall 2019, at  the conference on supergeometry at Luxembourg (December 2019), and at S.\,P. Novikov's seminar on geometry and mathematical physics in Moscow. We are grateful to all the participants for discussions, and we thank especially M. Gekhtman, H.~Khudaverdian, I. Penkov, A. Schwarz, M. Shapiro, and A. Srivastava. We also thank the referee for detailed comments that helped to improve the text.

\section{Background: super exterior powers and super Grassmannians}\label{sec.extern}
%

\subsection{Basics} \label{subsec.extalgebra}

Let $V=V_0\oplus V_1$ be a superspace
or a free module over a commutative superalgebra.   Let $T(V)$, $S(V)$, $\L(V)$
be its \emph{tensor}, \emph{symmetric}, and \emph{exterior} algebras, respectively.
The superalgebras $S(V)=\bigoplus_{k\geq 0}S^k(V)$ and $\L(V)=\bigoplus_{k\geq 0}\L^k(V)$ are defined by
$S(V)=T(V)/\langle v\otimes w-(-1)^{\vt\wt}w\otimes v\rangle$ and $\L(V)=T(V)/\langle v\otimes w+(-1)^{\vt\wt}w\otimes v\rangle$. (Here $v,w\in V$.) They inherit $\Z$-grading from the tensor algebra. It is independent of $\ZZ$-grading which is induced from $V$.
The tilde over a symbol denotes its parity.  Recall that neither the symmetric nor the exterior powers have a top term.

An \emph{even supermatrix}
has
its rows and columns
labeled with a parity and the elements in even-even and odd-odd positions are even, and those in even-odd and odd-even positions are odd.
\emph{Superdeterminant} or   \emph{Berezinian} is defined
for a  square {even}  {supermatrix} $g$ given in its standard form as
\begin{equation} \label{eq:ber}
    \Ber g = \Ber \begin{pmatrix}
              g_{00} & g_{01} \\
              g_{10} & g_{11}
            \end{pmatrix}  = \frac{\det\left(g_{00}-g_{01}g_{11}^{-1}g_{10}\right)}{\det g_{11}}= \frac{\det g_{00}}{\det\left(g_{11}-g_{10}g_{00}^{-1}g_{01}\right)}\,.
\end{equation}
\emph{Parity reversion} of a (super)matrix $A\mapsto A^{\Pi}$ changes the parities of the rows and columns to the opposite.

The \textit{inverse Berezinian} $\Ber^*g$ of a matrix $g$ is defined (following~\cite{rabin:duke}, \cite{tv:ber}) as the Berezinian of the parity-reversed matrix.

$\Ber g$  can be extended to a multilinear function of the even rows or even columns filled by vectors of arbitrary parity,  \emph{the inverse Berezinian   $\Ber^*g$ can be extended to a multilinear function of  the odd rows or odd columns}. Thus $\Ber g$ and $\Ber^* g$ become defined on different domains and it is convenient to consider them in parallel.

Another necessary extension~\cite{tv:superdif} of the notion is
to case where only one row or column, in an even position for $\Ber g$ or in an odd position for $\Ber^*g$, is filled by a vector of the ``wrong'' parity (odd or even, respectively), so giving an odd quantity as $\Ber g$ or $\Ber^*g$. We shall refer to such an argument as a \emph{ghost row} or \emph{ghost column}. They arise in the row or column expansion formulas~\cite{tv:superdif}.

\subsection{``Voronov--Zorich'' super exterior powers}

We give here more details on already mentioned in Sec.~\ref{sec.intro} ``Voronov--Zorich'' exterior powers, which were originally discovered in search of the ``correct''   superanalog of differential forms~\cite{tv:compl,tv:pdf,tv:cohom-1988}, \cite{tv:git,tv:dual,tv:cartan1}.

\label{subsec.exteriorpowers}
For a superspace  $V$, where  $\dim V=n|m$,  an element of the \emph{exterior power} $\L^{r|s}(V)$ or a \emph{multivector of degree $r|s$}
is a function $F(p)=F(p^1,\ldots,p^{r+s})$ of $r$ even and $s$ odd covectors $p^i\in V^*$ satisfying two conditions: the ``covariance''
\begin{equation}\label{eq.bercond}
    F(p\cdot g)=F(p)\cdot\Ber g\,,
\end{equation}
for all $g\in GL(r|s)$, and the ``fundamental equations''
\begin{equation}\label{eq.fund}
    \dder{F}{{p_a{}^i}}{{p_b{}^j}}  + (-1)^{\itt\,\jt +\at(\itt+\jt)} \dder{F}{p_a{}^j}{p_b{}^i} = 0\,.
\end{equation}
(We assume that $r\geq 0$; in fact, the theory has to be complemented by a construction which allows also negative $r$, which we omit here.)
Here $p$ stands for the
array of arguments $p^1,\ldots,p^{r+s}$. We use  left coordinates of covectors  writing them as columns, so $p$ as the argument of $F(p)$ can be perceived as an (even) $n|m \times r|s$  matrix.

Equation~\eqref{eq.bercond} expresses the behavior of the function under a non-degenerate linear transformation of  the columns $p^i$.  In particular, it follows that $F(p)$ is invariant under     elementary transformations (adding to a column a multiple of another column) and  it is homogeneous of degree $+1$ with respect to each of  $r$ even columns and homogeneous of degree $-1$ in each of  $s$  odd columns. Hence $F(p)$ is defined only for linearly independent odd covectors, so $s$ has to be between $0$ and $m$. We should see $F(p)$ as a meromorphic function of the  arguments.

\begin{example} Suppose $s=0$.  Then
the sign in~\eqref{eq.fund} is always ``plus''. This implies, by setting $i=j$,   vanishing of all second derivatives in the variables $p_a{}^i$ for each fixed $i$; hence $F(p)$ has to be an affine function $A^ap_a{}^i+B^i$ in each $p^i$. Together with the homogeneity, this gives linearity in each $p^i$. Hence $F(p)$ is multilinear,
\begin{equation*}
    F(p)=F^{a_1\ldots a_r}p_{a_1}^1\ldots p_{a_r}^r\,.
\end{equation*}
One can check that the remaining part of equations~\eqref{eq.fund} and~\eqref{eq.bercond} is equivalent to the antisymmetry in the arguments $p^1,\ldots,p^r$. Hence  we have an identification
\begin{equation*}
    \L^{r|0}(V)=\L^r(V)\,,
\end{equation*}
where is the right-hand side we use the interpretation of the elements of   $\L^r(V)$ as   antisymmetric multilinear functions.
\end{example}

Note that in the purely even case $m=0$  implies $s=0$, so we   only have the exterior powers $\L^k(V)=\L^{k|0}(V)$.

\begin{example} Fundamental equations~\eqref{eq.fund} contain in particular the following system. If we set the parities of all indices to $1$, i.e. consider the odd-odd block of the matrix $(p_a^i)$ (note that the matrix entries there are even), then~\eqref{eq.fund} takes the form
\begin{equation*}
    \dder{F}{{p_a{}^i}}{{p_b{}^j}}  - \dder{F}{p_a{}^j}{p_b{}^i}=0\,.
\end{equation*}
This is  the  generalized  \emph{F.~John's equation} which is ubiquitous  in integral geometry~\cite{gelfand-gindikin-graev:intgeometry1980} and is key in the theory of general hypergeometric functions~\cite{gelfand:genhypergeom-1}.
\end{example}

The last example shows that there is a deep   connection between the theory of super exterior powers and integral geometry in the sense of Gelfand--Gindikin--Graev. Moreover, the defining system~\eqref{eq.fund} for $r|s$-multivectors can be seen as a ``super generalization'' of Gelfand's general hypergeometric system. This relation was   used in~\cite{tv:cohom-1988}, but  there is definitely  more to explore.

\begin{example}[``non-linear wedge product''] \label{ex.nonlinwedge}
Let $\up_1,\ldots,\up_r,\up_{r+1},\ldots,\up_{r+s}$ be a sequence of $r$ even and $s$ odd linearly independent vectors.
Define a function $[\up_1,\ldots,\up_r|\up_{r+1},\ldots,\up_{r+s}]$ as follows:
\begin{equation}\label{eq.nonlinwed}
    [\up_1,\ldots,\up_r|\up_{r+1},\ldots,\up_{r+s}](p):=\Ber (\up_i{}^ap_a{}^j)\,.
\end{equation}
These functions  satisfy~\eqref{eq.bercond} and~\eqref{eq.fund}, and so are elements of $\L^{r|s}(V)$. In the purely even case, when $s=0$, \eqref{eq.nonlinwed} would coincide with $\up_1\wed \ldots \wed \up_r$ up to a factor. Hence we have a \emph{non-linear analog of   wedge product} or   a \emph{simple multivector} of degree $r|s$.
\end{example}

In the super case, it is   not known if simple multivectors such as   $[\up_1,\ldots,\up_r|\up_{r+1},\ldots,\up_{r+s}]$ span the whole $\L^{r|s}(V)$. This space can be infinite-dimensional and its explicit description is not yet at hand.


\subsection{Recollection of super Grassmannians (and some related supergeometry)}\label{subsec.recol}

Suppose $V$ is a vector
space of dimension $n|m$. Fix a
dimension $r|s$, where $0\leq r\leq n$, $0\leq s\leq m$. Then the \emph{Grassmann supermanifold} (or \emph{super Grassmann manifold} or \emph{super Grassmannian}) $G_{r|s}(V)$ of $r|s$-planes in $V$ can be defined as the ``universal base'' for families of $r|s$-dimensional linear subspaces in $V$. 
See, for example,   Manin~\cite{manin:gaugeeng}, A.~Voronov--Manin--Penkov~\cite{voronov-manin-penkov:1988}, and Penkov~\cite{penkov:BWB1990} for formal definitions.

\begin{example}
Let $V=V_0\oplus V_1$ be a
superspace over
$\RR$ with
a basis consisting of
even $\ep_1,\ldots,\ep_n$ and odd $\ep_{\hat 1},\ldots,\ep_{\hat m}$, that span $V_0$ and $V_1$ respectively.
To allow the coordinates be taken from some auxiliary  commutative superalgebra $A$,
consider
 the tensor product $A\otimes V$ obtaining an $n|m$-dimensional free module over $A$.
If we take the coordinates $x^a$, $x^{\mu}$ as indeterminates (even and odd, respectively) and choose the superalgebra $A$ as the algebra of functions of these indeterminates (polynomial or $\fun$), then the vector $\xp\in (A\otimes V)_0$, $\xp=x^1\ep_1+\ldots+x^n\ep_n+x^{\hat 1}\ep_{\hat 1}+\ldots+x^{\hat m}\ep_{\hat m}$, is a \emph{universal even vector} in $V$ (i.e., giving the universal family in the sense that every family of even vectors with some other algebra $A'$ is obtained an algebra homomorphism $A\to A'$)\footnote{Some caution is required for $\fun$ functions as opposed to polynomials; but we do not discuss it here.}. The algebra  of functions of the variables $x^a$, $x^{\mu}$ is by definition the \emph{algebra of functions  on $V$ regarded as a supermanifold}.
\end{example}

\begin{remark}
We are writing coordinates of vectors to the left of basis elements, i.e. our free modules are left modules. 
Everything can be also easily reformulated for right modules and right coordinates.
\end{remark}

We denote the supermanifold obtained as explained from a $\ZZ$-graded vector space $V$  by the same symbol $V$. It  carries a    \emph{linear structure}  as a collection of privileged coordinate systems arising from homogeneous bases in $V$.

The term ``superspace''  from now will conveniently embrace  both interpretations of $V$.

Speaking of $V$ as a supermanifold, we   normally suppress the adjective ``universal'' for vectors and with an abuse of notation write   $\xp\in V$ or $\up\in V$, etc.

\begin{example}[Projective superspace] A  non-zero even vector in $V$
spans an $1|0$-dimensional subspace (a line)  in $V$. A  universal even vector $\up\in V$ spans  a  ``universal line''
Thus we arrive at the notion of the \emph{projective superspace} $P(V)$ associated with a superspace $V$ (or the \emph{projectivization} of $V$). $P(V)$ is a supermanifold of dimension $n-1|m$ if $\dim V=n|m$. It comes with a canonical atlas associated with a given basis in $V$. Linear coordinates on $V$,  $u^1,\ldots,u^n,u^{\hat 1}, \ldots, u^{\hat m}$ (the first $n$ even and the last $m$ odd),
where one of $u^a$, $a=1,\ldots,n$ is invertible, and   considered up to proportionality, are by definition \emph{homogeneous coordinates} on $P(V)$:
\begin{equation*}
    (u^1:\ldots:u^n:u^{\hat 1}: \ldots: u^{\hat m}) \, .
\end{equation*}
This gives $n$ canonical \emph{affine charts} on $P(V)$   obtained by dividing by one of the even homogeneous coordinates so to make it $1$. Changes of coordinates between charts have the same form as in the ordinary   case. For $V=\R{n|m}$, we write $P(V)=\RR P^{n-1|m}$. Similarly for  the complex version $\CC P^{n-1|m}=P(\CC^{n|m})$.
\end{example}

\begin{example}[$\RR P^{0|m}$ and $\RR P^{1|m}$]
For $\RR P^{0|m}$, there are homogeneous coordinates
$$(u:\x^1:\ldots:\x^m)=(1:\h^1:\ldots:\h^m)$$ (here
we are using Latin letters for even variables and Greek letters for odd variables),
$\h^i=\x^i/u$. Hence $\RR P^{0|m}=\R{0|m}$.
The underlying topological space is a point.

Consider now $\RR P^{1|m}$. We have homogeneous coordinates $$(u^1:u^2: \x^1:\ldots:\x^m)\,,$$
where either $u^1$ or $u^2$ is invertible. Hence there are two affine charts with $$(1:w^2_{(1)}:\h^1_{(1)}:\ldots:\h^m_{(1)})$$ and $$(w^1_{(2)}:1:\h^1_{(2)}:\ldots:\h^m_{(2)})\,,$$
and the changes of coordinates are: $w^2_{(1)}=1/w^1_{(2)}$ and $\h^i_{(1)}=\h^i_{(2)}/w^1_{(2)}$, where $i=1\ldots,m$. The underlying topological space is $\RR P^1$ and the supermanifold   $\RR P^{1|m}$ can be identified with the vector bundle $\Pi(\underbrace{E\oplus\ldots \oplus E}_{\text{$m$ copies}})$, where $E\to \RR P^1$ is the tautological line bundle.
\end{example}

One can similarly work out the structure of $\RR P^{n|m}$ and $\CC P^{n|m}$ in general.

The projective superspace $P(V)$ is by definition the super Grassmannian $G_{1|0}(V)$. The passage to general super Grassmannians   presents  no difficulty. Before  the definition,
we give two typical examples.

\begin{example}[$G_{2|0}(V)$] \label{ex.g20}
A $2$-plane in a superspace $V$ of dimension $n|m$ (note that $2=2|0$) is spanned by two even linearly independent vectors $\up_1$ and $\up_2$ in $V$:
\begin{equation*}
    \begin{pmatrix}
      \up_1 \\
      \up_2 \\
    \end{pmatrix}=
    \begin{pmatrix}
       x^1_1 & \ldots & x^n_1 & \vline & \xi^1_1 & \ldots & \xi^m_1 \\
       x^1_2 & \ldots & x^n_2 & \vline & \xi^1_2 & \ldots & \xi^m_2
     \end{pmatrix}\,.
\end{equation*}
Here
we write $x^a_i$  and $\x^{\mu}_i$  instead of
$u^a_i$  and $u^{\hat \mu}_i$ respectively, where $a=1,\ldots, n$, $\mu=1,\ldots,m$.
The variables denoted by Latin letters are even and those denoted by Greek letters are odd. The above matrix is analogous to the array of homogeneous coordinates for a projective superspace. We will denote it $U$ and refer to as \emph{homogeneous coordinates} on the {super Grassmannian} $G_{2|0}(V)=G_2(V)$ (yet to be   defined).
A choice of an invertible   $2\times 2$-submatrix corresponds by definition to an \emph{affine chart}; there are $\binom{n}{2}=n(n-1)/2$   such charts corresponding to a choice of two even columns, and for each of them we consider instead of $U$ a ``normalized'' matrix e.g.
\begin{equation*}
    \tensor[^{(n-1,n)}]{U}{}
    =
    \begin{pmatrix}
       x^1_1 & \ldots &  x^{n-2}_1 & 1 & 0 & \vline & \xi^1_1 & \ldots & \xi^m_1 \\
       x^1_2 & \ldots & x^{n-2}_2 & 0 & 1 & \vline & \xi^1_2 & \ldots & \xi^m_2
     \end{pmatrix}\,,
\end{equation*}
whose entries are by definition the \emph{inhomogeneous coordinates} in this chart. The pair $(n-1,n)$ as a left superscript indicates the chart. In general, the charts are numbered by pairs $(a,b)$ where $1\leq a<b\leq n$. In the detailed notation, it has to be attached to all matrix entries, e.g. $\tensor*[^{(n-1,n)}]{x}{^1_1}$ and  $\tensor*[^{(n-1,n)}]{\xi}{^1_1}$. Above they are suppressed for simplicity. The passage from one affine chart to another is by dividing by a $2\times 2$-submatrix from the left so to get the identity matrix in the required places, e.g.
\begin{equation*}
    \tensor[^{(12)}]{U}{}
    = \begin{pmatrix}
       \tensor*[^{(n-1,n)}]{x}{^1_1} &  \tensor*[^{(n-1,n)}]{x}{^2_1}  \\
       \tensor*[^{(n-1,n)}]{x}{^1_2} &  \tensor*[^{(n-1,n)}]{x}{^2_2}
     \end{pmatrix}^{-1} \tensor[^{(n-1,n)}]{U}{}\,.
\end{equation*}
The collection of inhomogeneous coordinates in each affine chart together with the described transformations of coordinates between the charts  can be seen as  an ``abstract atlas'' (see Berezin~\cite{berezin:antieng}, Manin~\cite{manin:gaugeeng}) specifying a supermanifold, which   by definition is the \emph{super Grassmannian} $G_{2|0}(V)=G_2(V)$. We have
\begin{equation*}
    \dim G_2(V) = 2(n-2)|2m
\end{equation*}
(the number of inhomogeneous coordinates in each chart). The underlying topological space is the ordinary Grassmannian $G_2(\R{n})$ of $2$-planes in $n$-space.
\end{example}

\begin{example}[$G_{1|1}(V)$] \label{ex.g11}
In
the case of $1|1$-planes in an $n|m$-dimensional  superspace $V$,
a plane is spanned by one even and one odd linearly independent vectors in $V$, e.g. $\up_1$ and $\up_{\hat 1}$.
(Here linear independence amounts to each $\up_1$ and $\up_{\hat 1}$ being non-zero, i.e. having an invertible even coordinate.) We can write their coordinates as a matrix
\begin{equation*}
    U=
    \begin{pmatrix}
       x^1  & \ldots & x^n  & \vline & \xi^1  & \ldots & \xi^m  \\
       \hline \vphantom{\int_a^b}
       \h^1  & \ldots & \h^n  & \vline & y^1  & \ldots & y^m
     \end{pmatrix}\,.
\end{equation*}
with one even    and one odd row and $n$ even and $m$ odd columns.  Here we write $x^a$  and $\x^{\mu}$  instead of
$u^a_1$  and $u^{\hat \mu}_1$, 
where $a=1,\ldots, n$, $\mu=1,\ldots,m$, for the even vector $\up_1$. Similarly we write $\h^a$ and $y^{\mu}$ instead of
$u^a_{\hat 1}$  and $u^{\hat \mu}_{\hat 1}$ for the odd vector $\up_{\hat 1}$.
The matrix $U$ is the matrix of \emph{homogeneous coordinates} on the  super Grassmannian  $G_{1|1}(V)$. A choice of an invertible $x^a$ and a choice of an  invertible $y^{\mu}$, i.e. a choice of one even and one odd column  giving an invertible $1|1\times 1|1$-submatrix in $U$,  makes it possible to replace $U$ by a normalized matrix, e.g.
\begin{equation*}
    \tensor[^{(n|m)}]{U}{}
    =
    \begin{pmatrix}
       x^1  & \ldots & x^{n-1} & 1 & \vline & \xi^1  & \ldots & \xi^{m-1} & 0  \\
       \hline \vphantom{\int_a^b}
       \h^1  & \ldots & \h^{n-1} & 0  & \vline & y^1  & \ldots & y^{m-1} &  1
     \end{pmatrix}
\end{equation*}
(where we suppressed the left superscript  $(n|m)$ for all matrix entries),
which is by definition the matrix of \emph{inhomogeneous coordinates} in the corresponding affine chart. Here the affine chart has the ``number'' $(n|m)$ and in general  the charts are numbered by pairs $(a|\mu)$, where $a=1,\ldots,n$, $\mu=1,\ldots,m$. Changes of coordinates are similar to the previous example, where now one needs to divide from the left by an invertible $1|1\times 1|1$-submatrix; e.g. the transformation between the $(n|m)$th chart and the $(1|1)$th chart is given by
\begin{equation*}
    \tensor[^{(1|1)}]{U}{}
    = \begin{pmatrix}
       \tensor*[^{(n|m)}]{x}{^1} & \vline &  \tensor*[^{(n|m)}]{\xi}{^1}  \\
        \hline  \vphantom{\int_a^b}
       \tensor*[^{(n|m)}]{\h}{^1} & \vline &  \tensor*[^{(n|m)}]{y}{^1}
     \end{pmatrix}^{-1} \tensor[^{(n|m)}]{U}{}\,.
\end{equation*}
The collection of inhomogeneous coordinates in each affine chart together with the given transformations of coordinates define  an  abstract atlas. The supermanifold it specifies is  by definition the \emph{super Grassmannian} $G_{1|1}(V)$. We have
\begin{equation*}
    \dim G_{1|1}(V) = n+m-2| n+m-2
\end{equation*}
(the number of inhomogeneous coordinates in any affine chart) and the underlying topological space is the product $\RR P^{n-1}\times \RR P^{m-1}$.
\end{example}

Note that in the above examples, in order to introduce an atlas for the projective superspace $P(V)$ or the super Grassmannians  $G_{2|0}(V)$ and $G_{1|1}(V)$ we assumed some given basis in $V$ (to make it possible to consider coordinates of vectors). ``Canonical atlases'' that we introduced are functions of such a basis. We continue with this assumption.

Now it is clear how the general super Grassmannian is defined. Fix $r$ and $s$ such that $0\leq r\leq n$ and $0\leq s\leq m$. An $r|s$-dimensional subspace in $V$ (short: $r|s$-plane) is specified by an array of $r$ even and $s$ odd vectors in $V$ that span it as a basis. Call them $\up_1 \ldots, \up_r, \up_{\hat 1}, \ldots, \up_{\hat s}$.  Such a basis is defined up to a non-degenerate linear transformation. As before, ``vectors'' and ``transformations'' are understood as universal, i.e. we start from the supermanifold
\begin{equation*}
    \underbrace{V\times \ldots \times V}_{\text{$r$ times}}\times
    \underbrace{\Pi V \times \ldots \times \Pi V}_{\text{$s$ times}}\,,
\end{equation*}
in which the condition of linear independence defines an open domain; denote it $\St_{r|s}(V)$\,\footnote{It is the analog of what topologists call ``non-compact Stiefel manifold'', i.e. the space of all $k$-frames in $n$-space as opposed to the space of orthonormal 
frames  which is the standard  (compact) Stiefel manifold.}. The supergroup $\GL(r|s)$ acts from the left on $\St_{r|s}(V)$ by linear transformations. The \emph{super Grassmann manifold} $G_{r|s}(V)$ by definition is the quotient
\begin{equation}
    G_{r|s}(V):= \St_{r|s}(V)/\GL(r|s)\,.
\end{equation}
In particular, we immediately see its dimension: 
\begin{equation}\label{eq.dimgrassrs}
    \dim  G_{r|s}(V) =
    (r|s)(n-r|m-s)= r(n-r)+ s(m-s)\,|\, r(m-s)+s(n-r)\,,
\end{equation}
i.e. $\dim \St_{r|s}(V)-\dim \GL(r|s)$, where  $\dim \St_{r|s}(V)= (r|s)(n|m)$ and $\dim \GL(r|s)=(r|s)^2$.  (Recall that $a|b=a+b\Pi$ with the multiplication rule $\Pi^2=1$.) The formula for the dimension  has the same form as for ordinary Grassmannians.

We can construct a standard atlas for the super Grassmannian $G_{r|s}(V)$ in the same way as we did in the examples.  Take the matrix $U$ of size $(r|s)\times (n|m)$ whose rows are the (left) coordinates $u_i{}^a$ of the vectors $\up_{i}$. Here $i\in \{1, \ldots, r\}\cup \{\hat 1, \ldots, \hat s\}$ and $a\in \{1, \ldots, n\}\cup \{\hat 1, \ldots, \hat m\}$. Hence
\begin{equation}\label{eq.U}
    U=\begin{pmatrix}
                    u_1{}^{\!1}    & \ldots & u_1{}^{\!n} & \vline & u_1{}^{\!\hat 1}  & \ldots & u_1{}^{\!\hat m} \\
                    \ldots    & \ldots & \ldots & \vline & \ldots  & \ldots & \ldots \\
                    u_r{}^{\!1}    & \ldots & u_r{}^{\!n} & \vline & u_r{}^{\!\hat 1}  & \ldots & u_r{}^{\!\hat m} \\
                   \hline  \vphantom{\int_a^b}
                   u_{\hat 1}{}^{\!1}  & \ldots & u_{\hat 1}{}^{\!n} & \vline & u_{\hat 1}{}^{\!\hat 1}  & \ldots & u_{\hat 1}{}^{\!\hat m}\\
                    \ldots    & \ldots & \ldots & \vline & \ldots  & \ldots & \ldots \\
                   u_{\hat s}{}^{\!1}  & \ldots & u_{\hat s}{}^{\!n} & \vline & u_{\hat s}{}^{\!\hat 1}  & \ldots & u_{\hat s}{}^{\!\hat m}\\
                  \end{pmatrix}\,.
\end{equation}
Matrix elements   $u_i{}^a$  are even for $i=1,\ldots,r$ and $a=1\ldots,n$ or $i=\hat 1,\ldots,\hat s$ and $a=\hat 1\ldots,\hat m$ (diagonal blocks) and odd for $i=1,\ldots,r$ and $a=\hat 1\ldots,\hat m$ or $i=\hat 1,\ldots,\hat s$ and $a=1\ldots,n$  (antidiagonal blocks). The matrix $U$ is the matrix of  \emph{homogeneous coordinates} for the super Grassmannian $G_{r|s}(V)$. As before, it can be convenient to use different letters for variables in different blocks and write also
\begin{equation}\label{eq.Ux}
    U=\begin{pmatrix}
                    x_1{}^{\!1}    & \ldots & x_1{}^{\!n} & \vline & \x_1{}^{\!1}  & \ldots & \x_1{}^{\!m} \\
                    \ldots    & \ldots & \ldots & \vline & \ldots  & \ldots & \ldots \\
                    x_r{}^{\!1}    & \ldots & x_r{}^{\!n} & \vline & \x_r{}^{\!1}  & \ldots & \x_r{}^{\!m} \\
                   \hline  \vphantom{\int_a^b}
                   \h_{1}{}^{\!1}  & \ldots & \h_{1}{}^{\!n} & \vline & y_{1}{}^{\!1}  & \ldots & y_{1}{}^{\!m}\\
                    \ldots    & \ldots & \ldots & \vline & \ldots  & \ldots & \ldots \\
                   \h_{s}{}^{\!1}  & \ldots & \h_{s}{}^{\!n} & \vline & y_{s}{}^{\!1}  & \ldots & y_{s}{}^{\!m}
                  \end{pmatrix}\,.
\end{equation}
By construction, the matrix $U$ has rank $r|s$, i.e. contains an invertible $(r|s)\times (r|s)$-submatrix. To obtain affine charts as before, we assume the invertibility of a particular such a submatrix. They are numbered by a choice of $r$ even and $s$ odd columns. So we can use a symbol $\ach= 
(a_1<\ldots <a_r\,|\,a_{\hat 1}<\ldots < a_{\hat s})$ 
for numbering charts. In the $\ach$th \emph{affine chart}, the matrix $U$ is replaced by the matrix of \emph{inhomogeneous coordinates} $\tensor[^{\ach}]{U}{}$, which   is effectively an $(r|s)\times (n-r|m-s)$-matrix augmented by the identity matrix of size $r|s$ spread over the positions indicated by $\ach$. (So the number of independent variables $(r|s)(n-r|m-s)$ coincides with the dimension calculated above~\eqref{eq.dimgrassrs}.) It is obtained from the matrix of homogeneous coordinates $U$ by
\begin{equation}\label{eq.inhomcoord}
    \tensor[^{\ach}]{U}{}= (U^{\aund})^{-1} U\,,
\end{equation}
where the notation with a  superscript such as $U^{\aund}$ denotes  the submatrix consisting of the indicated columns (with numbers $\aund$). We also refer to inhomogeneous coordinates as \emph{affine coordinates}. Changes of coordinates between affine charts are given by
\begin{equation}\label{eq.affcoordchange}
    \tensor[^{\ach}]{U}{}=  ((\tensor[^{\bch}]{U}{})^{\aund})^{-1}\, \tensor[^{\bch}]{U}{}\,.
\end{equation}
It is all the same, apart for a more abstract notation, as in Examples~\ref{ex.g20} and~\ref{ex.g11}.

The underlying topological space of the super Grassmannian $G_{r|s}(\R{n|m})$ is the product of the ordinary Grassmannians $G_r(\R{n})\times G_s(\R{m})$.
Note also the  symmetry under  parity reversion:
\begin{equation}\label{eq.piongrass}
    G_{r|s}(V)=G_{s|r}(\Pi V)\,,
\end{equation}
where the identification is given by $L\mapsto \Pi L\subset \Pi V$ for $L\subset V$. In terms of matrices $U$ or  $\tensor[^{\ach}]{U}{}$, it the usual  parity reversion on (super)matrices, e.g. $U\mapsto U^{\Pi}$.

In the following we  shall  habitually  suppress the prefix ``super'' speaking about the Grassmann supermanifolds $G_{r|s}(V)$. (Unless there is a risk of  confusion or we wish to make a stress.)

\subsection{Classical Pl\"{u}cker embedding. Problem in the supercase}

Recall first  the classical Pl\"{u}cker embedding.   It arises from considering a non-zero   $k$-vector $\up_1\wed \ldots \wed \up_k$ corresponding to a $k$-plane in $n$-space $V$, which is defined up to proportionality. The components of this multivector are called \emph{Pl\"{u}cker coordinates} of a plane. Algebraically, Pl\"{u}cker coordinates are minors of the matrix of homogeneous coordinates on the Grassmannian. They  satisfy a system of homogeneous quadric equations expressing the condition that a given non-zero $k$-vector $T\in \L^k(V)$ is \emph{simple}, i.e. a wedge product of $k$ vectors. If $T=T^{a_1\ldots a_k}\,e_{a_1}\wed\ldots\wed e_{a_k}$, these equations read:
\begin{equation}\label{eq.class-plueck}
    T^{a_1\ldots a_{k-1}b}T^{c_1\ldots c_k}
    =\sum_{j=1}^{k}  T^{a_1\ldots a_{k-1}c_j}T^{c_1\ldots c_{j-1}bc_{j+1}\ldots c_k}\,,
\end{equation}
for all $a_1,\ldots,a_k,b,c_1,\ldots,c_k$.
(We give them in the form convenient for comparison with the relations that will be obtained in the super case.)
For example, when $k=2$, for $T=T^{ab}e_a\wed e_b$, one has
\begin{equation}\label{eq.class-pluecktwo}
    T^{ac}T^{bd}
    =  T^{ab}T^{cd} +  T^{ad}T^{bc}\,.
\end{equation}
It is proved that this system~\eqref{eq.class-plueck} specifies the Grassmannian $G_k(V)$ as an algebraic variety in the   projectivization of the vector space of all $k$-vectors.   These equations are known as \emph{Pl\"{u}cker relations}. See e.g.~\cite{hodge-pedoe:II}, \cite{postnikov:geom2}, \cite{shaf:linearalg}. (Also~\cite{miller-sturmfels:2005} and also~\cite{kleiman-laksov:1972}, \cite{barnabei-brini-rota:1985}, and \cite{anick-rota:1991}.)

In the supercase the situation becomes substantially more  complicated. The naive wedge product of   basis vectors of a subspace $L\subset V$, even and odd indiscriminately, giving  an element of $\L^{r+s}(V)$, cannot work. (Since the result unlike the classical case will not be in the top exterior power of $L$, which simply does not exist.)
Also, if one starts from matrices, a  naive idea that ``one just needs to replace ordinary minors by Berezinians'' cannot work that simple already because   Berezinian by definition returns an even value and   odd variables are necessary for super Grassmannians, so a generalization of Berezinians will be needed. (That was briefly mentioned   in Section~\ref{sec.extern}. ``Super'' minors and   cofactors are   non-obvious, see~\cite{tv:superdif}.) Another difficulty is that  the ``correct'' space of multivectors $\L^{r|s}(V)$ for general $s$ does not have an explicit description.   
All these problems concern the  general case of non-zero $r$ and $s$ where $r|s$ is the dimension of a subspace $L$ (an $r|s$-plane) in a superspace $V$. We shall turn to it in Section~\ref{sec.general}.

If $s=0$, we have a ``purely even'' $r$-plane in an $n|m$-space such as in Example~\ref{ex.g20}. We will show that this case allows treating in a way essentially analogous to the ordinary case.  We will refer to the case $s=0$ as ``algebraic''. (The case $r=0$ becomes the same after applying parity reversion.)  We will study it in the next section.

\section{Super Pl\"{u}cker embedding.   ``Algebraic'' case} \label{sec.simpl}

\subsection{Algebraic case: $k$-planes in $n|m$-space}\label{subsec.algcase}
Consider planes of dimension $k=k|0$ in an $n|m$-dimensional superspace $V$. Here $0\leq k\leq n$. As we will see, this case  gives interesting answers  in spite of being the simplest from the supergeometry viewpoint.

Every $k$-plane is spanned by $k$ independent even vectors $\up_1,\ldots,\up_k$ that make its basis. (We always assume the possibility of odd parameters without mentioning this each time.) Components of these vectors can be written as an even $k\times n|m$ matrix with independent rows.
It is the  matrix of  homogeneous coordinates  of the plane $L$ as a point of the Grassmann supermanifold $G_k(V)$.
Consider the exterior product
\begin{equation}\label{eq.u1weduk}
    T:=\up_1\wed \ldots \wed \up_k\,.
\end{equation}
It is an even element of the vector space  $\L^k(V)$. The multivector $T$   is \emph{non-degenerate} in the following sense: in the expansion over $e_{a_1}\wed\ldots \wed e_{a_k}$ at least one of its components where all the indices are even is invertible. (Indeed, if we reduce modulo odd elements, this will be the ordinary condition of linear independence for the vectors $\up_1,\ldots,\up_k$.) It is a stronger condition on an even multivector $T$ than  ``being non-zero'' (or more precisely, having at least one  invertible component). (Example: $e_1\wed e_2+e_2\wed e_3$ is non-degenerate, while $e_{\hat 1}\wed e_{\hat 1}$, not.) Hence, in particular, a non-degenerate multivector spans a line in the space of multivectors, so there is a well-defined map
\begin{equation}\label{eq.simpleplueck}
    G_k(V) \to P(\L^k(V))\,
\end{equation}
given by $L\mapsto \up_1\wed \ldots \wed \up_k$, which is
directly analogous to the classical Pl\"{u}cker map. Due to the non-degeneracy, the image of~\eqref{eq.simpleplueck} is contained in some ``affine part'' of the projective superspace $P(\L^k(V))$, i.e. covered by only part of the affine charts.

The following facts   provide a superanalog of the classical Pl\"ucker theory for our case.

Let $T\in \L^k(V)$ be an even multivector. It can be considered as an alternating multilinear function of $k$ covectors. 
If
\begin{equation}\label{eq.Twedge}
    T=T^{a_1\ldots a_k}e_{a_1}\wed \ldots \wed e_{a_k}
\end{equation}
with the (super)antisymmetry condition $T^{a_1\ldots a_ia_{i+1} \ldots a_k}=-(-1)^{\at_i\at_{i+1}}T^{a_1\ldots a_{i+1}a_i \ldots a_k}$, then
\begin{equation}\label{eq.Ttens}
    T=T^{a_1\ldots a_k}e_{a_1}\otimes \ldots \otimes e_{a_k}\,,
\end{equation}
exactly as in the ordinary case, and for arbitrary covectors $p^1, \ldots,p^k$
\begin{equation*}
    T(p^1, \ldots,p^k)= T^{a_1\ldots a_k} p^1_{a_1}\ldots p^k_{a_k} (-1)^{\pt^1(\at_2+\ldots+\at_k)+\ldots+\pt^{k-1}\at_k}\,.
\end{equation*}
In particular, if all arguments $p^i\in V^*$ are \emph{even}, then
\begin{equation}
    T(p^1, \ldots,p^k)= T^{a_1\ldots a_k} p^1_{a_1}\ldots p^k_{a_k}
\end{equation}
and for   covectors of the dual basis,
\begin{equation}
    T(e^{a_1},\ldots, e^{a_k})= T^{a_1\ldots a_k}
    (-1)^{\at^{1}(\at_2+\ldots+\at_k)+\ldots+\at^{k-1}\at_k}\,.
\end{equation}
(Here $e_a(e^b)=\d_a^b$.)
If all arguments but one in $ T(p^1, \ldots,p^k)$ are fixed, this gives a linear function on covectors, i.e. a vector. Call the span of the so obtained vectors, the \emph{associated space} of an even multivector $T$, notation:
\begin{equation}
    L_T:=\spa\left(\,\vphantom{T^{a_1,\ldots a_{k-1}b}}T(p^1,\ldots,p^{k-1},-)\,|\, p^i\in V^*\,\right)  \subset V
\end{equation}
(order inessential because of the antisymmetry). Here the covectors $p^i$ may be of arbitrary parity. A priori nothing can be said about the dimension of the subspace $L_T$ and from nowhere it follows that $L_T$  should be purely even. (In spite of $T$ itself being even.)
\begin{lemma}\label{lm.LT}
\begin{equation}
    \label{eq.LT}
    L_T=\spa\left(T^{a_1,\ldots a_{k-1}b}e_b(-1)^{\bt(\at_1+\ldots+\at_{k-1})}\ |\  \forall  a_1,\ldots, a_{k-1} \right)\,.
\end{equation}
\end{lemma}
\begin{remark}
The sign is~\eqref{eq.LT} is not a common sign, but  involved in the summation and thus cannot be dropped. It is important because it contributes to the relations, see below.
\end{remark}
\begin{proof}[Proof of Lemma~\ref{lm.LT}] To obtain vectors spanning $L_T$, it is sufficient to insert basis covectors as the first $k-1$ arguments into $T$ (leaving the last one free). We have
\begin{multline*}
    T(e^{b_1},\ldots,e^{b_{k-1}},p)=T^{a_1\ldots a_k}(e_{a_1}\otimes \ldots \otimes e_{a_k}) (e^{b_1},\ldots,e^{b_{k-1}},p)= \\
    T^{a_1\ldots a_k}e_{a_1}(e^{b_1})\ldots e_{a_{k-1}}(e^{b_{k-1}})\, e_{a_k}(p)
    (-1)^{\bt_1(\at_2+\ldots+\at_k)+\ldots+ \bt_{k-1}\at_k}=\\
    T^{b_1\ldots b_{k-1}a_k} e_{a_k}(p) \, (-1)^{\e(b_1,\ldots,b_{k-1})} (-1)^{(\bt_1+\ldots+\bt_{k-1})\at_k}\,,
\end{multline*}
where $\e(b_1,\ldots,b_{k-1})= \bt_1(\bt_2+\ldots+\bt_{k-1})+\bt_2(\bt_3+\ldots+\bt_{k-1})+\ldots+ \bt_{k-2}\bt_{k-1}$ is some common sign which     makes no difference for the span (so can be dropped). The second sign participates in the summation, hence cannot be dropped.
\end{proof}

\begin{example}\label{ex.2in4.1}
We shall find the space $L_T$ in the concrete case of a $2$-vector  in a $4|1$-space. (The dimensions $2$ and $4$ give the simplest case in the classical situation where the Pl\"{u}cker relations arise; in this example and later, we will be probing the closest to it supercase.) A general even $2$-vector $T$ has the form
\begin{multline}\label{eq.T}
    T= 2T^{12}e_1\wed e_2 + 2T^{13}e_1\wed e_3 + 2T^{14}e_1\wed e_4 +
    2T^{23}e_2\wed e_3 + 2T^{24}e_2\wed e_4 + 2T^{34}e_3\wed e_4  \\
   +  2\t^1e_1\wed \e + 2\t^2e_2\wed \e + 2\t^3e_3\wed \e + 2\t^4e_4\wed \e +
    s\,\e\wed \e\,.
\end{multline}
Here the components $T^{ab}$, $a,b=1,\ldots, 4$, and $s$ are even and the components $\t^{a}$, $a=1,\ldots, 4$, are odd. We use $\t^{a}$ for $T^{a\hat 1}$ and $s$ for $T^{\hat 1\hat 1}$. Expansion~\eqref{eq.T} corresponds to the formula $T=T^{ab}e_a\wed e_b$ where the summation is extended to all values $a,b\in \{1,\ldots,4,\hat 1\}$ and $T^{ab}=-(-1)^{\at\bt}T^{ba}$. To treat $T$ as a function on covectors, replace $2e_a\wed e_b=e_a\otimes e_b-e_b\otimes e_a$ and $2e_a\wed \e=e_a\otimes \e-\e \otimes e_a$, $a=1,\ldots 4$, and $\e\wed \e=\e\otimes \e$. Inserting a covector $p=e^ap_a$ (here summation is over all values of $a$) as the first argument gives after rearrangement
\begin{multline}\label{eq.Tp}
    T(p,-)= p_1\left(T^{12}e_2+T^{13}e_3+T^{14}e_4+\t^{1}\e\right)+
    p_2\left(T^{21}e_1+T^{23}e_3+T^{24}e_4+\t^{2}\e\right)+\\
    p_3\left(T^{31}e_1+T^{32}e_2+T^{34}e_4+\t^{3}\e\right)+
    p_4\left(T^{41}e_1+T^{42}e_2+T^{43}e_3+\t^{4}\e\right)+\\
    p_{\hat 1}(-1)^{\pt}\left(\t^{1}e_1+\t^{2}e_2+\t^{3}e_3+ \t^{4}e_4+s\e\right)\,.
\end{multline}
Therefore the subspace $L_T\subset V$ is spanned by the vectors
\begin{align}\label{eq.f1}
    f^1&:= T^{12}e_2+T^{13}e_3+T^{14}e_4+\t^{1}\e\,,\\
    f^2&:= T^{21}e_1+T^{23}e_3+T^{24}e_4+\t^{2}\e\,,\\
    f^3&:= T^{31}e_1+T^{32}e_2+T^{34}e_4+\t^{3}\e\,,\\
    f^4&:= T^{41}e_1+T^{42}e_2+T^{43}e_3+\t^{4}\e\,,\\
    f^{\hat 1} &:= \t^{1}e_1+\t^{2}e_2+\t^{3}e_3+ \t^{4}e_4+s\e\,,
\end{align}
where $f^1, f^2, f^3, f^4$ are even and $f^{\hat 1}$ is odd. Note that $f^a=T^{ab}e_b$ for $a=1,\ldots,4$ and $f^{\hat 1}= T^{b\hat 1}e_b=-T^{\hat 1 b}(-1)^{\bt}$, so up to a sign we have obtained exactly the vectors   given by formula~\eqref{eq.LT} in Lemma~\ref{lm.LT}. For a ``generic'' $T$, we will have $L_T= V$\,.
\end{example}

In analogy with the classical case, we call an   even multivector $T\in \L^k(V)$ \emph{simple} if $T=\up_1\wed \ldots \wed \up_k$ for some independent even vectors $\up_1,\ldots,\up_k$.

\begin{lemma}
For a  simple multivector $T=\up_1\wed \ldots \wed \up_k$,
\begin{equation*}
    L_T=\spa(\up_1,\ldots,\up_k)\,.
\end{equation*}
\end{lemma}
\begin{proof} Clearly, filling all arguments but one in $T=\up_1\wed \ldots \wed \up_k$ gives a linear combination of $\up_i$, so $L_T\subset \spa(\up_1,\ldots,\up_k)$. For the converse, complement the vectors $\up_1,\ldots,\up_k$ to a basis in $V$ and fill the slots of $T$ with  $k-1$ of the first $k$ covectors of the dual basis, which would give    $\pm \up_i$, $i=1,\ldots,k$. Hence all $\up_i$ are in $L_T$.
\end{proof}

\begin{example}
In the setup of Example~\ref{ex.2in4.1}, if $T=\up\wed \wp$, then (since both $\up$ and $\wp$ are even), $\up\wed \wp=\frac{1}{2}(\up\otimes \wp- \wp\otimes \up)$ exactly as in the classical case, so for an arbitrary $p$,
\begin{equation*}
    T(p,-)= \frac{1}{2}(\langle\up, p\rangle  \wp- \langle\wp,p\rangle  \up)\,.
\end{equation*}
Specifically for the vectors $f^a$ that, as we have obtained, span the subspace $L_T$, we can see that they reduce to linear combinations of $\up$ and $\wp$, given explicitly as $f^a=\frac{1}{2}(-1)^{\at}\left|\begin{smallmatrix}
                     u^a & w^a \\
                     \up & \wp
                   \end{smallmatrix}\right|$\,, $a\in\{1,\ldots,4,\hat 1\}$.
\end{example}

\begin{theorem}[``abstract form of super Pl\"ucker relations'']
An even multivector $T\in \L^k(V)$ is simple if and only if it is non-degenerate and
\begin{equation}\label{eq.wwedT}
    \wp\wed T=0
\end{equation}
for every $\wp\in L_T$\,.
\end{theorem}
\begin{proof} 
If $T=\up_1\wed \ldots \wed \up_k$ for   even vectors $\up_1,\ldots,\up_k$, then (since every vector $\wp\in L_T$ is a linear combination of $\up_i$) $\wp\wed T= \sum c^i\up_i \wed \up_\wed \ldots \wed \up_k=0$. Here $\up_i\wed \up_i=0$ as the vectors are even. For the converse, we use the non-degeneracy to observe that  there are at least $k$ independent even vectors in $L_T$. Denote them $\up_1,\ldots,\up_k$. They can be complemented to a basis in  $V$ and we can write $T=\up_1\wed S+R$ where both $S$ and $R$ do not contain $\up_1$. The condition $\up_1\wed T=0$ gives $R=0$. Hence $T$ is divisible by $\up_1$, and the process can be repeated. Eventually $T$ is divisible by $\up_1\wed \ldots \wed \up_k$, and we conclude that   $T= c\, \up_1\wed \ldots \wed \up_k$ for a constant $c$.
\end{proof}

To get a concrete form of the relations for the components $T^{a_1\ldots a_k}$, we use Lemma~\ref{lm.LT} and calculate the wedge product in the equation
\begin{equation}\label{eq.wwedTcompon}
    T^{a_1,\ldots a_{k-1}b}e_b(-1)^{\bt(\at_1+\ldots+\at_{k-1})} \wed T^{c_1\ldots c_k}e_{c_1}\wed \ldots \wed e_{c_k} =0\,.
\end{equation}
We arrive at the following statement.

\begin{theorem}[``super Pl\"ucker relations for $k$-planes in an $n|m$-dimensional space'']
\label{thm.splueckcompon}
An even multivector $T\in \L^k(V)$ in an $n|m$-dimensional superspace $V$ is simple and thus defines a $k$-plane $L\subset V$ if and only if it is non-degenerate and its components~\eqref{eq.Twedge} satisfy the relations
\begin{multline}\label{eq.superpluecksimpl}
    T^{a_1\ldots a_{k-1}b}T^{c_1\ldots c_k}\,(-1)^{\bt(\at_1+\ldots+\at_{k-1}+\ct_1+\ldots+\ct_{k})}
    \\
    =\sum_{j=1}^{k}  T^{a_1\ldots a_{k-1}c_j}T^{c_1\ldots c_{j-1}bc_{j+1}\ldots c_k}\,
    (-1)^{\bt(\ct_1+\ldots+\ct_{j-1}) + \ct_j(\at_1+\ldots+\at_{k-1}+\ct_{j+1}+\ldots+\ct_{k})}
\end{multline}
for all combinations of indices $a_1$, \ldots, $a_{k-1}$, $b$, $c_1$, \ldots, $c_k$.
\end{theorem}
(In~\eqref{eq.superpluecksimpl}, in the $j$th term of the sum the index $b$ is swapped with the index $c_j$.)
\begin{proof}
Though it is a relatively straightforward calculation,
the exact answer and particularly the signs are very important, so we present the calculation in full. We start from equation~\eqref{eq.wwedTcompon} and move $T^{c_1\ldots c_k}$ to the left (note it has parity $\ct_1+\ldots+\ct_k$), obtaining
\begin{equation*}
     T^{a_1,\ldots a_{k-1}b} T^{c_1\ldots c_k}\, (-1)^{\bt(\at_1+\ldots+\at_{k-1})}
      (-1)^{\bt(\ct_1+\ldots+\ct_k)}\, e_b\wed e_{c_1}\wed \ldots \wed e_{c_k}
     =0\,.
\end{equation*}
This is an equality satisfied by a  $(k+1)$-vector and  holding   for all combinations of the free indices $a_1,\ldots a_{k-1}$.  To be able to get rid of the basis wedge products $e_b\wed e_{c_1}\wed \ldots \wed e_{c_k}$ in the expansion, we need  to have   coefficients that are antisymmetric\,---\,in the super sense\,---\,with respect to the indices $b, c_1,\ldots, c_k$. Since there is already   antisymmetry in the indices $c_1,\ldots, c_k$, the required alternation amounts to swapping the index $b$ with each of the indices $c_1,\ldots, c_k$. 
We can proceed as follows. Rewrite $e_b\wed e_{c_1}\wed \ldots \wed e_{c_k}$ identically as the sum
\begin{multline*}
   \frac{1}{k+1}\left(e_b\wed e_{c_1}\wed \ldots \wed e_{c_k}-
  (-1)^{\bt\ct_1} e_{c_1}\wed e_b\wed \ldots \wed e_{c_k}+\ldots + \right.\\
  \left.(-1)^j(-1)^{\bt(\ct_1+\ldots+\ct_{j})} e_{c_1}\wed \ldots\wed e_{c_{j-1}}\wed e_b\wed e_{c_{j+1}}\wed \ldots \wed e_{c_k} +\ldots + \right.\\
  \left.(-1)^k(-1)^{\bt(\ct_1+\ldots+\ct_{k})} e_{c_1}\wed \ldots\wed e_{c_k} \wed e_b\right)\,.
\end{multline*}
Denote for brevity $T^{a_1,\ldots a_{k-1}b}(-1)^{\bt(\at_1+\ldots+\at_{k-1})}$ by $R^b$.
The left-hand side of the above equation becomes the sum of $k+1$ terms  with the common term of the form
\begin{equation*}
    R^b T^{c_1\ldots c_k}\,(-1)^j(-1)^{\bt(\ct_{j+1}+\ldots+\ct_k)}\,
     e_{c_1}\wed \ldots\wed e_{c_{j-1}}\wed e_b\wed e_{c_{j+1}}\wed \ldots \wed e_{c_k}\,.
\end{equation*}
Rename $c_1$ to $b$, $c_2$ to $c_1$, \ldots, $c_j$ to $c_{j-1}$, and $b$ to $c_j$, leaving the indices $c_{j+1}$, \ldots, $c_k$ unchanged. This gives the common term as
\begin{equation*}
    R^{c_j} T^{bc_1\ldots c_{j-1}c_{j+1}\ldots c_k}\,
    (-1)^j(-1)^{\ct_j(\ct_{j+1}+\ldots+\ct_k)}e_b\wed  e_{c_1}\wed \ldots \wed e_{c_k}
\end{equation*}
where the coefficients have acquired the alternating form. Hence our equation  is equivalent to
\begin{equation*}
    R^b T^{c_1\ldots c_k}\,(-1)^{\bt(\ct_1+\ldots+\ct_k)}
    +\sum_{j=1}^k
    R^{c_j} T^{bc_1\ldots c_{j-1}c_{j+1}\ldots c_k}\,
    (-1)^j(-1)^{\ct_j(\ct_{j+1}+\ldots+\ct_k)} =0\,.
\end{equation*}
Finally, we can move the index $b$ in $T^{bc_1\ldots c_{j-1}c_{j+1}\ldots c_k}$ to the $j$th place. This   gives  the sign $(-1)^{j-1}(-1)^{\bt(\ct_1+\ldots +\ct_{j-1})}$ and we obtain
\begin{equation*}
    R^b T^{c_1\ldots c_k}\,(-1)^{\bt(\ct_1+\ldots+\ct_k)}
    -
    \sum_{j=1}^k
    R^{c_j} T^{c_1\ldots c_{j-1}bc_{j+1}\ldots c_k}\,
    (-1)^{\bt(\ct_1+\ldots +\ct_{j-1})}(-1)^{\ct_j(\ct_{j+1}+\ldots+\ct_k)} =0\,,
\end{equation*}
By substituting back $R^b=T^{a_1,\ldots a_{k-1}b}(-1)^{\bt(\at_1+\ldots+\at_{k-1})}$, we arrive at
\begin{multline*}
    T^{a_1,\ldots a_{k-1}b} T^{c_1\ldots c_k}\,(-1)^{\bt(\at_1+\ldots+\at_{k-1}+\ct_1+\ldots+\ct_k)}
    \\
   - \sum_{j=1}^k
    T^{a_1,\ldots a_{k-1}c_j}
    T^{c_1\ldots c_{j-1}bc_{j+1}\ldots c_k}\,
    (-1)^{\ct_j(\at_1+\ldots+\at_{k-1})}
    (-1)^{\bt(\ct_1+\ldots +\ct_{j-1})}(-1)^{\ct_j(\ct_{j+1}+\ldots+\ct_k)} =0\,,
\end{multline*}
as claimed.
\end{proof}

\subsection{Examples. Analysis of the super Pl\"{u}cker relations} \label{subsec.ex}

Theorem~\ref{thm.splueckcompon} completely describes the image of the ``super Pl\"{u}cker map'' \eqref{eq.simpleplueck}. In order to obtain a full analogy with the classical theory, we need in addition to show that the map given by \eqref{eq.simpleplueck} is invertible (on its image). We shall do that later and now can just assume this is a fact. (See Theorem~\ref{thm.embedalg} in the next section.) Then  we can conclude that, exactly as classical, in the algebraic case the super Grassmann manifold $G_k(V)$ is embedded as a (super) algebraic variety in the projective superspace $P(\L^k(V))$ by the ``super Pl\"{u}cker map'' \eqref{eq.simpleplueck} and is specified there by the ``super Pl\"{u}cker relations''~\eqref{eq.superpluecksimpl}.

Note that Theorem~\ref{thm.splueckcompon} also includes the  non-degeneracy 
as a condition additional to the algebraic relations. We will show that this extra condition can be dropped (at least for $k=2$, see Theorem~\ref{thm.reducplueck} below; but we believe it is true in general).

Recall first that  for $\dim V=n|m$,
\begin{align*}
    \dim \L^k(V)&=  \sum_{i=0}^k \binom{n}{k-i}\binom{m+i-1}{i}\Pi^i \\
    & =\Bigl. \sum_{0\leq i\leq k\,, \ \text{$i$ even}}  \binom{n}{k-i}\binom{m+i-1}{i}\ \Bigr|\,\sum_{0\leq i\leq k\,, \ \text{$i$ odd}} \binom{n}{k-i}\binom{m+i-1}{i}
\end{align*}
and $\dim P(\L^k(V))=\dim \L^k(V) -1$, while
\begin{equation*}
    \dim G_{k}(V) =k(n-k)    |km\,.
\end{equation*}
The simplest situation when   relations should occur is (as in the classical setup) when $k=2$. Till the end of the subsection we will be concerned with this case.  We have for it
\begin{equation*}
    \dim P(\L^2(V))=\Bigl.\frac{n(n-1)}{2}+ \frac{m(m+1)}{2}-1\;\Bigr|\;nm
\end{equation*}
and
\begin{equation*}
    \dim G_{2}(V) =2(n-2)    \,|\,2m\,.
\end{equation*}
Therefore (if $m>0$) relations occur already for $n=2$, i.e. for $2$-planes in the $2|m$-space! Let us put down some values of $\d:=\dim P(\L^2(V))-\dim G_{2}(V)$. We have
\begin{equation}\label{eq.codim}
    \d =\Bigl.\frac{n(n-1)}{2}-1-2(n-2)+ \frac{m(m+1)}{2}\;\Bigr|\;(n-2)m\,,
\end{equation}
so
\begin{align*}
    n=2: \qquad \d&= \Bigl.\frac{m(m+1)}{2}\;\Bigr|\;0\\
    n=3: \qquad \d&= \Bigl.\frac{m(m+1)}{2}\;\Bigr|\;m\\
    n=4: \qquad \d&= \Bigl.1+\frac{m(m+1)}{2}\;\Bigr|\;2m\\
    n=5: \qquad \d&= \Bigl.3+\frac{m(m+1)}{2}\;\Bigr|\;3m\\
\end{align*}
The ``regular case'' starts from $n=4$, while $n=2$ and $n=3$ are   two special cases. Note that $\d$ gives the number of independent relations,
while the actual number of the (super) Pl\"{u}cker relations can be greater   because they are not in  general independent.

Consider examples of super Pl\"{u}cker relations.

\begin{example} \label{ex.ktwo}Let $k=2$. For an even bivector $T=T^{ab}e_a\wed e_b$, the ``super Pl\"ucker'' relations will be
 \begin{equation}\label{eq.plueckfor2}
    T^{ab}T^{cd}(-1)^{\bt(\at+\ct+\dt)} =
    T^{ac}T^{bd}(-1)^{\ct(\at+ \dt)} +
    T^{ad}T^{cb}(-1)^{\bt\ct + \at\dt}\,.
 \end{equation}
 Note that here $\widetilde{T^{ab}}=\at+\bt$ and $T^{ab}=-(-1)^{\at\bt}T^{ba}$.
\end{example}

One can  see that a permutation of indices $a,b,c,d$ will not change the relation~\eqref{eq.plueckfor2} for given $a,b,c,d$. This follows from the construction of relations in the previous subsection, but also can be checked directly. (See Example~\ref{ex.ktwocont} below, where it will become clear.) In order to analyze~\eqref{eq.plueckfor2} for all possible combinations of indices, it is convenient to express $T$ in a more detailed way as
\begin{multline}\label{eq.Tgeneral}
    T=T^{ab}e_a\wed e_b\\
    =2\sum_{{a<b\,,}\atop{a,b=1,\ldots,n}} T^{ab}e_a\wed e_b + 2 \sum_{{a=1,\ldots,n\,,}\atop{\mu=1,\ldots,m}}\t^{a\mu}e_a\wed \e_{\mu} +
    \sum_{\la =1,\ldots,m} S^{\la \la}\e_{\la}\wed \e_{\la}+
    2\sum_{{\la< \mu\,,}\atop{\la,\mu=1,\ldots,m}} S^{\la \mu}\e_{\la}\wed \e_{\mu}\,.
\end{multline}
In the first line of~\eqref{eq.Tgeneral}, the summation is over all combinations of indices $a,b$ that run over $\{1,\ldots,n\}\sqcup \{\hat 1,\ldots,\hat m\}$. In the second line, we separate even and odd indices so that in the first term  the values of $a,b$ are restricted to $\{1,\ldots,n\}$ and we use $\la,\mu\in \{1,\ldots,m\}$. Here   $\e_{\mu}:=e_{\hat \mu}$, $\t^{a\mu}:=T^{a\hat \mu}$ and $S^{\la \mu}:=T^{\hat\la \hat\mu}$. The components $T^{ab}$, $a,b=1,\ldots,n$, and $S^{\la \mu}$ are even. The components $\t^{a\mu}$ are odd.
Compare~\eqref{eq.T} in Example~\ref{ex.2in4.1}.

\begin{example}[continued] \label{ex.ktwocont}
Let us analyze the relations obtained in Example~\ref{ex.ktwo}. Consider the cases when all indices $a,b,c,d$ are even; when only one index is odd, $d=\hat\mu$; when two indices are odd, $c=\hat\la$ and $d=\hat\mu$; when three indices are odd, $b=\nu$, $c=\hat\la$ and $d=\hat\mu$; and when all four indices are odd, $a=\hat \kappa$, $b=\nu$, $c=\hat\la$ and $d=\hat\mu$. Altogether the full set of  the super Pl\"{u}cker
relations for the variables $T^{ab}$, $S^{\la\mu}$ (even) and $\t^{a\la}$ (odd) reads:
\begin{align}
\label{eq.plueck.ev}
    T^{ab}T^{cd} &= T^{ac}T^{bd}  + T^{ad}T^{cb}  \\
    \intertext{(which is the   ordinary  Pl\"{u}cker relation for the even variables $T^{ab}$), and}
\label{eq.plueck.eeeo}
    T^{ab}\t^{c\mu} &= T^{ac}\t^{b\mu}  + T^{cb}\t^{a\mu}\,, \\
\label{eq.plueck.eeoo}
    T^{ab}S^{\la\mu} &= - \t^{a\la}\t^{b\mu}  - \t^{a\mu}\t^{b\la}\,, \\
\label{eq.plueck.eooo}
    \t^{a\nu}S^{\la\mu} &= - \t^{a\la}S^{\mu\nu}  - \t^{a\mu}S^{\la\nu}\,, \\
\label{eq.plueck.oooo}
    S^{\kappa\nu}S^{\la\mu} &= - S^{\kappa\la}S^{\mu\nu}  - S^{\kappa\mu}S^{\la\nu}\,.
\end{align}
Here $a,b,c,d=1,\ldots,n$ and $\kappa,\la,\mu,\nu=1,\ldots,m$. Relations~\eqref{eq.plueck.ev}, \eqref{eq.plueck.eeoo}  and~\eqref{eq.plueck.oooo} are even, while relations~\eqref{eq.plueck.eeeo},  \eqref{eq.plueck.eooo} are odd.
\end{example}

A further specialization of Examples~\ref{ex.ktwo} and~\ref{ex.ktwocont} is the simplest  super case  $m=1$.

\begin{example}[continued] \label{ex.ktwocontmone}
For $m=1$,   set $\t^{a}:=\t^{a1}$ and $s:=S^{11}$. Relations~\eqref{eq.plueck.eeeo}, \eqref{eq.plueck.eeoo}, \eqref{eq.plueck.eooo}  and~\eqref{eq.plueck.oooo}  become
\begin{align}
\label{eq.plueck.ttheta}
    T^{ab}\t^{c} &= T^{ac}\t^{b}  + T^{cb}\t^{a}\,, \\
\label{eq.plueck.ts}
    T^{ab}s &= - 2\t^{a}\t^{b}\,, \\
\label{eq.plueck.thetas}
    \t^{a}s &= 0\,, \\
\label{eq.plueck.skw}
    s^2 &= 0\,.
\end{align}
The nilpotence relation $s^2=0$ is very important. Recall that we assume that at least one even component of the bivector $T$ is invertible. Since $s$ cannot be invertible by~\eqref{eq.plueck.skw}, we obtain that one of $T^{ab}$, $a,b=1,\ldots,n$, must be invertible, i.e. in other words,    the non-degeneracy of $T$. Furthermore,   we can analyze the relations by considering them in a domain  in which a particular  variable  $T^{ab}$ is invertible (since such domains cover the image of the super Grassmannian in $P(\L^2(V))$). We can express
\begin{align}\label{eq.thetacviathetaab}
    \t^c&=\frac{1}{T^{ab}}(T^{ac}\t^b+T^{cb}\t^a)
\intertext{and}
\label{eq.sviathetaab}
    s&=-\frac{2}{T^{ab}}\t^a\t^b\,.
\end{align}
Note that~\eqref{eq.thetacviathetaab} holds for all $c$ including $a$ and $b$; however, for $c=a$ or $c=b$, the relation~\eqref{eq.thetacviathetaab} becomes $\t^a=\t^a$ or $\t^b=\t^b$, respectively. Relations~\eqref{eq.sviathetaab} and~\eqref{eq.thetacviathetaab} (for $c\neq a,b$) are independent. Since there are exactly $1|n-2$ of them, we may think that they give extra independent relations in addition to independent relations on $T^{ab}$ so to obtain the number in~\eqref{eq.codim} for $m=1$.

We can show that formulas~\eqref{eq.thetacviathetaab} and~\eqref{eq.sviathetaab}, with $a,b$ fixed, taken together with~\eqref{eq.plueck.ev} imply the full set of relations~\eqref{eq.plueck.ttheta}, \eqref{eq.plueck.ts}, \eqref{eq.plueck.thetas} and \eqref{eq.plueck.skw}. Indeed, consider~\eqref{eq.plueck.ttheta} where instead of $a,b$ we have some $p,q$\,:
\begin{equation*}
    T^{pq}\t^c=T^{pc}\t^q+T^{cq}\t^p\,.
\end{equation*}
By substituting $\t^c,\t^p,\t^q$ from~\eqref{eq.thetacviathetaab} and getting rid of the denominator, we obtain for the left-hand side $T^{pq}(T^{ac}\t^b+T^{cb}\t^a)$ and for the right-hand side, $T^{pq}(T^{ac}\t^b+T^{qb}\t^a)+T^{cq}(T^{ap}\t^b+T^{pb}\t^a)$. Comparing the coefficients at $\t^a$ and $\t^b$ gives us the classical Pl\"{u}cker relations, viz. $T^{pq}T^{ac}=T^{pc}T^{aq}+T^{cq}T^{ap}$ and $T^{pq}T^{cb}=T^{pc}T^{qb}+T^{cq}T^{pb}$\,. Hence~\eqref{eq.plueck.ttheta} holds for arbitrary combinations of indices. Next we need to establish
\begin{equation*}
    T^{pq}s=-2\t^p\t^q
\end{equation*}
for arbitrary $p$ and $q$. By substituting $s$ from~\eqref{eq.sviathetaab} and $\t^p$, $\t^q$ from~\eqref{eq.thetacviathetaab} and getting rid of the denominator $(T^{ab})^2$ and the common factor   $-2$, we arrive at a relation $T^{pq}T^{ab}\t^a\t^b=(T^{ap}\t^b+T^{pb}\t^a)(T^{aq}\t^b+T^{qb}\t^a)$, equivalent to $T^{pq}T^{ab}=-T^{ap}T^{qb}+T^{pb}T^{aq}$, which is again the classical Pl\"{u}cker relation. Hence we get~\eqref{eq.plueck.ts} (for arbitrary indices). Further, we need the relation
\begin{equation*}
    \t^c s=0
\end{equation*}
for arbitrary $c$. This follows from~\eqref{eq.sviathetaab} and~\eqref{eq.thetacviathetaab} directly (by vanishing of the product of $\t^a\t^b$ with a linear combination of $\t^a$ and $\t^b$). Finally, the formula~\eqref{eq.sviathetaab} for $s$ implies $s^2=0$.
\end{example}

Analysis performed in Example~\ref{ex.ktwocontmone} for $k=2$ and $n|m=n|1$ tells that in this case, the condition of  non-degeneracy in Theorem~\ref{thm.splueckcompon} can be dropped, because it follows from the algebraic  relations themselves; and also that the full set of  super Pl\"{u}cker relations~\eqref{eq.plueck.ev}, \eqref{eq.plueck.ttheta}, \eqref{eq.plueck.ts}, \eqref{eq.plueck.thetas} and \eqref{eq.plueck.skw} can be effectively reduced to   two relations~\eqref{eq.plueck.ev}, \eqref{eq.plueck.ttheta} that involve only part of the variables.

We will show that this holds for  general $n|m$   when $k=2$. (In the next Section~\ref{sec.general}, we will show by a different method how this subset of ``essential'' Pl\"{u}cker variables and the corresponding ``reduced'' set of  super Pl\"{u}cker relations arise for arbitrary $k$. See Theorem~\ref{thm.essplueckr0}.)

\begin{theorem}[``reduced super Pl\"{u}cker relations for $k=2$ in $n|m$-space]
\label{thm.reducplueck}
The super Pl\"{u}cker relations 
for the components $T^{ab}$, $\t^{a\la}$ and $S^{\la\mu}$ of an even bivector $T$,
where it is assumed that at least one even component is invertible, imply that all  the variables $S^{\la\mu}$ are nilpotent and hence that at least one of the variables $T^{ab}$ is invertible (i.e. that the bivector $T$ is non-degenerate). The  algebra generated by $T^{ab}$, $\t^{a\la}$ and $S^{\la\mu}$ is isomorphic to the algebra generated by $T^{ab}$, $\t^{a\la}$ alone subject to the relations 
\begin{align}
\label{eq.plueck.even}
    T^{ab}T^{cd} &= T^{ac}T^{bd}  + T^{ad}T^{cb}\,,  \\
\label{eq.plueck.odd}
    T^{ab}\t^{c\mu} &= T^{ac}\t^{b\mu}  + T^{cb}\t^{a\mu}
\end{align}
(which we shall refer to as the `reduced super Pl\"{u}cker relations').
\end{theorem}
\begin{proof}
Consider $T^{ab}$, $\t^{a\la}$ and $S^{\la\mu}$ satisfying the relations~\eqref{eq.plueck.ev}, \eqref{eq.plueck.eeeo},  \eqref{eq.plueck.eeoo},  \eqref{eq.plueck.eooo}  and~\eqref{eq.plueck.oooo}. We first show that they imply the nilpotence of $S^{\la\mu}$ for each $\la,\mu$. Indeed, consider equation~\eqref{eq.plueck.oooo} with $\kappa\la=\mu\nu$, which gives
\begin{equation*}
    S^{\mu\nu}S^{\nu\mu}  = - S^{\mu\nu}S^{\mu\nu}  - S^{\mu\mu}S^{\nu\nu}
\end{equation*}
or (since $S^{\mu\nu}=S^{\nu\mu}$)
\begin{equation*}
    2(S^{\mu\nu})^2=- S^{\mu\mu}S^{\nu\nu}\,.
\end{equation*}
In particular, if $\mu=\nu$, we get $(S^{\mu\mu})^2=0$.
Hence $(S^{\la\mu})^2=0$ for all $\la,\mu$. So the variables $S^{\la\mu}$ cannot be invertible; therefore one of the variables $T^{ab}$ must be invertible. Fix $a,b$. In the domain where $T^{ab}$ is invertible, we can express
\begin{equation*}
    \t^{c\mu}=\frac{1}{T^{ab}}(T^{ac}\t^{b\mu}+T^{cb}\t^{a\mu}) \qquad \text{and}\qquad
    S^{\la\mu}=-\frac{1}{T^{ab}}(\t^{a\la}\t^{b\mu}+\t^{a\mu}\t^{b\la})
\end{equation*}
(automatically there will be $S^{\la\mu}=S^{\mu\la}$). This holds for all $\la,\mu,c$.
We shall show now that if $\t^{c\mu}$ and $S^{\la\mu}$ are given by these formulas for all $\la,\mu=1,\ldots,m$ (without placing any restrictions on the variables $\t^{a\mu}$ and $\t^{b\mu}$  and assuming that $T^{pq}$ for all $p,q=1,\ldots,n$ satisfy the classical Pl\"{u}cker relations), then the full set of super Pl\"{u}cker relations 
is satisfied. Consider the relation
\begin{equation*}
    T^{pq}\t^{c\mu}=T^{pc}\t^{q\mu} + T^{cq}\t^{p\mu}\,,
\end{equation*}
for arbitrary $p,q,c$ and $\mu$, which we need to establish. Substitute $\t^{c\mu}, \t^{q\mu}, \t^{p\mu}$ from the formulas above. After the multiplication by the denominator $T^{ab}$, this gives the relation
\begin{equation*}
    T^{pq}(T^{ac}\t^{b\mu}+ T^{cb}\t^{a\mu})=
    T^{pc}(T^{aq}\t^{b\mu}+ T^{qb}\t^{a\mu}) +
    T^{cq}(T^{ap}\t^{b\mu}+ T^{pb}\t^{a\mu})\,,
\end{equation*}
which is identically satisfied by the virtue of the classical Pl\"{u}cker relations for $T^{pq}$.
Consider now another relation to be established:
\begin{equation*}
    T^{pq}S^{\la\mu}=-\t^{p\la}\t^{q\mu} - \t^{p\mu}\t^{q\la}\,.
\end{equation*}
Again substitute the expressions for $S^{\la\mu}$, $\t^{p\la}$, etc. from the above. This will give after the multiplication by $-(T^{ab})^2$ the relation
\begin{multline*}
     T^{pq}T^{ab}(\t^{a\la}\t^{b\mu}+\t^{a\mu}\t^{b\la})
    = (T^{ap}\t^{b\la}+T^{pb}\t^{a\la})(T^{aq}\t^{b\mu}+T^{qb}\t^{a\mu}) \\
    + (T^{ap}\t^{b\mu}+T^{pb}\t^{a\mu})(T^{aq}\t^{b\la}+T^{qb}\t^{a\la})\,,
\end{multline*}
where in the right-hand side after opening the brackets the terms with $\t^{a\la}\t^{a\mu}$ and $\t^{b\la}\t^{b\mu}$ will cancel and the remaining terms assemble into $(T^{pa}T^{qb}+T^{pb}T^{aq})(\t^{a\la}\t^{b\mu}+\t^{a\mu}\t^{b\la})$, which equals the left-hand side by by the virtue of the classical Pl\"{u}cker relation. The next relation that we need to establish is
\begin{equation*}
    \t^{c\nu}S^{\la\mu}=-\t^{c\la}S^{\mu\nu} - \t^{c\mu}S^{\la\nu}\,.
\end{equation*}
After the substitution and getting rid of the denominator, we obtain
$(T^{ac}\t^{b\nu}+T^{cb}\t^{a\nu})(\t^{a\la}\t^{b\mu}+ \t^{a\mu}\t^{b\la})$
for the left-hand side
and $-(T^{ac}\t^{b\la}+T^{cb}\t^{a\la})(\t^{a\mu}\t^{b\mu}+\t^{a\nu}\t^{b\mu})$
for the right-hand side. By multiplying through and simplification, the right-hand side will give
\begin{multline*}
    -T^{ac}\left(\t^{b\la}(\t^{a\mu}\t^{b\nu}+\underline{\t^{a\nu}\t^{b\mu}})+
    \t^{b\mu}(\t^{a\la}\t^{b\nu}+\underline{\t^{a\nu}\t^{b\la}})\right)-\\
    T^{cb}\left(\t^{a\la}(\underline{\t^{a\mu}\t^{b\nu}}+\t^{a\nu}\t^{b\mu})
    +\t^{a\mu}(\underline{\t^{a\la}\t^{b\nu}}+\t^{a\nu}\t^{b\la})\right)=\\
    -T^{ac}\t^{b\nu}(\t^{b\la}\t^{a\mu}+\t^{b\mu}\t^{a\la})
    +T^{cb}\t^{a\nu}(\t^{a\la}\t^{b\mu}+\t^{a\mu}\t^{b\la})=
    (T^{ac}\t^{b\nu} + T^{cb}\t^{a\nu})(\t^{a\la}\t^{b\mu}+\t^{a\mu}\t^{b\la})\,,
\end{multline*}
which is exactly the left-hand side
(underlined are the terms that cancel). Hence this identity is proved.
And finally we need to check the relation~\eqref{eq.plueck.oooo}, which can be written as
\begin{equation*}
    S^{\kappa\nu}S^{\la\mu}  +S^{\kappa\la}S^{\mu\nu}  + S^{\kappa\mu}S^{\la\nu}=0\,,
\end{equation*}
where  everything is put to one side. After substitution and multiplying by the denominator, we have
\begin{multline*}
    (\t^{a\kappa}\t^{b\la}+\t^{a\la}\t^{b\kappa})(\t^{a\mu}\t^{b\nu}+\t^{a\nu}\t^{b\mu})+
    (\t^{a\kappa}\t^{b\mu}+\t^{a\mu}\t^{b\kappa})(\t^{a\la}\t^{b\nu}+\t^{a\nu}\t^{b\la})+\\
    (\t^{a\kappa}\t^{b\nu}+\t^{a\nu}\t^{b\kappa})(\t^{a\la}\t^{b\mu}+\t^{a\mu}\t^{b\la})\,,
\end{multline*}
and after opening the brackets all the terms in the long sum will cancel. Thus the theorem is fully proved.
\end{proof}

The statement of Theorem~\ref{thm.reducplueck} can be given a geometric interpretation as follows. The homogeneous coordinate ring of the super Grassmannian $G_2(V)$ with $\dim V=n|m$ is generated by even variables $T^{ab}=-T^{ba}$ and odd variables $\t^{a\mu}$, $a,b=1,\ldots,n$, $\mu=1,\ldots,m$,  satisfying the relations~\eqref{eq.plueck.even} and \eqref{eq.plueck.odd}.  Equation~\eqref{eq.plueck.even} being the classical Pl\"{u}cker relation describes the underlying ordinary Grassmannian $(G_2(V))_0=G_2(V_0)$, $\dim V_0=n$. What is the meaning of  the odd equation~\eqref{eq.plueck.odd}?
\begin{theorem}
Equation~\eqref{eq.plueck.odd} for each $\mu=1,\ldots,m$, specifies the parity-reversed tautological subspace in $V_0$ for each point of the Grassmannian $G_2(V_0)$. Thus the relations~\eqref{eq.plueck.even} and \eqref{eq.plueck.odd} together describe the  super Grassmannian $G_2(V)$ as the (super) vector bundle $\Pi(\underbrace{E\oplus\ldots \oplus E}_{\text{$m$ times}})$ over the ordinary Grassmannian $G_2(V_0)$, where $E\to G_2(V_0)$ is the tautological bundle.
\end{theorem}
\begin{proof}
Indeed, suppose $T_0=T^{ab}e_a\wed e_b\in \L^2(V_0)$ (where $a,b=1,\ldots,n$) satisfies~\eqref{eq.plueck.even}.  Then $T^{ab}=\frac{1}{2}\begin{vmatrix}
                          u_1^a & u_1^b \\
                          u_2^a & u_2^b
                        \end{vmatrix}
$ for some vectors $\up_1,\up_2$ that span a plane $L_0\subset V_0$ and the relation~\eqref{eq.plueck.odd} becomes
\begin{equation*}
    \begin{vmatrix}
      \t^{a\mu} & \t^{b\mu} & \t^{c\mu} \\
      u_1^a & u_1^b & u_1^c \\
      u_2^a & u_2^b & u_2^c
    \end{vmatrix}=0\,,
\end{equation*}
which exactly the condition for an ``odd vector''  with the coordinates $\t^{a\mu}$ (for a fixed $\mu$) in $V_0$ to belong to the subspace spanned by $\up_1,\up_2$\,.
\end{proof}


For further illustration, we consider   two exceptional ``low-dimensional'' examples peculiar for the super case.

\begin{example}[$k=2$ in $2|m$-space]
\label{ex.2in2m}
We have for a bivector $T$,
\begin{equation}\label{eq.2in2m}
    T=2T^{12}e_1\wed e_2 +2\t^{1\mu} e_1\wed \e_{\mu} +\sum_{\mu}S^{\mu\mu}\e_{\mu}\wed \e_{\mu} + 2\sum_{\la<\mu} S^{\la\mu}\e_{\la}\wed \e_{\mu}\,,
\end{equation}
where $T^{12}$ is invertible.
The non-empty super Pl\"{u}cker relations are
\begin{align}
\label{eq.plueck.12oo}
    T^{12}S^{\la\mu} &= - \t^{1\la}\t^{2\mu}  - \t^{1\mu}\t^{2\la}\,, \\
\label{eq.plueck.1s}
    \t^{a\nu}S^{\la\mu} &= - \t^{a\la}S^{\mu\nu}  - \t^{a\mu}S^{\la\nu}\,, \\
\label{eq.plueck.ss}
    S^{\kappa\nu}S^{\la\mu} &= - S^{\kappa\la}S^{\mu\nu}  - S^{\kappa\mu}S^{\la\nu}\,.
\end{align}
There are no reduced relations.
We can normalize $T^{12}=1$. Then   $\t^{1\mu}$ and $\t^{2\mu}$ remain as the only independent variables.  The variables  $S^{\la\mu}$ for all $\la,\mu=1,\ldots,m$  are expressed
from~\eqref{eq.plueck.12oo} as
$S^{\la\mu}=-\t^{1\la}\t^{2\mu}-\t^{1\mu}\t^{2\la}$   and relations~\eqref{eq.plueck.1s},~\eqref{eq.plueck.ss} are automatically satisfied. This corresponds to 
$G_2(\R{2|m})\cong \R{0|2m}$.
\end{example}

\begin{example}[$k=2$ in $3|m$-space]
We have
\begin{multline}\label{eq.2in3m}
    T=2\left(T^{12}e_1\wed e_2+T^{13}e_1\wed e_3+T^{23}e_2\wed e_3\right) 
    +2\left(\t^{1\mu} e_1\wed \e_{\mu}+\t^{2\mu} e_2\wed \e_{\mu}+ \t^{3\mu} e_3\wed \e_{\mu}\right)\\
     +\sum_{\mu}S^{\mu\mu}\e_{\mu}\wed \e_{\mu} + 2\sum_{\la<\mu} S^{\la\mu}\e_{\la}\wed \e_{\mu}\,.
\end{multline}
There are no relations for $T^{12},T^{13},T^{23}$. The variables $\t^{a\mu}$ satisfy the linear relation (same for each $\mu$)
\begin{equation}
    T^{12}\t^{3\mu}=T^{13}\t^{2\mu}-T^{23}\t^{1\mu}\,,
\end{equation}
so in the domain where e.g. $T^{12}$ is invertible, the variables $\t^{1\mu}, \t^{2\mu}$ can be taken as independent. The variables $S^{\la\mu}$ can be expressed as in   previous Example~\ref{ex.2in2m} via $\t^{1\mu}, \t^{2\mu}$ and all relations for them will be identically satisfied. The underlying space of the super Grassmannian $G_{2|0}(\R{3|m})$
is the projective space $\RR P^3$ and our analysis of the super Pl\"{u}cker relations gives the description
$G_{2|0}(\R{3|m})\cong \underbrace{\Pi E_1^{\perp}\oplus \ldots \Pi E_1^{\perp}}_{\text{$m$ times}}$, where $E_1^{\perp}\to \RR P^3$ is the annihilator bundle (of rank $2$) for the tautological line bundle on the projective space.
\end{example}

The simplest example with a non-trivial even part of the reduced super Pl\"{u}cker relations is $k=2$ and $n|m=4|1$ (the setup of Example~\ref{ex.2in4.1}).
\begin{example}[continued from Example~\ref{ex.2in4.1} and Example~\ref{ex.ktwocontmone}]
We have
\begin{equation}
    T= 2\sum_{a<b} T^{ab}e_a\wed e_b + 2 \t^{a}e_a\wed \e  +
     s\,\e \wed \e \,.
\end{equation}
Here $a,b=1,\ldots,4$. The components $T^{ab}$ satisfy the only non-trivial relation
\begin{equation}\label{eq.classplueck24}
    T^{13}T^{24}=T^{12}T^{34}+T^{14}T^{23}\,,
\end{equation}
which is the classical Pl\"{u}cker relation for $G_2(\R{4})$ and $\t^a$ satisfy the linear relations
\begin{equation}\label{eq.reltheta2in41}
    T^{ab}\t^c=T^{ac}\t^b+T^{cb}\t^a\,,
\end{equation}
for all $a,b,c$ (there will be four different relations here).
Besides that, there is the relation
\begin{equation}
    T^{ab}s=-2\t^a\t^b
\end{equation}
from where $s$ can be expressed in each open domain where $T^{ab}$ is invertible, and further relations for $\t^a$ and $s$
\begin{equation}
    \t^a\,s=0  \quad \text{and}\quad s^2=0\,,
\end{equation}
which will be then identically satisfied. The  reduced  super Pl\"{u}cker  relations \eqref{eq.classplueck24},\eqref{eq.reltheta2in41} correspond  to the description $G_{2|0}(\R{4|1})\cong \Pi E$, where $E\to G_2(\R{4})$ is the tautological bundle.
\end{example}

This example of $2$-planes in the $4|1$-space was the only example of super Pl\"{u}cker embedding  previously known in the literature and the super Pl\"{u}cker relations for it were first obtained by Cervantes--Fioresi--Lled\'{o}~\cite{cervantes:quantum-chiral-2011}. (The super Grassmannian $G_{2|0}(4|1)$ can be interpreted   as a ``compactified super Minkowski'' space following twistor approach~\cite{manin:gaugeeng}.)

\section{Super Pl\"{u}cker embedding. General case: $r|s$-planes in $n|m$-space} \label{sec.general}

\subsection{Reformulation as a matrix problem} \label{subsec.reformul}

We return to the case of an $r|s$-plane $L$ in the $n|m$-dimensional superspace $V$. Such a  plane in spanned by a basis consisting of $r$ even and $s$ odd vectors,
\begin{equation}\label{eq.rsplane}
    L=\spa(\up_1,\ldots,\up_r\,|\,\up_{\hat 1},\ldots,\up_{\hat s})\,.
\end{equation}
As we already discussed, the naive idea of taking the wedge product of all basis vectors regardless of their parities (which would be an element of $\L^{r+s}(V)$) will not lead to a quantity invariantly associated with the plane. Instead we should use the ``non-linear wedge product'' introduced in Example~\ref{ex.nonlinwedge} in subsection~\ref{subsec.exteriorpowers}, $[\up_1,\ldots,\up_r\,|\,\up_{\hat 1},\ldots,\up_{\hat s}]\in \L^{r|s}(V)$. By definition, it is a function of $r$ even and $s$ odd covectors given by
\begin{equation}\label{eq.nwedge}
    [\up_1,\ldots,\up_r\,|\,\up_{\hat 1},\ldots,\up_{\hat s}](p^1,\ldots,p^r\,|\,p^{\hat 1},\ldots,p^{\hat s})=
    \Ber\left(\langle \up_i, p^j\rangle\right)\,.
\end{equation}
Here $i,j\in \{1,\ldots,r\}\cup \{\hat 1,\ldots,\hat s\}$ and $i$ is the first index (row number) while $j$ is the second index (column number).
\begin{example}
If $s=0$, then all vectors $\up_i$ are even and
\begin{equation*}
    [\up_1,\ldots,\up_r](p^1,\ldots,p^r)=\det \left(\langle \up_i, p^j\rangle\right)\,,
\end{equation*}
where $i,j=1,\ldots,r$. Hence in this case
\begin{equation*}
    [\up_1,\ldots,\up_r]= r!\,\up_1\wed \ldots\wed\up_r\,.
\end{equation*}
\end{example}
So we have agreement with constructions in the previous section.
If the basis vectors $\up_i$ are replaced by other basis vectors $\up_{i'}=g_{i'}{}^{i}\up_i$, where $g=(g_{i'}{}^{i})\in \GL(r|s)$, then
\begin{equation}\label{eq.changenonlinwed}
    [\up_{1'},\ldots,\up_{r'}\,|\,\up_{\hat 1'},\ldots,\up_{\hat s'}]=\Ber g\cdot[\up_1,\ldots,\up_r\,|\,\up_{\hat 1},\ldots,\up_{\hat s}]\,.
\end{equation}
Hence the assignment $L\mapsto [\up_1,\ldots,\up_r\,|\,\up_{\hat 1},\ldots,\up_{\hat s}]$ can be seen as an  invertible
$\L^{r|s}(V)$-valued section of the line bundle $(\Ber E)^*\to G_{r|s}(V)$, where $E\to G_{r|s}(V)$ is the tautological bundle (the ``tautological section''). Passing to the projectivization gives a well-defined map
\begin{equation}\label{eq.splueckprelim}
    \pl\co G_{r|s}(V) \to P(\L^{r|s}(V))\,, \quad L\mapsto \pl(L):=\cl [\up_1,\ldots,\up_r\,|\,\up_{\hat 1},\ldots,\up_{\hat s}]
\end{equation}
(here $\cl$ stands for  equivalence class). This is our \textbf{preliminary definition} of the \emph{super Pl\"{u}cker map} for the general case of $r|s$-planes. We will shortly see that it is not entirely satisfactory and we will  how it needs to be  amended.

Fix  a basis in the superspace $V$. Then vectors in $V$ are described by their coordinates, which we will write as rows (i.e. left coordinates). A basis $\up_i$ spanning a plane $L$ will be represented by an $r|s\times n|m$ matrix $U$, which we consider as the matrix of homogeneous coordinates of a point $L$ of the super Grassmannian $G_{r|s}(V)$ (see subsection~\ref{subsec.recol}). Homogeneity means with respect to the action of the supergroup $\GL(r|s)$, i.e. $U\sim gU$ for all $g\in GL(r|s)$. Similarly, covectors $p\in V^*$ can be represented by their right coordinates written as columns, so that the pairing between $V$ and $V^*$ takes the form $\langle\up, p\rangle=u^ap_a$ (summation over all indices $a\in   \{1,\ldots,n\}\cup \{\hat 1,\ldots,\hat m\}$). Respectively, an array of $r|s$ covectors will be represented by an $n|m\times r|s$ matrix.

Our   problem can reformulated in terms of matrices as follows.

Given an even $r|s\times n|m$ matrix $U$ of rank $r|s$, we define the  \textbf{``Pl\"{u}cker transform''} of $U$ as a function $\pl(U)$ on even $n|m\times r|s$ matrices  $P$ of rank $r|s$,
\begin{equation}\label{eq.plup}
    \pl(U)(P):=\Ber (UP)\,.
\end{equation}
(We use the same notation as in~\eqref{eq.splueckprelim}, though in~\eqref{eq.splueckprelim} we consider equivalence classes.)
Now the   problem is:

(1) to find the matrix $U$ from a given function $\pl(U)$, and

(2)  to describe the image of $\pl$ among functions of matrices.

Note that $\pl(U)(P)$ is a rational function of the matrix $P$  with the domain of definition depending on $U$. Whether one should regard    the map $U\mapsto \pl(U)$   itself as ``rational'', is a subtle point. On one hand, it is given by rational functions of the matrix entries of $U$. On the other hand, the function $\pl(U)(P) =\Ber (UP)$ as a rational function of a matrix $P$ is defined for \emph{all} matrices $U$ of rank $r|s$; there is no situation of division by  zero (identically) for some matrix $U$. The expression $\pl(U)(P)$ has poles as a function of $P$ and these poles vary with $U$. From this viewpoint, we perhaps can think of $U\mapsto \pl(U)$  as a regular map (from matrices to functions).

From~\eqref{eq.changenonlinwed}, which in the matrix form becomes
\begin{equation}
    \pl({gU})  =\Ber g \cdot\pl(U)\,,
\end{equation}
for   $g\in \GL(r|s)$, it follows that the matrix $U$ is defined by the function $\pl({U})$   non-uniquely. To take care of this non-uniqueness, one can impose an extra   ``gauge condition''. Basically, this is a different language for replacing $U$ by a matrix of inhomogeneous coordinates in a particular affine chart on the super Grassmannian. Such charts (see subsection~\ref{subsec.recol}) are   numbered by   choices of $r|s$ columns in $U$, i.e. $r$ even and $s$ odd columns giving an invertible square submatrix. One can divide by it from the left and obtain a matrix where the corresponding submatrix is replaced by the $r|s\times r|s$ identity  matrix. Then we can explore if the Pl\"ucker transform is invertible when restricted on such gauge-fixed matrices. Before developing   general theory, let us look at an example.

\subsection{Example: $1|1$-planes in $2|2$-space}
\label{subsec.11in22}

Let $U$ be a $1|1\times 2|2$ matrix of a gauge-fixed form
\begin{equation}\label{eq.matrixA}
   U=\begin{pmatrix}
        x   & 1 & \vline & \x  & 0 \\
        \hline \vphantom{\int_a^b}
        \h & 0 & \vline & y & 1 \\
      \end{pmatrix}\,.
\end{equation}
Consider its Pl\"ucker transform $\pl(U)$. For a matrix $P$,
\begin{equation}\label{eq.matrixB}
    P=\begin{pmatrix}
        p^{1}_1 &   \vline & p^{\hat 1}_1 \\[3pt]
        p^{1}_2 &   \vline & p^{\hat 1}_2 \\[3pt]
        \hline \vphantom{\int_a^b}
        p^{1}_{\hat 1} & \vline & p^{\hat 1}_{\hat 1} \\[3pt]
        p^{1}_{\hat 2} & \vline & p^{\hat 1}_{\hat 2} \\
      \end{pmatrix}=
      \begin{pmatrix}
        p_1 & \vline & \pi_1 \\[3pt]
        p_2 & \vline & \pi_2 \\[3pt]
        \hline \vphantom{\int_a^b}
        p_{\hat 1} & \vline & \pi_{\hat 1} \\[3pt]
        p_{\hat 2} & \vline & \pi_{\hat 2} \\
      \end{pmatrix}\,,
\end{equation}
we have explicitly
\begin{multline}\label{eq.plAP}
   \pl(U)(P)=\Ber(UP)=\Ber\begin{pmatrix}
                             x p_1 + p_2 +\x p_{\hat 1} & \vline
                             & x \pi_1 + \pi_2 +\x \pi_{\hat 1} \\
                             \hline \vphantom{\int_a^b}
                             \h p_1 + y p_{\hat 1} + p_{\hat 2} & \vline
                             & \h \pi_1 + y \pi_{\hat 1} + \pi_{\hat 2} \\
                           \end{pmatrix} \\
   = \frac{x p_1 + p_2 +\x p_{\hat 1}
    -(\h p_1 + y p_{\hat 1} + p_{\hat 2})
    (\h \pi_1 + y \pi_{\hat 1} + \pi_{\hat 2})^{-1}
    (x \pi_1 + \pi_2 +\x \pi_{\hat 1})}{\h \pi_1 + y \pi_{\hat 1} + \pi_{\hat 2}}\,.
\end{multline}
Two facts can be observed. For an even $P$, the entries $\pi_1$, $\pi_2$, $p_{\hat 1}$ and $p_{\hat 2}$ are odd, so the entries $\x$ and $\h$ of $U$ always appear in $\pl(U)(P)$ accompanied by a nilpotent (odd) factor. This is no wonder, since $\Ber (UP)$ is even, so the coefficient at
$\x$ or $\h$ in the expansion    must be odd. Secondly, the entry $y$ of the matrix $U$ appears only in the denominator in $\pl(U)(P)$, hence solving for it   requires taking the inverse.

The first observation, at the first glance, makes  it impossible to solve for $\x$ and $\h$. However, the way out is given by the property of $\pl(U)$ following from the properties of Berezinian. Namely, as Berezinian is a multilinear function with respect to even rows or even columns, the function $\pl(U)(P)$ is also multilinear in the $r$ even columns of the matrix $P$. So in our particular case, $\pl(U)(P)$ can be
defined for  a column of arbitrary  parity  put  in the first (even) position  (and keeping an odd independent column  in the odd position).

Let $e^1$, $e^2$, $e^{\hat 1}$ and $e^{\hat 2}$ be the covectors of the standard basis:
\begin{equation*}
    e^1=\begin{pmatrix}
          1 \\
          0 \\
          \hline \vphantom{\int_a^b}
          0 \\
          0 \\
        \end{pmatrix}\,,
    e^2=\begin{pmatrix}
          0 \\
          1 \\
          \hline \vphantom{\int_a^b}
          0 \\
          0 \\
        \end{pmatrix}\,,
    e^{\hat 1}=\begin{pmatrix}
          0 \\
          0 \\
          \hline \vphantom{\int_a^b}
          1 \\
          0 \\
        \end{pmatrix}\,,
     e^{\hat 2}=\begin{pmatrix}
          0 \\
          0 \\
          \hline \vphantom{\int_a^b}
          0 \\
          1 \\
        \end{pmatrix}\,.
\end{equation*}
Then, writing the argument of $\pl(U)$ as an array of covectors, we have:
\begin{align*}
    \pl(U)(e^1|e^{\hat 1})&= \Ber\begin{pmatrix}
                                 x & \vline & \x \\
                                 \hline \vphantom{\int_a^b}
                                 \h & \vline & y \\
                               \end{pmatrix}= \frac{x-\x y^{-1}\h}{y}\,,\\
    \pl(U)(e^1|e^{\hat 2})&= \Ber\begin{pmatrix}
                                 x  & \vline& 0 \\
                                 \hline \vphantom{\int_a^b}
                                 \h & \vline & 1 \\
                               \end{pmatrix}= x\,,\\
      \pl(U)(e^2|e^{\hat 1})&= \Ber\begin{pmatrix}
                                 1 & \vline & \x \\
                                 \hline \vphantom{\int_a^b}
                                 0 & \vline & y \\
                               \end{pmatrix}= \frac{1}{y}\,,\\
      \pl(U)(e^2|e^{\hat 2})&= \Ber\begin{pmatrix}
                                 1  & \vline& 0 \\
                                 \hline \vphantom{\int_a^b}
                                 0  & \vline& 1 \\
                               \end{pmatrix}= 1\,,
\end{align*}
and with a ``ghost'' (wrong parity) column as the first argument,
\begin{align*}
    \pl(U)(e^{\hat 1}|e^{\hat 1})&= \Ber\begin{pmatrix}
                                 \x & \vline & \x \\
                                 \hline \vphantom{\int_a^b}
                                 y & \vline & y \\
                               \end{pmatrix}= 0\,,\\
    \pl(U)(e^{\hat 1}|e^{\hat 2})&= \Ber\begin{pmatrix}
                                 \x & \vline & 0 \\
                                 \hline \vphantom{\int_a^b}
                                 y & \vline & 1 \\
                               \end{pmatrix}= \x\,,\\
      \pl(U)(e^{\hat 2}|e^{\hat 1})&= \Ber\begin{pmatrix}
                                 0 & \vline & \x \\
                                 \hline \vphantom{\int_a^b}
                                 1 & \vline & y \\
                               \end{pmatrix}= -\frac{\x}{y^2}\,,\\
      \pl(U)(e^{\hat 2}|e^{\hat 2})&= \Ber\begin{pmatrix}
                                 0 & \vline & 0 \\
                                 \hline \vphantom{\int_a^b}
                                 1 & \vline & 1 \\
                               \end{pmatrix}= 0\,.
\end{align*}
Hence, in particular, we recover $x$ and $\x$ as
\begin{align}
    x&= \pl(U)(e^1|e^{\hat 2})\,,\\
    \x&=  \pl(U)(e^{\hat 1}|e^{\hat 2})\,.
\end{align}

To obtain similarly $y$ and $\h$, we shall use the inverse Berezinian function $\Ber^*g$   introduced in Sec.~\ref{subsec.extalgebra}. We define the \emph{$\Pi$-dual Pl\"ucker transform} $\pl^{\,*}(U)$ of a matrix $U$
as the function
\begin{equation}\label{eq.plueckinv}
    \pl^*(U)(P):=\Ber^*(UP)\,.
\end{equation}
It is now linear in the column $p^{\hat 1}$, so a ``ghost'', i.e. wrong parity, value can be substituted.
Then similarly to the above, we    have
\begin{align*}
    \pl^*(U)(e^1|e^{\hat 1})&= \Ber^*\begin{pmatrix}
                                 x & \vline & \x \\
                                 \hline \vphantom{\int_a^b}
                                 \h & \vline & y \\
                               \end{pmatrix}= \frac{y-\h x^{-1}\x}{x}\,,\\
    \pl^*(U)(e^1|e^{\hat 2})&= \Ber^*\begin{pmatrix}
                                 x & \vline & 0 \\
                                 \hline \vphantom{\int_a^b}
                                 \h & \vline & 1 \\
                               \end{pmatrix}= \frac{1}{x}\,,\\
      \pl^*(U)(e^2|e^{\hat 1})&= \Ber^*\begin{pmatrix}
                                 1 & \vline & \x \\
                                 \hline \vphantom{\int_a^b}
                                 0 & \vline & y \\
                               \end{pmatrix}= y\,,\\
      \pl^*(U)(e^2|e^{\hat 2})&= \Ber^*\begin{pmatrix}
                                 1 & \vline & 0 \\
                                 \hline \vphantom{\int_a^b}
                                 0 & \vline & 1 \\
                               \end{pmatrix}= 1\,,
\end{align*}
and also  with an even covector in the odd position:
\begin{align*}
    \pl^*(U)(e^1|e^{1})&= \Ber^*\begin{pmatrix}
                                 x & \vline & x \\
                                 \hline \vphantom{\int_a^b}
                                 \h & \vline & \h \\
                               \end{pmatrix}= 0\,,\\
    \pl^*(U)(e^1|e^{2})&= \Ber^*\begin{pmatrix}
                                 x & \vline & 1 \\
                                 \hline \vphantom{\int_a^b}
                                 \h & \vline & 0 \\
                               \end{pmatrix}= -\frac{\h}{x^2}\,,\\
      \pl^*(U)(e^2|e^{1})&= \Ber^*\begin{pmatrix}
                                 1 & \vline & x \\
                                 \hline \vphantom{\int_a^b}
                                 0 & \vline & \h \\
                               \end{pmatrix}= \h\,,\\
      \pl^*(U)(e^2|e^{2})&= \Ber^*\begin{pmatrix}
                                 1 & \vline & 1 \\
                                 \hline \vphantom{\int_a^b}
                                 0 & \vline & 0 \\
                               \end{pmatrix}= 0\,.
\end{align*}
Hence, in particular,
\begin{align}
    y&= \pl^*(U)(e^2|e^{\hat 1})\,,\\
    \h&=  \pl^*(U)(e^2|e^{1})\,.
\end{align}

This example demonstrates the following: (1) it is indeed possible to reconstruct the matrix $U$ provided we use \emph{both} functions $\pl(U)$ and $\pl^*(U)$; (2) reconstruction formulas require  \emph{extending} each of the functions $\pl(U)$ and $\pl^*(U)$ so to allow   ``ghost'' covectors as their arguments. We also see that there is redundancy in the data provided by the values of $\pl(U)$ and $\pl^*(U)$. (This redundancy will be the source of    ``super Pl\"ucker relations'' in this framework.)

\subsection{General formulation. Proof of embedding}
\label{subsec.proofemb}

Return to the geometric language. The above preliminary definition of the super Pl\"{u}cker map for the super Grassmannian $G_{r|s}(V)$ (given by formula~\eqref{eq.splueckprelim}) will be now amended as follows.
\begin{proposition}
There is a natural identification
\begin{equation}\label{eq.pigrass}
    G_{r|s}(V)= G_{s|r}(\Pi V)\,.
\end{equation}
\end{proposition}
\begin{proof}
To every $r|s$-plane $L\subset V$ corresponds an $s|r$-plane $\Pi L\subset \Pi V$, and conversely.
\end{proof}
In parallel with the map $\pl\co L\mapsto [\up_1,\ldots,\up_r\,|\,\up_{\hat 1},\ldots,\up_{\hat s}]\in \L^{r|s}(V)$, where we do not pass to  equivalence classes, we consider a map $\pl^*$, where $\pl^*(L)$ is defined as $\pl(\Pi L)\in \L^{s|r}(\Pi V)$ (also without passing to  equivalence classes). We shall consider $\pl(L)$ and $\pl^*(L)$ together. In the matrix language, for the matrix of homogeneous coordinates $U$ of a plane $L$, these are the functions $\pl(U)$ and $\pl^*(U)$ given by~\eqref{eq.plup} and~\eqref{eq.plueckinv}. As we have explained, when considered on the original domain of even  $n|m\times r|s$ matrices $P$, the   functions $\pl(U)$ and $\pl^*(U)$ are  related simply by
\begin{equation*}
    \pl^*(U)=\frac{1}{\pl(U)}\,.
\end{equation*}
However, we need to consider each of the functions $\pl(U)$ and $\pl^*(U)$   \emph{extended} by   multilinearity to larger domains  (with the possibility of ``ghost'' columns),   different for $\pl(U)$ and $\pl*(U)$.  Hence we need both functions. Note that while $\pl(gU)=\Ber g\cdot \pl(U)$, we have
\begin{equation*}
    \pl^*(gU)=(\Ber g)^{-1}\cdot \pl^*(U)\,.
\end{equation*}
With this in mind, we define finally
the \textbf{super Pl\"ucker map} for the super Grassmannian $G_{r|s}(V)$ as the map
\begin{equation}
\Pl\co G_{r|s}(V)\to P_{1,-1}\left(\L^{r|s}(V)\oplus \L^{s|r}(\Pi V)\right)\,, \quad
     L\mapsto  \cl(\pl(L), \pl^*(L))\,
\end{equation}
(with    capital $\textswab{P}$). Here $P_{1,-1}$ denotes the weighted projective (super)space with weights $+1$ and $-1$ respectively, for the corresponding direct summands. Namely, points of
\begin{equation*}
    P_{1,-1}\left(\L^{r|s}(V)\oplus \L^{s|r}(\Pi V)\right)
\end{equation*}
are equivalence classes of pairs $(F, G)$, where $F\in \L^{r|s}(V)$ and $G\in \L^{s|r}(\Pi V)$, both no-zero, and $(F, G)\sim (\la F, \la^{-1}G)$.

\begin{remark}
In the standard usage, weights for
weighted projective spaces  are     in $\mathbb{N}$. Standard weighted projective spaces   are spaces with singularities (if the weights are unequal), not smooth manifolds.  In our case, we have to consider an analog of weighted projective spaces with weights $1$ or $-1$.  The space $P_{1,-1}\left(\L^{r|s}(V)\oplus \L^{s|r}(\Pi V)\right)$ is itself infinite-dimensional, but below we shall construct its finite-dimensional version, which will also be a weighted projective superspace with weights $\pm 1$. We will see that in such a case, it will be a \emph{smooth}  (non-singular) supermanifold.
\end{remark}

\begin{theorem}\label{thm.embedgen}
The super Pl\"ucker map $\Pl$  gives an embedding of the  super   Grassmannian  $G_{r|s}(V)$ into $P_{1,-1}\left(\L^{r|s}(V)\oplus \L^{s|r}(\Pi V)\right)$.
\end{theorem}

Our initial version of the super Pl\"{u}cker map, $\pl\co G_{r|s}(V) \to P(\L^{r|s}(V))$, introduced in subsection~\ref{subsec.reformul} is the composition  $\pl=\pi_1\circ \Pl$ of $\Pl$  with the natural projection
\begin{equation}
    \pi_1\co P_{1,-1}\left(\L^{r|s}(V)\oplus \L^{s|r}(\Pi V)\right)\to P(\L^{r|s}(V))\,.
\end{equation}
(A similar composition $\pi_2\circ \Pl$ gives the map $\pl^*\co G_{r|s}(V) \to P(\L^{s|r}(\Pi V))$.)
While  insufficient in general, the map $\pl$ is sufficient in  the ``algebraic'' case $s=0$. 
The following statement completes the analysis   performed by different method in Section~\ref{sec.simpl}.

\begin{theorem}\label{thm.embedalg}
For $s=0$, the super Pl\"ucker map
\begin{equation*}
    \pl\co G_r(V)\to P(\L^{r}(V))
\end{equation*}
is an embedding.
\end{theorem}

We shall give proofs of Theorem~\ref{thm.embedgen} and Theorem~\ref{thm.embedalg} along similar lines. We shall evaluate the functions $\pl(U)$ and $\pl^*(U)$ at some combinations of basis covectors are show that this is sufficient to reconstruct $U$ up to equivalence, i.e. to reconstruct $L\in G_{r|s}(V)$. In other words, we will construct the inverse to $\Pl$ (and to $\pl$ for $s=0$) on the image.  (Later we shall come back to the second problem, of describing the image.)

Recall the construction of local coordinates (inhomogeneous) on the super Grassmannian given in subsection~\ref{subsec.recol}. Let $U$ as usual stand for the matrix of homogeneous coordinates on $G_{r|s}(V)$. In the open domain where a square $r|s\times r|s$ submatrix $U^{\aund}$ is invertible, the inhomogeneous coordinates in the $\aund$th chart are given by the matrix entries of $(U^{\aund})^{-1}U$. Here $\aund$ is a multi-index, $\aund=a_1,\ldots,a_r|\hat\mu_1,\ldots,\hat\mu_r$, giving the indices of $r$ even and $s$ odd columns in $U$. In order to express the entries of $(U^{\aund})^{-1}U$, we will use the  ``super Cramer  rule''.

Note that the familiar Cramer rule  for solving linear equations can be interpreted as the formula for the coordinates of the column-vector $c=A^{-1}b$,  for an  $r\times r$ invertible matrix $A$ and an arbitrary column-vector $b$:  $c_j={\det(A^1,\ldots,A^{j-1},b,A^{j+1},\ldots,A^r)}/{\det A}$,
i.e. at the top the column-vector $b$ is inserted in the $j$th position in the matrix $A$ replacing its $j$th column $A^j$. Due to the multilinearity of determinant, this  formula holds true regardless of the nature of the coordinates of $b$. In particular, $b$ can be an odd vector.
Now,  in the supercase, for an even invertible $r|s\times r|s$ matrix $A$ and a column-vector $b$ (of size $r|s$), the  \textbf{super Cramer rule}  gives   the coordinates of $c=A^{-1}b$ as
\begin{equation*}
    c_j=\begin{cases}
    \dfrac{\Ber(A^1,\ldots,A^{j-1},b,A^{j+1},\ldots,A^r|A^{\hat 1},\ldots,A^{\hat s})}{\Ber A} \quad  \text{for $j=1\ldots,r$}\,,
    \vspace{1em}\\
    \dfrac{\Ber^*(A^1,\ldots,A^r|A^{\hat 1},\ldots,A^{\hat j-1},b,A^{\hat j+1},\ldots,A^{\hat s})}{\Ber^* A}\quad  \text{for $j=\hat 1\ldots,\hat s$}\,.
    \end{cases}
\end{equation*}
Note that here the column-vector $b$ is inserted into an even position in the Berezinian $\Ber A$, 
which is multilinear in the even columns, hence such a substitution is valid regardless of the parity of $b$. Likewise, $b$ is inserted into an odd position, 
into the inverse Berezinian $\Ber^* A$, which is multilinear in the odd columns. Therefore this formula for $A^{-1}b$ applies for an even column-vector $b$ and returns an even column-vector and for an odd column-vector $b$ and returns an odd column-vector. (The super Cramer  rule was  obtained by Bergvelt and Rabin~\cite{rabin:duke}   using Gelfand--Retakh's quasideterminants~\cite{gelfand:retakh91,gelfand:retakh97}  and    by  a different approach in~\cite{tv:ber}.   See also~\cite{tv:superdif}.)

With this in mind, we can give the following formulas for the inhomogeneous coordinates on the super Grassmannian $G_{r|s}(V)$, $\dim V=n|m$, in the $\aund$th chart (where the submatrix $U^{\aund}$ is invertible)\,:
\begin{align}\label{eq.wjb}
    w_j{}^b&=\frac{u^{a_1\ldots a_{j-1}ba_{j+1}\ldots a_r|\hat\mu_1\ldots\hat \mu_s}}{u^{a_1\ldots a_r|\hat\mu_1\ldots\hat \mu_s}}\,, \quad j=1,\ldots,r\,, \ b=1,\ldots,n\,,\vspace{2em}\\
    \label{eq.wjhnu}
    w_j{}^{\hnu}&=\frac{u^{a_1\ldots a_{j-1}\hnu a_{j+1}\ldots a_r|\hat\mu_1\ldots\hat \mu_s}}{u^{a_1\ldots a_r|\hat\mu_1\ldots\hat \mu_s}}\,, \quad j=1,\ldots,r\,, \ \nu=1,\ldots,m\,,\vspace{2em}\\
    \label{eq.whbb}
    w_{\hb}{}^b&=\frac{u^{*a_1\ldots a_r|\hmu_1\ldots\hmu_{\beta-1}b\hmu_{\beta+1}\ldots\hmu_s}}{u^{*a_1\ldots a_r|\hmu_1\ldots\hmu_s}}\,, \quad \be=1,\ldots,s\,, \ b=1,\ldots,n\,,\vspace{2em}\\
    \label{eq.whbhnu}
    w_{\hb}{}^{\hnu}&=\frac{u^{*a_1\ldots a_r|\hmu_1\ldots\hmu_{\beta-1}\hnu\hmu_{\beta+1}\ldots\hmu_s}}{u^{*a_1\ldots a_r|\hmu_1\ldots\hmu_s}}\,, \quad \be=1,\ldots,s\,, \ \nu=1,\ldots,m\,.
\end{align}
Here $w_j{}^b$ etc. are the matrix entries of the matrix $W=(U^{\aund})^{-1}U$ and we have introduced  
notation for the minors and ``wrong'' minors (with one ``ghost'' column) of the matrix $U$\,:
\begin{equation}\label{eq.ua1mus}
    u^{a_1\ldots a_r|\hmu_1\ldots\hmu_s}:=\Ber U^{a_1\ldots a_r|\hmu_1\ldots\hmu_s}=\pl(U)(e^{a_1},\ldots,e^{a_r}|e^{\hmu_1},\ldots,e^{\hmu_s})\,,
\end{equation}
\begin{equation}\label{eq.usta1mus}
    u^{*a_1\ldots a_r|\hmu_1\ldots\hmu_s}:=\Ber^* U^{a_1\ldots a_r|\hmu_1\ldots\hmu_s}
    =\pl^*(U)(e^{a_1},\ldots,e^{a_r}|e^{\hmu_1},\ldots,e^{\hmu_s})\,,
\end{equation}
\begin{multline}\label{eq.ua1numus}
    u^{a_1\ldots a_{j-1}\hnu a_{j+1}\ldots a_r|\hmu_1\ldots\hmu_s}:=\Ber U^{a_1\ldots a_{j-1}\hnu a_{j+1}\ldots a_r|\hmu_1\ldots\hmu_s}\\
    =\pl(U)(e^{a_1},\ldots,e^{a_{j-1}},e^{\hnu}, e^{a_{j+1}},\ldots, e^{a_r}|e^{\hmu_1},\ldots,e^{\hmu_s})\,,
\end{multline}
\begin{multline}\label{eq.usta1bmus}
    u^{*a_1\ldots a_r|\hmu_1\ldots \hmu_{\be-1}b \hmu_{\be+1}\ldots \hmu_s}:=\Ber^* U^{a_1\ldots a_r|\hmu_1\ldots \hmu_{\be-1}b \hmu_{\be+1}\ldots \hmu_s}\\
    =\pl^*(U)(e^{a_1},\ldots,e^{a_r}|e^{\hmu_1},\ldots,e^{\hmu_{\be-1}},e^b,e^{\hmu_{\be+1}},\ldots,e^{\hmu_s})\,.
\end{multline}
Here $U^{a_1\ldots a_r|\hmu_1\ldots\hmu_s}$ is the   submatrix of the matrix $U$ obtained by choosing $r$ even and $s$ odd columns with the indicated indices. It is an even square matrix. Its Berezinian (or inverse Berezinian) therefore coincides with the evaluation of the Pl\"{u}cker transform (or the $\Pi$-dual Pl\"{u}cker transform) of $U$ at the basis covectors with the corresponding indices. 
At the same time, $U^{a_1\ldots a_{j-1}\hnu a_{j+1}\ldots a_r|\hmu_1\ldots\hmu_s}$ and $U^{a_1\ldots a_r|\hmu_1\ldots \hmu_{\be-1}b \hmu_{\be+1}\ldots \hmu_s}$ are ``wrong'' square matrices (neither even nor odd) obtained from the even matrix $U^{a_1\ldots a_r|\hmu_1\ldots\hmu_s}$ by replacing   its column  in the $j$th (even) position or the $\hat\be$th (odd) position, respectively, by a wrong parity column, the $\hnu$th column (odd) or the $b$th column (even) of the matrix $U$. By the properties of $\Ber$ and $\Ber^*$, the Berezinian (or the inverse Berezinian, respectively) of such a matrix makes sense and gives an odd value. (See also discussion in subsection~\ref{subsec.11in22}.)

Now we shall introduce \emph{auxiliary supermanifolds} denoted by $\pfin$ and $\pess$.

By definition, the supermanifold $\pfin$ is the weighted projective superspace with homogeneous coordinates 
\begin{equation}\label{eq.coorfinrs}
    \underbrace{u^{a_1\ldots a_r|\hmu_1\ldots\hmu_s}\ \text{(even)}\,, u^{a_1\ldots a_{r-1}\hnu|\hmu_1\ldots\hmu_s}\ \text{(odd)}}_{\text{weight $+1$}}\,,
    \underbrace{u^{*a_1\ldots a_r|\hmu_1\ldots\hmu_s}\ \text{(even)}\,, u^{*a_1\ldots a_{r}|b\hmu_1\ldots\hmu_{s-1}}\ \text{(odd)}}_{\text{weight $-1$}}\,.
\end{equation}
Here the symbols $u^{a_1\ldots a_r|\hmu_1\ldots\hmu_s}$, $u^{a_1\ldots a_{r-1}\hnu|\hmu_1\ldots\hmu_s}$, $u^{*a_1\ldots a_r|\hmu_1\ldots\hmu_s}$ and $u^{*a_1\ldots a_{r}|b\hmu_1\ldots\hmu_{s-1}}$ are initially defined for $a_1< \ldots<a_r$ and $\hmu_1<\ldots<\hmu_s$, while $\hnu\notin \{\hmu_1,\ldots,\hmu_s\}$, $b\notin \{a_1,\ldots,a_r\}$.

Likewise, the supermanifold $\pess$ is by definition the   projective superspace with homogeneous coordinates
\begin{equation}
    u^{a_1\ldots a_r}\ \text{(even)}\,, u^{a_1\ldots a_{r-1}\hmu}\ \text{(odd)}\,.
\end{equation}
Here $a_1< \ldots<a_r$.

For   convenience of notation, we extend the definition of $u^{a_1\ldots a_r|\hmu_1\ldots\hmu_s}$ and $u^{a_1\ldots \hnu \ldots a_{r-1}|\hmu_1\ldots\hmu_s}$ to arbitrary combinations of indices $a_1,\ldots,a_r,\hmu_1,\ldots,\hmu_s,\hnu$ and   arbitrary position of the index $\hnu$ by  the condition   of antisymmetry in the indices in the first $r$ even positions and in the last $s$ odd positions. Also,   the indices $\hmu_1,\ldots,\hmu_s$ must be all different, while some indices $a_1,\ldots,a_r$ may be coinciding (in which case the corresponding variables will be zero by the antisymmetry), and if $\hnu\in \{\hmu_1,\ldots,\hmu_s\}$, then
$u^{a_1\ldots \hnu \ldots a_{r-1}|\hmu_1\ldots\hmu_s}=0$. Similarly for $u^{*a_1\ldots a_r|\hmu_1\ldots\hmu_s}$ and $u^{*a_1\ldots a_{r}|\hmu_1\ldots b \ldots \hmu_{s-1}}$, i.e. antisymmetry inside each group of indices, the indices in the even positions must be distinct (but may coincide in the odd positions,   the corresponding variables equal to zero), and $u^{*a_1\ldots a_{r}|\hmu_1\ldots b \ldots \hmu_{s-1}}=0$ if $b\in \{a_1,\ldots,a_r\}$.

Also, the definition of the variables $u^{a_1\ldots a_r}$ and $u^{a_1\ldots a_{r-1}\hmu}$    for convenience of notation   is extended to arbitrary combinations of indices $a_1,\ldots,a_r$ and to arbitrary position of $\hmu$ by the condition of the antisymmetry in all indices.

\begin{proposition}
These formal properties of the variables $u^{a_1\ldots a_r|\hmu_1\ldots\hmu_s}$, etc., follow the properties of the corresponding minors (and ``wrong'' minors) of the matrix $U$.
\end{proposition}
\begin{proof} By the properties of Berezinian and determinant.
\end{proof}

We call the variables $u^{a_1\ldots a_r|\hmu_1\ldots\hmu_s}$, $u^{a_1\ldots \hnu \ldots a_{r-1}|\hmu_1\ldots\hmu_s}$, $u^{*a_1\ldots a_r|\hmu_1\ldots\hmu_s}$ and $u^{*a_1\ldots a_{r}|\hmu_1\ldots b \ldots\hmu_{s-1}}$  the  \textbf{{essential  super  Pl\"{u}cker coordinates}} for the super Grassmannian $G_{r|s}(V)$  and the variables  $u^{a_1\ldots a_r}$ and  $u^{a_1\ldots \hmu \ldots a_{r-1}}$ the \textbf{{essential super  Pl\"{u}cker coordinates}} for the super Grassmannian $G_{r}(V)$.

Define the mappings (which we denote by the same symbols $\Pl$ and $\pl$)
\begin{equation}\label{eq.splueckfin}
    \Pl\co G_{r|s}(V)\to \pfin\,,
\end{equation}
given by formulas~\eqref{eq.ua1mus}, \eqref{eq.usta1mus}, \eqref{eq.ua1numus} and \eqref{eq.usta1bmus}, and
\begin{equation}\label{eq.splueckess}
    \pl\co G_{r}(V)\to \pess\,,
\end{equation}
given by (the similar) formulas $u^{a_1\ldots a_r}=\det U^{a_1\ldots a_r}$ and $u^{a_1\ldots  \hnu  \ldots a_r} =\det U^{a_1\ldots  \hnu  \ldots a_r}$.

\begin{proof}[Proof of Theorem~\ref{thm.embedgen}]
There is a commutative diagram
\begin{equation}\label{eq.cdpfin}
    \begin{tikzcd}
    G_{r|s}(V) \arrow{r}{\Pl} \arrow[swap]{dr}{\Pl} & P_{1,-1}\left(\L^{r|s}(V)\oplus \L^{s|r}(\Pi V)\right) \arrow{d}{ } \\
     & \pfin
  \end{tikzcd}
\end{equation}
Thus it suffices to show that the map~\eqref{eq.splueckfin} has an inverse on its image. We can define a ``local'' inverse by  formulas~\eqref{eq.wjb}, \eqref{eq.wjhnu}, \eqref{eq.whbb}, \eqref{eq.whbhnu} in the domain of $\pfin$ where $u^{a_1\ldots a_r|\hmu_1\ldots\hmu_s}$ and $u^{*a_1\ldots a_r|\hmu_1\ldots\hmu_s}$ are invertible. The image of the $\aund$th  coordinate domain of $G_{r|s}(V)$ is contained in it (where $\aund={a_1\ldots a_r|\hmu_1\ldots\hmu_s}$). Note that the matrix $W$ given by these formulas does not equal  $U$, but $W\sim U$. On the intersection of any two such domains, we obtain $W\sim U$ and $W'\sim U$. Hence these ``local'' inverses glue into a global inverse map  $\Pl^{-1}$ from the image of $\Pl$ to the super Grassmannian.
\end{proof}

\begin{proof}[Proof of Theorem~\ref{thm.embedalg}]
Similarly with the previous proof,  consider a
commutative diagram
\begin{equation}\label{eq.cdpess}
    \begin{tikzcd}
    G_{r}(V) \arrow{r}{\pl} \arrow[swap]{dr}{\pl} & P\left(\L^{r}(V)\right) \arrow{d}{ } \\
     & \pess
  \end{tikzcd}
\end{equation}
``Local inverses'' to the map $\pl\co G_{r}(V)\to \pess$ defined on its image are given by the formulas
\begin{align*}
    w_j{}^b&=\frac{u^{a_1\ldots a_{j-1}ba_{j+1}\ldots a_r}}{u^{a_1\ldots a_r}}\,, \quad j=1,\ldots,r\,, \ b=1,\ldots,n\,,\vspace{2em}\\
    \label{eq.wjhnu}
    w_j{}^{\hnu}&=\frac{u^{a_1\ldots a_{j-1}\hnu a_{j+1}\ldots a_r}}{u^{a_1\ldots a_r}}\,, \quad j=1,\ldots,r\,, \ \nu=1,\ldots,m\,,
\end{align*}
(in the domain where $u^{a_1\ldots a_r}$ is invertible)
which are a particular case of~\eqref{eq.wjb}\,--\,\eqref{eq.whbhnu}. By the same argument, they glue together into a single inverse map for $\pl$ from the image of  $\pl\co G_{r}(V)\to \pess$ to the super Grassmannian $G_{r}(V)$. Hence this map is an embedding and the map $\pl\co G_{r}(V)\to P\left(\L^{r}(V)\right)$ is also an embedding.
\end{proof}

\begin{remark}
The vertical arrows in the commutative diagrams in~\eqref{eq.cdpfin} and~\eqref{eq.cdpess} are rational maps. The domain of definition of the projection $P\left(\L^{r}(V)\right)\to \pess$ is exactly the projectivization of the subset of non-degenerate $r$-vectors considered in Section~\ref{sec.simpl}.
\end{remark}

\begin{proposition}
In the algebraic case $s=0$, the variables $u^{a_1\ldots a_r}$ and $u^{a_1\ldots  a_{r-1}\hmu}$  evaluated on $G_r(V)$ coincide  up to the factor of $r!$  with the corresponding components $T^{a_1\ldots a_r}$ and $T^{a_1\ldots  a_{r-1}\hmu}$ of the $r$-vector $T=\up_1\wed\ldots\wed \up_r$ considered in the previous section.
\end{proposition}
\begin{proof} Directly.
\end{proof}

The multivector $T$ has other components, with more than one odd index; however, they are not needed for the inversion of the super Pl\"{u}cker map.  In subsection~\ref{subsec.ex}, we also showed that they can be excluded algebraically from the full set of the super Pl\"{u}cker relations.  This justifies the   terminology ``essential super Pl\"{u}cker coordinates''.

\begin{remark}
In the general case, the situation is more delicate. The  essential super Pl\"{u}cker coordinates  $u^{a_1\ldots a_r|\hmu_1\ldots\hmu_s}$, $u^{a_1\ldots a_{r-1}\hnu|\hmu_1\ldots\hmu_s}$, $u^{*a_1\ldots a_r|\hmu_1\ldots\hmu_s}$ and $u^{*a_1\ldots a_{r}|b\hmu_1\ldots\hmu_{s-1}}$ are not    ``components''  of the $r|s$-vector $F=[\up_1\ldots,\up_r|\up_{\hat 1},\ldots,\up_{\hat s}]$ in the same sense as $u^{a_1\ldots a_r}$ and $u^{a_1\ldots  a_{r-1}\hmu}$   for $s=0$. But as the evaluations of $F$ or $F^*=1/F$ at particular combinations of basis covectors, they make $\pfin$  a  finite-dimensional quotient  of the space $P_{1,-1}\left(\L^{r|s}(V)\oplus \L^{s|r}(\Pi V)\right)$. (Strictly speaking  $\pess$ and $\pfin$ are   quotients of certain dense open domains.)
\end{remark}

\subsection{Super Pl\"{u}cker relations for essential super Pl\"{u}cker coordinates} \label{subsec.genrelat} We want now to find relations (for the essential super Pl\"{u}cker coordinates)
that  specify the image of $G_{r|s}(V)$ in $\pfin$ and   the image of $G_{r}(V)$ in $\pess$. In the latter case, it will be interesting to compare with what we have obtained in Section~\ref{sec.simpl} by the algebraic method. Our strategy is based on the following simple statement about maps.

\begin{proposition}
Let $f\co A\to B$ and $g\co B\to A$ be maps of sets in the opposite directions.  Let $g$ restricted to $\Im f\subset B$ be the inverse for  the corestriction of  $f$, i.e., $g\circ f=\id_A$. Then the equation $(f\circ g)(b)=b$ for the unknown $b\in B$ specifies $\Im f$ as a subset of $B$.
\end{proposition}
\begin{proof} Directly.
\end{proof}
We shall apply this argument to the map
\begin{equation*}
    \Pl\co G_{r|s}\to \pfin
\end{equation*}
given by  formulas~\eqref{eq.ua1mus}\,--\,\eqref{eq.usta1bmus} and the map that we used before without any notation and that we shall now denote $\Pl^{-1}_{(\aund)}$,
\begin{equation*}
    \Pl^{-1}_{(\aund)}\co \pfin\to G_{r|s}(V)\,,
\end{equation*}
for some fixed $\aund=a_1\ldots a_r|\hmu_1\ldots\hmu_s$,
defined on the domain where the variables $u^{a_1\ldots a_r|\hmu_1\ldots\hmu_s}$ and $u^{*a_1\ldots a_r|\hmu_1\ldots\hmu_s}$  are invertible  and given by~\eqref{eq.wjb}\,--\,\eqref{eq.whbhnu}. Then $\Pl^{-1}_{(\aund)}\circ \Pl=\id_{G_{r|s}(V)}$ by the super Cramer rule:  in the matrix language, the composition gives $U\mapsto  (U^{\aund})^{-1}U\sim U$. This is what we have used in the previous subsection. But the composition in the opposite order, $\Pl\circ \Pl^{-1}_{(\aund)}$ is not   identity on   $\pfin$. The requirement for it being identity gives equations specifying the image of $\Pl$ in $\pfin$. The same argument applies for the simpler case of $\pl\co G_r(V)\to P(\L^r(V))$.

To see how it practically works and leads to  the desired super  Pl\"{u}cker relations, consider first the familiar classical case.

\begin{example} Consider $G_2(V)$ for $\dim V=n$. There are Pl\"{u}cker coordinates $u^{ab}$, $a,b=1,\ldots,n$. We consider them as homogeneous coordinates on the projective space $P(\L^2(V))$. Fix some pair $\aund=ab, a<b$. In the domain where $u^{ab}\neq 0$, we have the ``inverse Pl\"{u}cker map'' $\pl_{(\aund)}^{-1}$ that sends the collection of variables $u^{cd}$, $c,d=1,\ldots,n$, to a $2\times n$ matrix $U'$ with the entries $u'_1{}^c=u^{cb}/u^{ab}$ and $u'_2{}^c=u^{ac}/u^{ab}$. (Here we use $U'$ instead of $W$.) The requirement that $\pl\circ \pl^{-1}_{(\aund)}=\id$ is equivalent to ${u'}{}^{cd}=\la\, u^{cd}$ for all $c,d=1,\ldots,n$ with some factor $\la$ independent of $c$ and $d$. Here ${u'}{}^{cd}=\det U'{}^{cd}$. Explicitly:
\begin{equation*}
    \begin{vmatrix}
      u^{cb}/u^{ab} & u^{db}/u^{ab} \\
      u^{ac}/u^{ab} & u^{ad}/u^{ab} \\
    \end{vmatrix}=\la u^{cd}
\end{equation*}
for all $c,d$. When $cd=ab$, this gives $1=\la\,u^{ab}$. So $\la=(u^{ab})^{-1}$ and we obtain
\begin{equation*}
   (u^{ab})^{-2} \begin{vmatrix}
      u^{cb}  & u^{db}  \\
      u^{ac}  & u^{ad}  \\
    \end{vmatrix}=(u^{ab})^{-1} u^{cd}
\end{equation*}
or, by getting rid of the denominator,
\begin{equation*}
  \begin{vmatrix}
      u^{cb}  & u^{db}  \\
      u^{ac}  & u^{ad}  \\
    \end{vmatrix}= u^{ab}u^{cd}\,,
\end{equation*}
which is exactly the classical three-term Pl\"{u}cker relation
\begin{equation*}
    u^{ab}u^{cd}=u^{ac}u^{bd}+u^{ad}u^{cb}\,.
\end{equation*}
Note that we initially worked in the domain where $u^{ab}$ is invertible, but the final form of the relation does not depend on this assumption, and it has the same form in all such domains, hence   specifies the image of the Pl\"{u}cker map on the whole $P(\L^2(V))$.
\end{example}

For $r>2$, we shall see that an additional argument is needed in order to obtain the classical form of the Pl\"{u}cker relations:

\begin{example}
\label{ex.essrinn}
Consider now $G_r(V)$ for $\dim V=n$ when $r>2$. By arguing as above, we have the direct Pl\"{u}cker map $\pl$, which sends (the equivalence class of) an $r\times n$ matrix $U$ to (the equivalence class of) an array of homogeneous coordinates $u^{\aund}$, $\aund=a_1\ldots a_r$, by the formula $u^{\aund}=\det U^{\aund}$, for all $\aund$, and the (locally defined) inverse Pl\"{u}cker map $\pl_{(\aund)}^{-1}$ defined in the open domain where the variable $u^{\aund}$ is invertible, for a fixed chosen ${\aund}$, which sends the variables $u^{\bund}$, for all $\bund=b_1\ldots b_r$, to (the equivalence class of) a matrix $U'$, whose matrix entries are
$u'_i{}^b=u^{a_1\ldots a_{a_{i-1}ba_{i+1}}\ldots a_r}/u^{a_1\ldots a_r}$, where $i=1,\ldots,r$ and $b=1,\ldots,n$. The   requirement that $\pl\circ \pl^{-1}_{(\aund)}=\id$ given by the relation
\begin{equation*}
    (u^{a_1\ldots a_r})^{-r}\begin{vmatrix}
                              u^{b_1a_2\ldots a_r} & u^{b_2a_2\ldots a_r} & \ldots & u^{b_ra_2\ldots a_r}\\
                              u^{a_1b_1\ldots a_r} & u^{a_1b_2\ldots a_r} & \ldots & u^{a_1b_r\ldots a_r}\\
                              \ldots               & \ldots               & \ldots & \ldots    \\
                              u^{a_1a_2\ldots b_1} & u^{a_1a_2\ldots b_2} & \ldots & u^{a_1a_2\ldots b_r}
                            \end{vmatrix}=
                            \la\,u^{b_1\ldots b_r}
\end{equation*}
for all $b_1\ldots b_r$, with some factor $\la$ to be determined. As above, by using $\bund=\aund$, we find $\la=(u^{a_1\ldots a_r})^{-1}$. By substituting it   and getting rid of the denominator, we arrive at the relations
\begin{equation*}
                \begin{vmatrix}
                              u^{b_1a_2\ldots a_r} & u^{b_2a_2\ldots a_r} & \ldots & u^{b_ra_2\ldots a_r}\\
                              u^{a_1b_1\ldots a_r} & u^{a_1b_2\ldots a_r} & \ldots & u^{a_1b_r\ldots a_r}\\
                              \ldots               & \ldots               & \ldots & \ldots    \\
                              u^{a_1a_2\ldots b_1} & u^{a_1a_2\ldots b_2} & \ldots & u^{a_1a_2\ldots b_r}
                            \end{vmatrix}=
                            (u^{a_1\ldots a_r})^{r-1}\,u^{b_1\ldots b_r}
\end{equation*}
for all $a_1\ldots a_r$ and $b_1\ldots b_r$, that specify  the image of $\pl$ in $P(\L^r(V))$. This is a set of polynomial equations of degree $r$, not of degree  $2$  as one would   expect from the exterior algebra arguments. 
However, still acting in the domain where the variable $u^{\aund}$ is invertible, we shall show now how to reduce the degree. Substitute in the above relation $b_1:=a_1$. Then in the first row of the determinant all elements but $u^{b_1a_2\ldots a_r}$ that will become $u^{a_1a_2\ldots a_r}$, will vanish, the determinant will factorize, and the relation will give
\begin{equation*}
    \cof(u^{b_1a_2\ldots a_r})=(u^{a_1\ldots a_r})^{r-2}\,u^{a_1b_2\ldots b_r}\,.
\end{equation*}
Acting in the similar way and substituting $b_2:=a_1$ into the relation, we obtain also
\begin{equation*}
    \cof(u^{b_2a_2\ldots a_r})=(u^{a_1\ldots a_r})^{r-2}\,u^{b_1a_1b_3\ldots b_r}\,.
\end{equation*}
And so on, for all $j=1,\ldots,r$, we obtain
\begin{equation*}
    \cof(u^{b_ja_2\ldots a_r})=(u^{a_1\ldots a_r})^{r-2}\,u^{b_1\ldots a_1\ldots b_r}\,.
\end{equation*}
with $a_1$ at the $j$th place at the right-hand side. Substituting these expressions for the cofactors of the elements of the first row in the determinant, we arrive at the equation
\begin{equation*}
    \sum_{j=1}^{r} u^{b_ja_2\ldots a_r}(u^{a_1\ldots a_r})^{r-2}u^{b_1\ldots b_{j-1}a_1b_{j+1}\ldots b_r}=(u^{a_1\ldots a_r})^{r-1}\,u^{b_1\ldots b_r}\,.
\end{equation*}
After reducing by  the common factor of $(u^{a_1\ldots a_r})^{r-2}$, this finally gives the desired  classical quadric Pl\"{u}cker relation, in the form
\begin{equation*}
    u^{a_1\ldots a_r}u^{b_1\ldots b_r}=\sum_{j=1}^{r} u^{b_ja_2\ldots a_r}u^{b_1\ldots b_{j-1}a_1b_{j+1}\ldots b_r}\,,
\end{equation*}
for all $a_1\ldots a_r$ and $b_1\ldots b_r$.
\end{example}

We shall now show that the idea used in Example~\ref{ex.essrinn} carries over to   $r$-planes in the superspace.

\begin{theorem}[``super Pl\"{u}cker relations for essential variables for $r$-planes in   $n|m$-space'']
\label{thm.essplueckr0}
On $\pessth$, the super Pl\"{u}cker relations take the form
\begin{align}\label{eq.esssplueckev}
    u^{a_1\ldots a_r}u^{b_1\ldots b_r}&=\sum_{j=1}^{r} u^{b_ja_2\ldots a_r}u^{b_1\ldots b_{j-1}a_1b_{j+1}\ldots b_r} &\ \text{\emph{(even)}}
    \\
    \label{eq.esssplueckod}
    u^{a_1\ldots a_r}u^{b_1\ldots b_{r-1}\hmu}&=\sum_{j=1}^{r-1} u^{b_ja_2\ldots a_r}u^{b_1\ldots b_{j-1}a_1b_{j+1}\ldots b_{r-1}\hmu}+ u^{\hmu a_2\ldots a_r}u^{b_1\ldots b_{r-1}a_1} &\ \text{\emph{(odd)}}
\end{align}
for all $a_1,\ldots,a_r$, $b_1,\ldots,b_r$ in~\eqref{eq.esssplueckev} and all  $a_1,\ldots,a_r$, $b_1,\ldots,b_{r-1}$ and all $\hmu$  in~\eqref{eq.esssplueckod}.
\end{theorem}
\begin{proof}
We can write an $r\times n|m$ matrix $U$ as $U=(u_i{}^a|u_i{}^{\hmu})$, where $i=1,\ldots,r$, $=1,\ldots,n$, and $\mu=1,\ldots, m$. The direct Pl\"{u}cker map $\pl$ sends it to   homogeneous Pl\"{u}cker coordinates $u^{a_1\ldots a_r}$ (even), $u^{a_1\ldots a_{r-1}\hmu}$ (odd) by $u^{a_1\ldots a_r}=\det U^{a_1\ldots a_r}=\det (U^{a_1}\ldots U^{a_r})$ and
$u^{a_1\ldots a_{r-1}\hmu}= \det (U^{a_1}\ldots U^{a_r-1} U^{\hmu})$. The locally defined ``inverse Pl\"{u}cker map'' $\pl_{(\aund)}^{-1}$ defined in the domain where $u^{a_1\ldots a_r}$ is invertible, for some  fixed $a_1<\ldots <a_r$, maps a point  of $\pess$ specified by homogeneous Pl\"{u}cker coordinates $u^{a_1\ldots a_r}, u^{a_1\ldots a_{r-1}\hmu}$ to the matrix $U'$ with the matrix entries  $u'_i{}^b=u^{a_1\ldots a_{a_{i-1}ba_{i+1}}\ldots a_r}/u^{a_1\ldots a_r}$ and $u'_i{}^{\hmu}=u^{a_1\ldots a_{a_{i-1}\hmu a_{i+1}}\ldots a_r}/u^{a_1\ldots a_r}$. The composition $\pl_{(\aund)}^{-1}\circ \pl$ sends $U$ to $U'=(U^{\aund})^{-1}U\sim U$, so is the identity on  the corresponding open domain of  the super Grassmannian  $G_r(V)$. At the same time, the requirement for the composition in the opposite order $\pl\circ \pl_{(\aund)}^{-1}$ to be the identity map specifies the points of $\pess$ that are in the image of the super Pl\"{u}cker map. We get the conditions:
\begin{equation*}
    (u^{a_1\ldots a_r})^{-r}\begin{vmatrix}
                              u^{b_1a_2\ldots a_r} & u^{b_2a_2\ldots a_r} & \ldots & u^{b_ra_2\ldots a_r}\\
                              u^{a_1b_1\ldots a_r} & u^{a_1b_2\ldots a_r} & \ldots & u^{a_1b_r\ldots a_r}\\
                              \ldots               & \ldots               & \ldots & \ldots    \\
                              u^{a_1a_2\ldots b_1} & u^{a_1a_2\ldots b_2} & \ldots & u^{a_1a_2\ldots b_r}
                            \end{vmatrix}=
                            \la\,u^{b_1\ldots b_r}
\end{equation*}
and
\begin{equation*}
    (u^{a_1\ldots a_r})^{-r}
    \begin{vmatrix}
        u^{b_1a_2\ldots a_r} & u^{b_2a_2\ldots a_r} & \ldots & u^{b_{r-1}a_2\ldots a_r} & u^{\hmu a_2\ldots a_r}\\
        u^{a_1b_1\ldots a_r} & u^{a_1b_2\ldots a_r} & \ldots & u^{a_1b_{r-1}\ldots a_r}& u^{a_1\hmu \ldots a_r}\\
        \ldots               & \ldots               & \ldots & \ldots    \\
        u^{a_1a_2\ldots b_1} & u^{a_1a_2\ldots b_2} & \ldots & u^{a_1a_2\ldots b_{r-1}}& u^{a_1a_2\ldots \hmu}
    \end{vmatrix}=
    \la\,u^{b_1\ldots b_{r-1}\hmu}\,,
\end{equation*}
for all $b_1\ldots b_r$ and $\hmu$, with a universal factor $\la$ to be determined. Substituting $b_1\ldots b_r=a_1\ldots a_r$  to the first equality  gives, as above,   $\la =(u^{a_1\ldots a_r})^{-1}$. This gives relations in the form of the two equalities,
\begin{equation*}
    \begin{vmatrix}
    u^{b_1a_2\ldots a_r} & u^{b_2a_2\ldots a_r} & \ldots & u^{b_ra_2\ldots a_r}\\
                              u^{a_1b_1\ldots a_r} & u^{a_1b_2\ldots a_r} & \ldots & u^{a_1b_r\ldots a_r}\\
                              \ldots               & \ldots               & \ldots & \ldots    \\
                              u^{a_1a_2\ldots b_1} & u^{a_1a_2\ldots b_2} & \ldots & u^{a_1a_2\ldots b_r}
                            \end{vmatrix}=
                            (u^{a_1\ldots a_r})^{r-1}\,u^{b_1\ldots b_r}
\end{equation*}
and
\begin{equation*}
    \begin{vmatrix}
        u^{b_1a_2\ldots a_r} & u^{b_2a_2\ldots a_r} & \ldots & u^{b_{r-1}a_2\ldots a_r} & u^{\hmu a_2\ldots a_r}\\
        u^{a_1b_1\ldots a_r} & u^{a_1b_2\ldots a_r} & \ldots & u^{a_1b_{r-1}\ldots a_r}& u^{a_1\hmu \ldots a_r}\\
        \ldots               & \ldots               & \ldots & \ldots    \\
        u^{a_1a_2\ldots b_1} & u^{a_1a_2\ldots b_2} & \ldots & u^{a_1a_2\ldots b_{r-1}}& u^{a_1a_2\ldots \hmu}
    \end{vmatrix}=
    (u^{a_1\ldots a_r})^{r-1}\,u^{b_1\ldots b_{r-1}\hmu}\,.
\end{equation*}
As in the purely even example, we can reduce the degree by expressing the cofactors of the elements in the first row of each determinant from the same identities and reducing by the common factor. In order to do so, we apply  the first identity with the substitutions $b_{\ell} :=a_1$, for   $\ell=1\ldots r$, and similarly for the second identity, for   $\ell=1\ldots r-1$. This gives as the cofactors, respectively, $(u^{a_1\ldots a_r})^{r-2}\,u^{b_1\ldots b_{\ell-1}a_1b_{\ell+1}\ldots b_r}$, $\ell=1\ldots r$ (for the first  determinant) and $(u^{a_1\ldots a_r})^{r-2}\,u^{b_1\ldots b_{\ell-1}a_1b_{\ell+1}\ldots b_{r-1}\hmu}$, $\ell=1\ldots r-1$, and $(u^{a_1\ldots a_r})^{r-2}\,u^{b_1\ldots b_{r-1}a_1}$ (for the second  determinant). Expanding the determinants in the first row and reducing by the common factor of $(u^{a_1\ldots a_r})^{r-2}$ gives~\eqref{eq.esssplueckev} and \eqref{eq.esssplueckod}.
\end{proof}

We can compare Theorem~\ref{thm.essplueckr0} with Theorem~\ref{thm.reducplueck} from Section~\ref{sec.simpl}  for $r=2$.

In Theorem~\ref{thm.reducplueck}, we derived the condition of non-degeneracy for a bivector $T$ in the image of the super Pl\"{u}cker map from    algebraic relations involving all its components of $T$ and showed that it is possible to reduce the set of components regarded as super Pl\"{u}cker coordinates to the subset corresponding to what we called in this section   essential super Pl\"{u}cker coordinates  and accordingly to reduce the full set of relations  obtained by the algebraic method to a subset satisfied by these essential coordinates.

Here, we work only with the essential super Pl\"{u}cker coordinates from the start  (under the assumption of the non-degeneracy),  and acting by a different (actually, more direct) method, deduce that the super Pl\"{u}cker relations~\eqref{eq.esssplueckev} and \eqref{eq.esssplueckod} satisfied by essential super Pl\"{u}cker coordinates are the necessary and sufficient conditions specifying the image of the super Pl\"{u}cker map for $G_r(V)$.
In this way, this generalizes  Theorem~\ref{thm.reducplueck},  from $2$-planes  to, now, arbitrary $r$-planes in $n|m$-space. (It is worth noting that in the ordinary (not super) case  all Pl\"{u}cker coordinates are essential, so there is no need for a distinction and this notion does not arise.\footnote{Our direct method possibly gives also a new proof   for the classical Pl\"{u}cker relations;  except for $r=2$,  we could not find  such a proof in the literature.})

In the general case of $r|s$-planes,  relations for the essential super Pl\"{u}cker coordinates take a more complicated form in comparison with the case of $r$-planes. Consider first the example of $1|1$-planes. Recall that homogeneous coordinates on $\pfinoo$ are 
\begin{equation}\label{eq.coorfin11}
    \underbrace{u^{a|\hmu}\ \text{(even)}\,, u^{\hnu|\hmu}\ \text{(odd)}}_{\text{weight $+1$}}\,,
    \underbrace{u^{*a|\hmu}\ \text{(even)}\,, u^{*a|b}\ \text{(odd)}}_{\text{weight $-1$}}\,.
\end{equation}

\begin{theorem}[``super Pl\"{u}cker relations for  $1|1$-planes in   $n|m$-space'']\label{thm.splueck11}
The super Pl\"{u}cker relations for the variables~\eqref{eq.coorfin11} are:
\begin{align}\label{eq.sp0}
    u^{a|\hmu}u^{*a|\hmu}&=1\,,\\
    u^{a|\hmu}u^{b|\hnu} &=\Ber\begin{pmatrix}
                                 u^{b|\hmu} & \vline & u^{\hnu|\hmu} \\
                                 \hline \vphantom{\int_a^b}
                                 u^{*a|b} & \vline & u^{*a|\hnu} \\
                               \end{pmatrix}\,,
                               \label{eq.sp1}\\
    u^{a|\hmu}u^{\hla|\hnu} &=\Ber\begin{pmatrix}
                                 u^{\hla|\hmu} & \vline & u^{\hnu|\hmu} \\
                                 \hline \vphantom{\int_a^b}
                                 u^{*a|\hla} & \vline & u^{*a|\hnu} \\
                               \end{pmatrix}\,,
                               \label{eq.sp2}\\
                               \label{eq.sp3}
    u^{*a|\hmu}u^{*b|c} &=\Ber^*\begin{pmatrix}
                                 u^{b|\hmu} & \vline & u^{c|\hmu} \\
                                 \hline \vphantom{\int_a^b}
                                 u^{*a|b} & \vline & u^{*a|c} \\
                               \end{pmatrix}\,.
\end{align}
\end{theorem}
Relations~\eqref{eq.sp0}, \eqref{eq.sp1} are even, and relations~\eqref{eq.sp2}, \eqref{eq.sp3} are odd.
\begin{proof}
We have the matrix of homogeneous coordinates on the super Grassmannian $G_{1|1}(V)$
\begin{equation*}
    U=\begin{pmatrix}
        u_1{}^a & \vline & u_1{}^{\hmu} \\
                                 \hline \vphantom{\int_a^b}
        u_{\hat 1}{}^a & \vline & u_{\hat 1}{}^{\hmu}
      \end{pmatrix}=
      \begin{pmatrix}
        u_1{}^1 & \ldots &u_1{}^n & \vline & u_1{}^{\hat 1} & \ldots & u_1{}^{\hat m} \\
                                 \hline \vphantom{\int_a^b}
        u_{\hat 1}{}^1 & \ldots &u_{\hat 1}{}^n & \vline & u_{\hat 1}{}^{\hat 1} & \ldots & u_{\hat 1}{}^{\hat m}
      \end{pmatrix}\,,
\end{equation*}
where $a=1,\ldots,n$ and $\mu=1,\ldots,m$, and the essential super Pl\"{u}cker coordinates $u^{a|\hmu}$, $u^{\hnu|\hmu}$, $u^{*a|\hmu}$, and $u^{*a|b}$. The direct super Pl\"{u}cker map $\Pl$ sends $U$ to the collection
\begin{align*}
    u^{a|\hmu}&=\Ber\begin{pmatrix}
        u_1{}^a & \vline & u_1{}^{\hmu} \\
                                 \hline \vphantom{\int_a^b}
        u_{\hat 1}{}^a & \vline & u_{\hat 1}{}^{\hmu}
      \end{pmatrix}\,, \qquad
      u^{\hla|\hmu}=\Ber\begin{pmatrix}
        u_1{}^{\hla} & \vline & u_1{}^{\hmu} \\
                                 \hline \vphantom{\int_a^b}
        u_{\hat 1}{}^{\hla} & \vline & u_{\hat 1}{}^{\hmu}
      \end{pmatrix}\,,\\
      u^{*a|\hmu}&=\Ber^*\begin{pmatrix}
        u_1{}^a & \vline & u_1{}^{\hmu} \\
                                 \hline \vphantom{\int_a^b}
        u_{\hat 1}{}^a & \vline & u_{\hat 1}{}^{\hmu}
      \end{pmatrix}\,, \qquad
      u^{*a|b}=\Ber^*\begin{pmatrix}
        u_1{}^a & \vline & u_1{}^{b} \\
                                 \hline \vphantom{\int_a^b}
        u_{\hat 1}{}^a & \vline & u_{\hat 1}{}^{b}
      \end{pmatrix}\,.
\end{align*}
Also, in the open domain where the variables $u^{a|\hmu}$ and $u^{*a|\hmu}$ are invertible (for arbitrary fixed $a|\hmu$) we have a locally defined inverse super Pl\"{u}cker map $\Pl_{(a|\hmu)}^{-1}$, which sends the collection of essential super Pl\"{u}cker coordinates to a matrix $U'$ whose entries are given by the formulas
\begin{align*}
    u'_1{}^b&=\frac{u^{b|\hmu}}{u^{a|\hmu}}\,, \qquad u'_1{}^{\hnu}=\frac{u^{{\hnu}|\hmu}}{u^{a|\hmu}}\,,
    \\
    u'_{\hat 1}{}^b&=\frac{u^{*a|b}}{u^{*a|\hmu}}\,,
    \qquad u'_{\hat 1}{}^{\hnu}=\frac{u^{*a|\hnu}}{u^{*a|\hmu}}\,.
\end{align*}
Following the same pattern as above, we observe that $\Pl_{(a|\hmu)}^{-1}\circ \Pl=\id$ identically (on the equivalence classes of matrices), but the condition  that $\Pl\circ\Pl_{(a|\hmu)}^{-1}  =\id$ gives the relations that we are looking for. We arrive at the relations (for all $b,c,\hnu,\hla$):
\begin{align*}
     \Ber\begin{pmatrix}
        {u^{b|\hmu}}/{u^{a|\hmu}} & \vline & {u^{{\hnu}|\hmu}}/{u^{a|\hmu}} \\
                                 \hline \vphantom{\int_a^b}
        {u^{*a|b}}/{u^{*a|\hmu}} & \vline & {u^{*a|\hnu}}/{u^{*a|\hmu}}
      \end{pmatrix}&=\la\, u^{b|\hnu}\,, \quad
      \Ber\begin{pmatrix}
         u^{\hla|\hmu}/u^{a|\hmu} & \vline &  u^{\hnu|\hmu}/u^{a|\hmu} \\
                                 \hline \vphantom{\int_a^b}
        u^{*a|\hla}/u^{*a|\hmu} & \vline & u^{*a|\hnu}/u^{*a|\hmu}
      \end{pmatrix}=\la\, u^{\hla|\hnu}\,,\\
      \Ber^*\begin{pmatrix}
        {u^{b|\hmu}}/{u^{a|\hmu}} & \vline & {u^{{\hnu}|\hmu}}/{u^{a|\hmu}} \\
                                 \hline \vphantom{\int_a^b}
        {u^{*a|b}}/{u^{*a|\hmu}} & \vline & {u^{*a|\hnu}}/{u^{*a|\hmu}}
      \end{pmatrix}&=\la^{-1}\,u^{*b|\hnu}\,, \quad
      \Ber^*\begin{pmatrix}
        {u^{b|\hmu}}/{u^{a|\hmu}} & \vline & {u^{c|\hmu}}/{u^{a|\hmu}} \\
                                 \hline \vphantom{\int_a^b}
        {u^{*a|b}}/{u^{*a|\hmu}} & \vline & {u^{*a|c}}/{u^{*a|\hmu}}
      \end{pmatrix}=\la^{-1}\,u^{*b|c}\,,
\end{align*}
where $\la$ is a factor to be determined (not to be confused with the index $\hla$). If we substitute $a|\hnu$ as $b|\hnu$, we obtain that
\begin{equation*}
    \la=(u^{a|\hmu})^{-1}= u^{*a|\hmu}\,.
\end{equation*}
Hence, by substituting   back into the relations and using the properties of $\Ber$ and $\Ber^*$,  we obtain
\begin{align*}
    ({u^{a|\hmu}})^{-1}{u^{*a|\hmu}} \Ber\begin{pmatrix}
        {u^{b|\hmu}}  & \vline & {u^{{\hnu}|\hmu}} \\
                                 \hline \vphantom{\int_a^b}
        {u^{*a|b}}  & \vline & {u^{*a|\hnu}}
      \end{pmatrix}
      &=(u^{a|\hmu})^{-1}\, u^{b|\hnu}\,, \\
      ({u^{a|\hmu}})^{-1}{u^{*a|\hmu}}\Ber\begin{pmatrix}
         u^{\hla|\hmu}  & \vline &  u^{\hnu|\hmu}  \\
                                 \hline \vphantom{\int_a^b}
        u^{*a|\hla}  & \vline & u^{*a|\hnu}
      \end{pmatrix}
      &=(u^{a|\hmu})^{-1}\, u^{\hla|\hnu}\,,\\
      {u^{a|\hmu}}({u^{*a|\hmu}})^{-1}\Ber^*\begin{pmatrix}
        {u^{b|\hmu}}  & \vline & {u^{{\hnu}|\hmu}}  \\
                                 \hline \vphantom{\int_a^b}
        {u^{*a|b}}  & \vline & {u^{*a|\hnu}}
      \end{pmatrix}
      &=u^{a|\hmu}\,u^{*b|\hnu}\,, \\
      {u^{a|\hmu}}({u^{*a|\hmu}})^{-1}\Ber^*\begin{pmatrix}
        {u^{b|\hmu}}  & \vline & {u^{c|\hmu}}  \\
                                 \hline \vphantom{\int_a^b}
        {u^{*a|b}}  & \vline & {u^{*a|c}}
      \end{pmatrix}
      &=u^{a|\hmu}\,u^{*b|c}\,.
\end{align*}
Note that the first and the third equations are equivalent. Together with the condition $(u^{a|\hmu})^{-1}= u^{*a|\hmu}$ that we obtained, this gives~\eqref{eq.sp0}--\eqref{eq.sp3}.
\end{proof}

Note that relations~\eqref{eq.sp0}--\eqref{eq.sp3}  are homogeneous in the sense of weight. Getting rid of the denominators gives the following form of the relations.
\begin{corollary}[alternative form of super Pl\"{u}cker relations for  $1|1$-planes]
Without denominators, the relations between the   super Pl\"{u}cker coordinates are as follows:
\begin{align}\label{eq.altsp1}
    u^{a|\hmu}u^{b|\hnu} &=u^{a|\hnu}u^{b|\hmu} + u^{\hmu|\hnu}u^{*a|b}(u^{a|\hmu})^2\,,\\
    \label{eq.altsp2}
    u^{a|\hnu}u^{\hla|\hmu} &=u^{a|\hmu}u^{\hla|\hnu}+ u^{\hnu|\hmu}u^{*a|\hla}(u^{a|\hnu})^2\,,\\
    u^{a|\hmu}u^{*a|c}u^{b|\hmu} &=u^{a|\hmu}u^{*a|b}u^{c|\hmu}+ u^{*b|c}(u^{b|\hmu})^2\,. \label{eq.altsp3}
\end{align}
\end{corollary}\qed

It is interesting to compare relations~\eqref{eq.altsp1}--\eqref{eq.altsp3} with other  super Pl\"{u}cker relations that we have obtained before. In the relations that we obtained previously, such as~\eqref{eq.plueck.even},\eqref{eq.plueck.odd} and \eqref{eq.esssplueckev}, \eqref{eq.esssplueckod}, there is no mixing of odd variables with the even ones, and the odd variables enter  only linearly (suggesting   interpretation in terms of vector bundles). A particular new feature of relations~\eqref{eq.altsp1}--\eqref{eq.altsp3} is the emergence of  the product of odd variables  in an additive term  in the  even relation~\eqref{eq.altsp1}, and the  non-linearity of the odd relations~\eqref{eq.altsp2},\eqref{eq.altsp3}.

Now, for the case of $r|s$-planes, we will give only the relations in the Berezinian form. (Possibly they can be simplified as is the case for $r|0$, but we do not possess this simplification at present.) We will use notations such as $\hmund=\hmu_1\ldots\hmu_s$ and $\aund=a_1\ldots a_s$ for brevity.

\begin{theorem}[``super Pl\"{u}cker relations for  $r|s$-planes in   $n|m$-space'']
\label{thm.splueckrs}
The super Pl\"{u}cker relations for the variables~\eqref{eq.coorfinrs}, the homogeneous coordinates on   $\pfinth$, are:
\begin{align}
\label{eq.relrs0}
&u^{\aund|\hmund}\,u^{*\aund|\hmund}=1
\\
\label{eq.relrs1}
    &\Ber\left(\begin{array}{ccc}
        u^{a_1\ldots a_{i-1}b_ja_{i+1}\ldots a_r|\hmund} & \vline & u^{a_1\ldots a_{i-1}\hnu_{\be}a_{i+1}\ldots a_r|\hmund} \\
                                 \hline \vphantom{\int_A^B}
         u^{*\aund|\hmu_1\ldots\hmu_{\al-1}b_j\hmu_{\al+1}\ldots \hmu_s} & \vline & u^{*\aund|\hmu_1\ldots\hmu_{\al-1}\hnu_{\be}\hmu_{\al+1}\ldots \hmu_s} \vphantom{\int_A^B}\\
        \end{array}\right)_{{i,j=1\ldots r,}\atop{\al,\be=1\ldots s}}
         =(u^{\aund|\hmund})^{r+s-1}u^{\bund|\hnund}\,,
         \\
         \label{eq.relrs2}
        &\Ber\left(\begin{array}{c:ccc}
        u^{a_1\ldots a_{i-1}b_ja_{i+1}\ldots a_r|\hmund}
        & u^{a_1\ldots a_{i-1}\hla a_{i+1}\ldots a_r|\hmund}
        & \vline
        & u^{a_1\ldots a_{i-1}\hnu_{\be}a_{i+1}\ldots a_r|\hmund} \\
                                 \hline \vphantom{\int_A^B}
         u^{*\aund|\hmu_1\ldots\hmu_{\al-1}b_j\hmu_{\al+1}\ldots \hmu_s}
         & u^{*\aund|\hmu_1\ldots\hmu_{\al-1}\hla\hmu_{\al+1}\ldots \hmu_s}
         &\vline
         & u^{*\aund|\hmu_1\ldots\hmu_{\al-1}\hnu_{\be}\hmu_{\al+1}\ldots \hmu_s} \vphantom{\int_A^B}\\
        \end{array}\right)_{{{i=1\ldots r,}\atop{j=1\ldots r-1,}}\atop{\al,\be=1\ldots s}}
        \notag\\
        &  \hspace{290pt} =(u^{\aund|\hmund})^{r+s-1}u^{b_1\ldots b_{r-1}\hla|\hnund}\,,
        \\
\label{eq.relrs3}
    &\Ber^*\left(\begin{array}{ccc:c}
        u^{a_1\ldots a_{i-1}b_ja_{i+1}\ldots a_r|\hmund} & \vline
        & u^{a_1\ldots a_{i-1} c a_{i+1}\ldots a_r|\hmund}
        & u^{a_1\ldots a_{i-1}\hnu_{\be}a_{i+1}\ldots a_r|\hmund} \\
                                 \hline \vphantom{\int_A^B}
         u^{*\aund|\hmu_1\ldots\hmu_{\al-1}b_j\hmu_{\al+1}\ldots \hmu_s} & \vline
         & u^{*\aund|\hmu_1\ldots\hmu_{\al-1} c \hmu_{\al+1}\ldots \hmu_s}
         & u^{*\aund|\hmu_1\ldots\hmu_{\al-1}\hnu_{\be}\hmu_{\al+1}\ldots \hmu_s} \vphantom{\int_A^B}
         \\
        \end{array}\right)_{{i,j=1\ldots r,}\atop{{\al=1\ldots s}\atop{\be=1\ldots s-1}}}
         \notag\\
        &  \hspace{260pt}=(u^{\aund|\hmund})^{-r-s+1}u^{*b_1\ldots b_r|c\hnu_1\ldots \hnu_{s-1}}\,.
\end{align}
\end{theorem}
(In the Berezinians in~\eqref{eq.relrs1}--\eqref{eq.relrs3}, $i$,$\al$ label   rows, while $j$,$\be$ label  columns; 
there is one column of ``wrong'' parity   in an even position~\eqref{eq.relrs2} and in an odd position in \eqref{eq.relrs3}.)
\begin{proof}
Very much along the same lines as for Theorem~\ref{thm.essplueckr0} and Theorem~\ref{thm.splueck11}. We have the ``direct'' super Pl\"{u}cker map $\Pl$ that sends the $r|s\times n|m$ matrix $U$ of homogeneous coordinates on the super Grassmannian $G_{r|s}(V)$ to the homogeneous coordinates~\eqref{eq.coorfinrs} on $\pfin$ according to formulas~\eqref{eq.ua1mus}--\eqref{eq.usta1bmus}. And in the domain where $u^{\aund|\hmund}$ and $u^{*\aund|\hmund}$ (for some fixed $\aund|\hmund$) are invertible, there is the ``inverse'' super Pl\"{u}cker map $\Pl^{-1}_{(\aund|\hmund)}$ that sends the variables $u^{\bund|\hnund}$ etc. to a matrix $U'$ according to formulas~\eqref{eq.wjb}--\eqref{eq.whbhnu} (with $u'_j{}^b$ written instead of $w_j{}^b$, etc.). For the composition in one order, we have $\Pl_{(\aund|\hmund)}^{-1}\circ \Pl=\id$ automatically, while for the composition in the opposite order, the condition that $\Pl\circ \Pl_{(\aund|\hmund)}^{-1}=\id$ specifies the image of $\Pl$ and gives the relations we are looking for. Exactly as above, this amounts to the  equations of the form
\begin{align*}
     u'{}^{b_1\ldots b_r|\hnu_1\ldots \hnu_s}&=\la\, u^{b_1\ldots b_r|\hnu_1\ldots \hnu_s}\,, \qquad\ \ \
      u'{}^{b_1\ldots b_{r-1}\hla|\hnu_1\ldots \hnu_s} =\la\, u^{b_1\ldots b_{r-1}\hla|\hnu_1\ldots \hnu_s}\,,\\
      u'{}^{*b_1\ldots b_{r}|\hnu_1\ldots \hnu_s}&=\la^{-1}\,u^{*b_1\ldots b_{r}|\hnu_1\ldots \hnu_s}\,, \qquad
      u'{}^{*b_1\ldots b_{r}|c\hnu_1\ldots \hnu_{s-1}} =\la^{-1}\,u^{*b_1\ldots b_{r}|c\hnu_1\ldots \hnu_{s-1}}\,,
\end{align*}
for all combinations of indices,  where $\la$ (not to be confused with the index $\hla$) is a factor to be determined. Here the variables with the prime at the left-hand side are obtained by the corresponding Berezinian formulas from $u'_j{}^b$, etc. We deduce, by using these relations in the particular case when $\bund|\hnund=\aund|\hmund$, that $\la=(u^{\aund|\hmund})^{-1}=u^{*\aund|\hmund}$, so in particular $u^{*\aund|\hmund}=(u^{\aund|\hmund})^{-1}$. By substituting that into the relations and simplifying, similarly with what we did in the proofs of Theorems~\ref{thm.essplueckr0} and~\ref{thm.splueck11}, we arrive at the relations in the form~\eqref{eq.relrs1}--\eqref{eq.relrs3}, together with \eqref{eq.relrs0}, as claimed.
\end{proof}

\subsection{``Khudaverdian's relations''} \label{subsec.khud}

There is an interesting alternative approach to   relations   satisfied by a simple multivector   suggested by H.~Khudaverdian  (private communication). The main idea  is essentially based on supergeometry and the   construction of exterior powers $\L^{r|s}(V)$.  As we shall show, it works non-trivially   already in the classical case giving relations in an unexpected form.

Consider first the general setup of $r|s$-planes in an $n|m$-space $V$. Recall once again that if $r|s$ independent vectors $\up_i$ span $L\subset V$, the same vectors (or their parity-reversed copies, to be more precise) span the $s|r$-plane $\Pi L$ in $\Pi V$. Consider the ``non-linear wedge product'' $F:=[\up_1\ldots,\up_r \,|\, \up_{\hat 1},\ldots \up_{\hat s}]$, which is an element of $\L^{r|s}(V)$,
\begin{equation}
    F(p^1,\ldots,p^r\,|\, p^{\hat 1}, \ldots, p^{\hat s}) = \Ber\bigl(\langle \up_i,p^j\rangle\bigr)\,,
\end{equation}
where $\langle \up_i,p^j\rangle=  u_i{}^ap_a{}^j$. (The function $F$ is $\pl(L)$  in the notation of subsection~\ref{subsec.proofemb}.) In the same way, these vectors considered as parity-reversed define a multivector $F^*:=[\up_{\hat 1},\ldots \up_{\hat s}\,|\,\up_1\ldots,\up_r]\in \L^{s|r}(\Pi V)$. On the formula level, this is just the Berezinian of the parity-reverse of the same matrix:
\begin{equation}
    F^*(p^{\hat 1}, \ldots, p^{\hat s}\,|\,p^1,\ldots,p^r) =
    \Ber\bigl(\langle \up_i,p^j\rangle\bigr)^{\Pi}=
    \Ber^*\bigl(\langle \up_i,p^j\rangle\bigr)\,.
\end{equation}

So, if we are a bit sloppy about parities of the arguments,
\begin{equation}
    F^*=\frac{1}{F}\,,
\end{equation}
where at the left-hand side $F^*$ is an element of $\L^{s|r}(\Pi V)$ and at the right-hand side $F$ is an element of $\L^{r|s}(V)$. (We have   discussed $F^*$ under the name ``$\Pi$-dual Pl\"{u}cker transform'' in subsections~\ref{subsec.11in22} and \ref{subsec.proofemb} and denoted it $\pl^*(U)$ or $\pl^*(L)$.)

Hence we have arrived at the following  statement:

\begin{theorem}[Khudaverdian, unpublished\footnote{We thank him for the permission to use it here.}]\label{thm.hovrel}
If $F\in \L^{r|s}(V)$ is a simple multivector corresponding to an $r|s$-plane, then the function $G=1/F$ regarded as a function of the parity-reversed arguments is also a multivector, of degree $s|r$, for $\Pi V$, i.e. $G$  satisfies fundamental equations~\eqref{eq.fund} with respect to the shifted parities of the indices. \qed
\end{theorem}

This can be effectively  reformulated as follows.
\begin{corollary}
If an even multivector $F\in \L^{r|s}(V)$ is simple, it satisfies an   additional  system of quadric equations:
\begin{equation}\label{eq.hovrel}
    \der{F}{p_a^i}\,\der{F}{p_b^j} - (-1)^{\itt\jt+\at(\itt+\jt)}\,\der{F}{p_a^j}\,\der{F}{p_b^i} -F\,\dder{F}{p_a^i}{p_b^j}=0\,,
\end{equation}
for all $i,j=1,\ldots,r,\hat 1,\ldots,\hat s$ and $a,b=1,\ldots,n,\hat 1,\ldots,\hat m$.
\end{corollary}
\begin{proof}
Suppose an array $p$ of $r|s$ covectors in an $n|m$-dimensional space $V$ is regarded as an array of $s|r$-covectors in the $m|n$-dimensional space $\Pi V$.
Then by definition  a function $G=G(p)$ which is an element of $\L^{s|r}(\Pi V)$ satisfies the system of fundamental equations~\eqref{eq.fund} with the reversed parities of the indices. It will be:
\begin{equation}\label{eq.pifund}
    \dder{G}{{p_a{}^i}}{{p_b{}^j}}  - (-1)^{\itt\,\jt +\at(\itt+\jt)} \dder{G}{p_a{}^j}{p_b{}^i} = 0
\end{equation}
(the minus sign between the terms instead of the plus sign in~\eqref{eq.fund}). Indeed, for the reversed parity we have $(\itt+1)(\jt +1) +(\at+1)(\itt+\jt)=1+ \itt\,\jt +\at(\itt+\jt)$. For $G=1/F$, we have
\begin{equation*}
    \der{G}{p_b{}^j}=-F^{-2}\,\der{F}{p_b{}^j}\,, \quad
    \dder{G}{{p_a{}^i}}{{p_b{}^j}}=
    2 F^{-3} \der{F}{p_a{}^i} \,\der{F}{p_b{}^j} -F^{-2}\dder{F}{{p_a{}^i}}{{p_b{}^j}}\,,
\end{equation*}
and after substituting in~\eqref{eq.pifund} and ridding of the denominators, we arrive at
\begin{equation}\label{eq.intermed}
    2 \left(\der{F}{p_a{}^i} \,\der{F}{p_b{}^j}
     - (-1)^{\itt\,\jt +\at(\itt+\jt)} \der{F}{p_a{}^j} \,\der{F}{p_b{}^i} \right)
    -F\,\dder{F}{{p_a{}^i}}{{p_b{}^j}}
    + (-1)^{\itt\,\jt +\at(\itt+\jt)} F\,\dder{F}{{p_a{}^j}}{{p_b{}^i}}=0\,.
\end{equation}
Since we assume that $F$ is an $r|s$-vector, it satisfies the original system~\eqref{eq.fund}, which is
\begin{equation*}
    \dder{F}{{p_a{}^i}}{{p_b{}^j}}  + (-1)^{\itt\,\jt +\at(\itt+\jt)} \dder{F}{p_a{}^j}{p_b{}^i} = 0\,
\end{equation*}
and from where we can express the last term in~\eqref{eq.intermed}, obtaining it coinciding with the penultimate term (with the same sign). By dividing by $2$, we get exactly~\eqref{eq.hovrel}.
\end{proof}

We shall refer to equations~\eqref{eq.hovrel} as the  \emph{Khudaverdian relations}.

Since the relations given by~\eqref{eq.hovrel}   are quadric, the same as the classical Pl\"ucker relations, it is very tempting to think that they simply coincide with the Pl\"ucker relations in the classical situation and provide the desired super analog in general. As we shall see, this is not exactly the case in general, though  is  true for $r|s=2|0$.

\begin{theorem}
For an even bivector $T\in \L^2(V)$ in an $n|m$-dimensional space $V$, the Khudaverdian relations~\eqref{eq.hovrel} are equivalent to the  super  Pl\"ucker relations~\eqref{eq.plueckfor2}.
\end{theorem}
\begin{proof} Take $T(p)=T^{cd}p^1_cp^2_d=(-1)^{\ct+\dt}\,p^1_cp^2_d\,T^{cd}$. Here $\widetilde{T^{cd}}=\ct+\dt$ and $T^{cd}=-(-1)^{\ct\dt}T^{dc}$. Note that    the relation for $T$ takes the form
\begin{equation}\label{eq.hovrel20}
    \der{T}{p_a^i}\,\der{T}{p_b^j} - \der{T}{p_a^j}\,\der{T}{p_b^i} -T\,\dder{T}{p_a^i}{p_b^j}=0\,,
\end{equation}
where the indices $i,j$  run over $1,2$. If $i=j$, then the relation is empty since the first two terms cancel and the last term vanishes because $T$ is bilinear. Hence the possibly nontrivial case is $i\neq j$ and we can take $i=1$, $j=2$. We have
$\lder{T}{p^1_a}=(-1)^{\at+\dt}p^2_dT^{ad}$, $\lder{T}{p^1_b}=(-1)^{\bt+\dt}p^2_dT^{bd}$, $\lder{T}{p^2_b}=(-1)^{\ct+\bt+\ct\bt}p^1_cT^{cb}$, $\lder{T}{p^2_a}=(-1)^{\ct+\at+\ct\at}p^1_cT^{ca}$, and, finally,
$\ldder{T}{p_a^1}{p_b^2}=(-1)^{\at+\bt+\at\bt}T^{ab}$. Substituting into~\eqref{eq.hovrel20} gives
\begin{multline*}
    (-1)^{\at+\dt}p^2_dT^{ad}(-1)^{\ct+\bt+\ct\bt}p^1_cT^{cb}
    -(-1)^{\ct+\at+\ct\at}p^1_cT^{ca}(-1)^{\bt+\dt}p^2_dT^{bd}
    - \\(-1)^{\ct+\dt}\,p^1_cp^2_d\,T^{cd}(-1)^{\at+\bt+\at\bt}T^{ab}=0\,.
\end{multline*}
After taking $p^1_cp^2_d$ out and getting rid of common sign factors, it becomes
\begin{equation*}
     T^{ad}T^{cb} (-1)^{\ct\bt+\at\ct}
     +T^{ac}T^{bd}(-1)^{\dt(\at+\ct)}
     -T^{ab}T^{cd}(-1)^{\bt(\at+ \ct+ \dt)  + \at(\ct+ \dt)}=0
\end{equation*}
or
\begin{equation*}
T^{ab}T^{cd}(-1)^{\bt(\at+ \ct+ \dt)}= T^{ac}T^{bd}(-1)^{\ct(\at+\dt)} + T^{ad}T^{cb}(-1)^{\ct\bt+\at\dt}=0\,,
\end{equation*}
which is exactly the super Pl\"{u}cker relation~\eqref{eq.plueckfor2} for $r=2$ (see Example~\ref{ex.ktwo}).
\end{proof}

In general,   Khudaverdian's relations are not equivalent to Pl\"ucker's relations:

\begin{example} Consider   $T\in \L^3(V)$, so $T(p)=T^{abc}p_1^ap_2^bp_3^c$\,. Working as above, we can show that the Khudaverdian relations~\eqref{eq.hovrel} are equivalent to the following relations for the components (below for simplicity we present  the formulas for the purely even case):
\begin{multline}\label{eq.hovrel3}
    T^{a_1a_2a_3}T^{b_1b_2b_3} + T^{a_1a_2b_3}T^{b_1b_2a_3} - T^{b_1a_2a_3}T^{a_1b_2b_3} \\
    +T^{b_2a_2b_3}T^{a_1b_1b_3}
    - T^{b_1a_2b_3}T^{a_1b_2a_3}
    + T^{b_2a_2b_3}T^{a_1b_1a_3}=0\,.
\end{multline}
Note the $6$ terms. At the same time, the Pl\"ucker relations will be
\begin{equation}\label{eq.plueckthree}
     T^{a_1a_2a_3}T^{b_1b_2b_3}= T^{b_1a_2a_3}T^{a_1b_2b_3}+ T^{b_2a_2a_3}T^{b_1a_1b_3}+ T^{b_3a_2a_3}T^{b_1b_2a_1}=0\,,
\end{equation}
with the $4$ terms.
One can observe directly (this is a  rewarding calculation) that the six-term relation~\eqref{eq.hovrel3} is precisely the difference of two four-term relations~\eqref{eq.plueckthree} with different combinations of indices, so it is their consequence, but not equivalent to   them.
\end{example}

In spite of the fact that it does not give all the relations, the value of this approach is that it is formulated in terms of the multivector $F$ or $T$  as a whole and not some  particular components of it. This seems an advantage particularly for general $r|s$ and possibly this method can give new results there.

\section{Discussion. Connection with   (conjectural) super cluster algebras} \label{sec.clust}

\subsection{General discussion of results.}
In the previous sections, we constructed a superanalog   of the classical Pl\"ucker embedding $G_k(V)\to P(\L^k(V))$ for a vector space $V$ with $\dim V=n$. For a general superspace $V$ with $\dim V=n|m$, we have obtained two versions of such a superanalog. In the case of $r$- or $r|0$-planes in the $n|m$-space $V$ (i.e. the case of purely even planes in the superspace) we constructed the super Pl\"{u}cker map as a map
\begin{equation*}
     \pl\co G_r(V)\to P(\L^r(V))\,,
\end{equation*}
$\dim V=n|m$, in the closest analogy with the classical case. In the general case of $r|s$-planes in the $n|m$-space $V$, the super Pl\"{u}cker map (that we constructed) is a map
\begin{equation*}
  \Pl\co G_{r|s}(V)\to  P_{1,-1}\left(\L^{r|s}(V)\oplus \L^{s|r}(\Pi V)\right)\,.
\end{equation*}
We  found that it is necessary to use multivectors both in $V$ itself and in the parity-reversed space $\Pi V$ for construction of the codomain,   otherwise it is not possible to get the invertibility.

In the previous sections, we proved that (in both cases) we obtain an embedding. We treated the two cases by different methods. While the case of $k$-planes admits a purely algebraic treatment very close to the classical case (but still presenting some surprising answers), for the general case we had to use a more sophisticated approach based on the Voronov-Zorich construction of exterior powers (where the superanalog of a multivector is a function on covectors, no longer multilinear as opposed to the classical prototype).

A key difference with the classical Pl\"{u}cker map, for both considered cases, is the necessity to distinguish  ``essential'' Pl\"{u}cker coordinates. In the classical situation of $k$-planes in $n$-space, the Pl\"{u}cker coordinates are all components of a non-zero $k$-vector treated as homogeneous coordinates in the corresponding projective space. There is no distinction of ``essential'' and ``inessential'' components. Unlike that, in the supercase we identified either a subset of  even and odd components of an even $r$-vector, for $r$-planes in $n|m$-space, or a finite number of even and odd variables obtained by evaluating the Pl\"{u}cker transform of an $r|s$-plane on combinations of covectors from $V$ and $\Pi V$. Therefore we effectively work with ``reduced'' Pl\"{u}cker maps,
\begin{equation*}
     \pl\co G_r(V)\to \pess\,,
\end{equation*}
or
\begin{equation*}
  \Pl\co G_{r|s}(V)\to  \pfin\,.
\end{equation*}
(in the latter case, $\pfin$ is a finite-dimensional supermanifold, which can be seen as a finite-dimensional quotient of $P_{1,-1}\left(\L^{r|s}(V)\oplus \L^{s|r}(\Pi V)\right)$). In the two cases, the idea of ``essential'' Pl\"{u}cker variables comes about differently. In the algebraic case of $r$-planes, it turns out that among all components of a simple multivector (i.e. in the image of the Pl\"{u}cker map) some even components are nilpotent by the virtue of the super Pl\"{u}cker relations as thus cannot be invertible, forcing   others (which we have identified as ``essential'') to be invertible. ``Non-essential'' components can be, further, eliminated and expressed via the ``essential'' components. In the general case of $r|s$-planes, ``essential'' Pl\"{u}cker coordinates emerge as particular values of the Pl\"{u}cker transform or those ``super minors'' of the matrix of homogeneous coordinates on the super Grassmannian (including generalized minors  giving odd values) that are necessary for inverting the super Pl\"{u}cker transform (on its image). When both approaches are applicable, we very satisfactorily come to the same collection of variables.

One feature that makes the super Pl\"{u}cker map substantially different from the classical situation, is the emergence of rational, not polynomial, functions in the case of $r|s$-planes. The emergence of rational, not polynomial in general, objects seems to be unavoidable in super situation. Compare in particular with~\cite{tv:ber}. However, it is a subtle question whether the map $G_{r|s}(V)\to P_{1,-1}\left(\L^{r|s}(V)\oplus \L^{s|r}(\Pi V)\right)$) should   be regarded as rational or regular. Indeed, although the formulas involve denominators, this map is defined on all supermanifold, and    the image of a particular plane being a rational function of covectors, with   poles depending on a plane. However, as soon as we use the reduced supermanifold as the codomain and $\Pl\co G_{r|s}(V)\to  \pfin$, our map becomes indeed rational. There is something to investigate here. One may argue, why use Berezinians as we do  and not try some polynomial functions instead. In principle, there are invariant polynomial functions $\Ber^+$ and $\Ber^{-}$ introduced in~\cite{tv:ber} (so that $\Ber A=\Ber^+ A/\Ber^{-} A$), but they are not multiplicative, so it is not clear how helpful they could be for replacement of the super Pl\"{u}cker coordinates as constructed in this paper. In all cases, it is the rational functions $\Ber$ and $\Ber^*$ that arise in the denominators of the formula for the inverse of a super matrix.

For the classical Pl\"{u}cker map, a key part of the theory is the description of its image as an algebraic variety in the corresponding projective space by the classical Pl\"{u}cker relations. In Section~\ref{sec.simpl} for $r$-planes and in  Section~\ref{sec.general} for $r|s$-planes (and with an additional clarification for $r$-planes), we obtained  superanalogs of Pl\"{u}cker relations in different forms.

In the purely algebraic case of $r$-planes $n|m$-space, there are two versions of the super Pl\"{u}cker relations. One is for all the components of an even multivector $T\in \L^r(V)$, see Theorem~\ref{thm.splueckcompon}. In particular it includes some algebraic relations (such as nilpotence condition) that look very far from the classical example. However, it is possible to identify a  reduced  subset of relations only for the essential super Pl\"{u}cker coordinates, and this version also coincides with the   relations obtained by a different method which we use  for the general case of $r|s$-planes. (The remaining part of the super Pl\"{u}cker relations for the case of $r$-planes makes it possible to eliminate the rest of the Pl\"{u}cker variables expressing them via the essential ones.) See Theorem~\ref{thm.essplueckr0}. These super Pl\"{u}cker relations for the  essential super Pl\"{u}cker variables  for $r$-planes   look relatively closer to the classical   relations. We discuss their application in the next subsection.

We also managed to obtain the super Pl\"{u}cker relations in the general case of $r|s$-planes, see Theorem~\ref{thm.splueck11} and Theorem~\ref{thm.splueckrs}. These look substantially more complicated in comparison with the classical case. We hope to study them further 
elsewhere.  Another topic for further study is the ``Khudaverdian relations''  that we introduced in subsection~\ref{subsec.khud} and explored for the case of $r$-planes, in particular $r=2$. They can   be helpful for obtaining and analyzing the super Pl\"{u}cker relations in other cases as well.  It is also interesting to find a connection with~\cite{kasman:2007} where, for the classical setup, it is proved that the Pl\"{u}cker relations for $2$-planes are in a certain precise sense universal.

\subsection{``Super cluster structure''.}
Super Pl\"ucker relations that we have obtained in this paper lead to a   link with another burning topic, which is ``super cluster algebras''.

Recall that cluster algebras were introduced by~Fomin and Zelevinsky, see~\cite{fomin-zelevinsky:clusterI} and~\cite{fomin-zelevinsky:clusterII}. Cluster algebras found applications in a great number of mathematical areas. (See for example the  book Gekhtman--Shapiro--Vainshtein~\cite{gekhshava:book2010},  where in particular  a connection   with Poisson geometry is developed.)  Informally, the idea of cluster algebras (or a ``cluster structure'') can be rendered as follows. A cluster algebra is a subalgebra  inside some field of rational functions which is generated as an algebra by a collection of groups of variables (called clusters  or 
``extended clusters'') such that each such cluster gives a basis of transcendence (i.e. rational generators) for the ambient field of rational functions and where every cluster is obtained from another cluster by replacing exactly one variable by a rational transformation of a very special form (``cluster mutation'') based on some ``mutation data'' attached to each cluster and encoded by a special form ``exchange matrix'' or a quiver. Some variables are not mutated, they are called stable and form a ground ring of the cluster algebra. The variables that are mutated are specifically referred to as ``cluster variables''. Together with mutation of variables when passing from one cluster to another, the mutation data themselves are also mutated by a certain rule. The exact rules of cluster mutations were the crucial discovery of Fomin and Zelevinsky.

Particular instances of these rules can be seen in many concrete geometric examples that  serve  as prototypes for the definition of cluster algebras. For our purposes it is crucial that the Pl\"{u}cker relations for the ordinary Grassmannian $G_k(\R{n})$ are one of them. The following example is now classical and can be found e.g. in~\cite[Ch. 2]{gekhshava:book2010}.

\begin{example}  The  Pl\"{u}cker relation for $G_2(\R{4})$
can be  written as
\begin{equation}\label{eq.plueckclassic}
     \boxed{T^{13}}\,\boxed{T^{24}}= T^{12}T^{34}+T^{14}T^{23}
\end{equation}
and in this form  it expresses a  mutation  between  the Pl\"{u}cker coordinates    $T^{13}$ and $T^{24}$ as cluster variables, each making a separate cluster. Altogether there are two clusters, $(T^{13})$ and $(T^{24})$.  The variables  labeled $12$, $34$, $14$,  and $23$ are not mutated and are regarded as generating the ground ring of the cluster algebra. A geometric picture can be used. Tensor indices $1,2,3,4$ are identified with the vertices of a square, so that components $T^{ab}$ are depicted as segments $(ab)$ (``diagonals''). The mutated cluster variables e.g. $T^{13}$ correspond to true diagonals, while the ground variables correspond to the sides.  We have boxed mutated
variables in the formula~\eqref{eq.plueckclassic}.   Similar geometric picture works for $G_2(\R{n})$, for other $n$. Clusters are formed by collections of  variables $T^{ab}$ corresponding to maximal sets of non-intersecting true diagonals in the $n$-gon.
\end{example}

Superanalog of  cluster algebras is highly sought after.  ``Super cluster algebras'' in general   remain   conjectural though  important steps towards their definition have been made very recently:  \cite{ovsienko-supercluster:2015},~\cite{ovsienko-shapiro:2018}, and~\cite{srivastava:tosupercluster-2017}. The notions  introduced in \cite{ovsienko-supercluster:2015},~\cite{ovsienko-shapiro:2018} in and~\cite{srivastava:tosupercluster-2017} are different. Perhaps none can be regarded as final. More information from concrete geometric and algebraic examples should be obtained. Since in this paper we obtained super Pl\"{u}cker relations for the super Grassmannians $G_r(\R{n|m})$ and $G_{r|s}(\R{n|m})$, it is natural to ask whether they provide any kind of a \textbf{``super cluster structure''}. We answer this question in the following way. We regard our results on super Pl\"{u}cker relations as ``experimental material'' for a future theory. Below we put forward an interpretation of our super Pl\"{u}cker relations in terms of a ``super cluster structure''.

We shall consider the cases of $G_2(\R{4|m})$ and $G_2(\R{5|m})$ for $m=1$. We believe that the case of arbitrary $m$ should make no substantial difference and actually the same should work for $G_2(\R{n|m})$ with general $n|m$. (With taking into account the known differences of the classical cluster structures of $G_2(\R{n})$ for small and larger values of $n$, see~\cite[Ch.2]{gekhshava:book2010}.)  But this requires further detailed analysis and will be given elsewhere.  Next step should be a generalization to arbitrary $G_r(\R{n|m})$. (While the case of $G_{r|s}(\R{n|m})$ remains intriguing.)

Instead of attempting a particular theory of  super cluster algebras, we    base our considerations on   natural conditions that one can   generalize from the ordinary case: (1) each (extended) cluster containing even and odd variables should generate rationally all other variables, i.e. the number of variables in every extended cluster should be equal to the superdimension; (2) a mutation can use division only by one particular cluster variable from the mutated cluster, and it is not allowed to divide by stable variables. We shall see that in order to satisfy these assumptions, we will have to drop some other properties familiar in the classical case.

How to introduce a ``super cluster structure'' into the considered super Grassmannians? We know that $G_{r|s}(\R{n|m})$ contains as a submanifold the product of ordinary Grassmannians $G_{r}(\R{n})\times G_{s}(\R{m})$. Hence the algebra of functions on the latter product is a quotient of the algebra of functions on the super Grassmannian. In particular, in the case  of $G_2(\R{n|m})$ we need to find a lifting of the classical cluster structure of $G_2(\R{n})$ by introducing odd variables into the clusters and figuring out   new mutations.

Recall the super Pl\"{u}cker variables and super Pl\"{u}cker relations obtained in Section~\ref{sec.simpl} and Section~\ref{sec.general}. In Section~\ref{sec.simpl}, we have identified among all the components of a non-degenerate $T\in\L^2(V)$ those components that we call  ``essential''. These are the even variables $T^{ab}=-T^{ba}$, $a,b=1,\ldots,n$, and the odd variables $\t^{a\mu}=-\t^{\mu a}$, $\mu= 1\ldots,  m$. (The sets of values for $a,b$ and $\mu$ are regarded as subsets of different copies of $\N$.) In the case of $m=1$ we simply write $\t^{a}$. There are also the   ``inessential components'' $S^{\mu\nu}$, which are even and  nilpotent; they can  be eliminated by using  part of  the super Pl\"{u}cker relations, namely  $T^{ab}S^{\la\mu}=-(\t^{a\la}\t^{b\mu}+\t^{a\mu}\t^{b\la})$, so that
the other super Pl\"{u}cker relations containing them will be automatically satisfied (see Section~\ref{subsec.ex}). In the approach developed in  Section~\ref{sec.general}, the ``inessential'' variables simply do not arise.
Only the essential variables will be important for super cluster structure.

Recall the super Pl\"{u}cker relations for the essential Pl\"{u}cker coordinates for $G_2(\R{n|m})$, i.e. equations~\eqref{eq.plueck.even} and \eqref{eq.plueck.odd}   from Theorem~\ref{thm.reducplueck}.
It   is convenient to write them as
\begin{align}
    T^{ac}T^{bd} = T^{ab}T^{cd} + T^{ad}T^{bc} \label{eq.even}\\
    \intertext{and}
    T^{ab}\t^{c\mu}= T^{ac}\t^{b\mu}+ \t^{a\mu}T^{cb} \label{eq.odd}
\end{align}
(emphasizing the fact that in the latter relation even and odd indices are on an equal footing). All indices $a,b,c,d$ in~\eqref{eq.even} and \eqref{eq.odd} have to be different. Note that to be able to obtain by a mutation an odd variable $\t^{c\mu}$, we must have in the initial cluster an even variable $T^{ab}$ (which must be a cluster variable since we will need to divide by it) together with the odd variables $\t^{a\mu}, \t^{b\mu}$. Therefore variables  $T^{ab}$ and $\t^{a\mu}, \t^{b\mu}$ should always go together in a cluster. This inevitably leads to a departure from one property of   classical cluster structure, namely that a mutation changes only one cluster variable. Namely, if we mutate an even variable $T^{ab}$, we must take along with it and mutate at the same time the odd variables $\t^{a\mu}$ and $\t^{b\mu}$.

\begin{example}[Super cluster structure of $G_2(\R{4|1})$]
\label{ex.supcl24}
We have   even variables $T^{ab}=-T^{ba}$ (so essentially six variables $T^{12}$, $T^{13}$, $T^{14}$, $T^{23}$, $T^{24}$, $T^{34}$) and odd variables $\t^a$,  with $a,b=1,\ldots,4$.   In addition to the  classical relation
\begin{equation}\label{eq.plueckclassicagain}
     \boxed{T^{13}}\,\boxed{T^{24}}= T^{12}T^{34}+T^{14}T^{23}
\end{equation}
involving only $T^{ab}$, there are the following odd Pl\"{u}cker relations:
\begin{align}
    \boxed{T^{13}}\,\boxed{\t^2}&= T^{12}\,\boxed{\t^3} + \boxed{\t^1}\,T^{23}\,,  \label{eq.oddmut123}\\
    \boxed{T^{13}}\,\boxed{\t^4}&= T^{14}\,\boxed{\t^3} + \boxed{\t^1}\,T^{43}\,,  \label{eq.oddmut134}  \\
     \boxed{T^{24}}\,\boxed{\t^1}&= T^{21}\,\boxed{\t^4} + \boxed{\t^2}\,T^{14}\,,   \label{eq.oddmut124} \\
       \boxed{T^{24}}\,\boxed{\t^3}&= T^{23}\,\boxed{\t^4} + \boxed{\t^2}\,T^{34}\,.    \label{eq.oddmut234}
\end{align}
The meaning of boxing some variables will be clear in a moment. We define the following \emph{``super clusters''}: $(T^{13}\,|\,\t^1,\t^3)$ and $(T^{24}\,|\,\t^2,\t^4)$. (From the above arguments, there is no other choice.) So above  boxed are all cluster variables. The variables $T^{12}, T^{23}, T^{34}, T^{14}$ are stable (as in the classical case).  There are two \emph{``even  super cluster mutations''} going in the opposite directions. From $(T^{13}\,|\,\t^1,\t^3)$ to $(T^{24}\,|\,\t^2,\t^4)$ we use \eqref{eq.plueckclassicagain} and \eqref{eq.oddmut123}, \eqref{eq.oddmut124} in order to express $T^{24}$ via $T^{13}$   and also $\t^2$, $\t^4$ via $T^{13}$, $\t^1$, $\t^3$ (and the stable variables as well, which we shan't mention in the future). And to go from $(T^{24}\,|\,\t^2,\t^4)$ to $(T^{13}\,|\,\t^1,\t^3)$ we similarly use \eqref{eq.plueckclassicagain} and \eqref{eq.oddmut124}, \eqref{eq.oddmut234}. Graphically this can be represented as
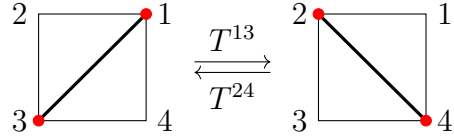
\begin{center}
\begin{tikzpicture}
\node (A) {\sqot};
\node (B) [right =of A] {\sqdc};
\path [->] let \p1=($(A)-(B)$),\n1={atan2(\y1,\x1)},\n2={180+\n1} in
     ($ (B.\n1)!2pt!90:(A.\n2) $) edge node {$T^{24}$} ($ (A.\n2)!2pt!-90:(B.\n1) $);
    \path [->] let \p1=($(B)-(A)$),\n1={atan2(\y1,\x1)},\n2={180+\n1} in
    ($ (A.\n1)!2pt!90:(B.\n2) $) edge node {$T^{13}$} ($ (B.\n2)!2pt!-90:(A.\n1) $);
\end{tikzpicture}
\end{center}
Here the even cluster variables are represented by proper diagonals of the square and odd cluster variables by red vertices. Arrows between the squares are the two mutations. They are labeled by the mutated even cluster variable (the one by which we divide). It is a little exercise to check that these two mutations are indeed mutually inverse for odd variables (for even variables, it is true by the definition). Basically, one has to check that if $\t^1,\t^3$ are expressed from $T^{24},\t^2,\t^4$ by using the odd relations~\eqref{eq.oddmut124}, \eqref{eq.oddmut234} and then $\t^2$, $\t^4$ are expressed from  $T^{13}$, $\t^1$, $\t^3$ by \eqref{eq.oddmut123}, \eqref{eq.oddmut124}, it will give the initial variables. For example, for $\t^1$ we have
\begin{multline*}
    \t^1=\frac{1}{T^{24}}(T^{21}\,\t^4 + \t^2\,T^{14})=
    \frac{T^{13}}{T^{12}T^{34}+T^{14}T^{23}}\left(T^{21}\,\frac{T^{14}\,\t^3 + \t^1\,T^{43}}{T^{13}} + \frac{T^{12}\,\t^3  +  \t^1\,T^{23}}{T^{13}}\,T^{14}\right) \\
    =\frac{T^{21}\,(T^{14}\,\t^3 + \t^1\,T^{43}) + (T^{12}\,\t^3  +  \t^1\,T^{23})\,T^{14}}{T^{12}T^{34}+T^{14}T^{23}} \\
    =\frac{(T^{21}T^{14}+T^{12}T^{14})\,\t^3
    + \t^1(T^{21}T^{43}+T^{23}T^{14})}%
    {T^{12}T^{34}+T^{14}T^{23}}=\t^1\,,
\end{multline*}
where we used the even relation~\eqref{eq.plueckclassicagain}. Note that all the expected properties are satisfied. In particular, by starting from one super cluster we can generate all  other  variables.
\end{example}

This example of $G_2(\R{4|1})$ does not exhibit all ``super cluster'' features. Namely, there is no situation when odd variables are mutated without simultaneously mutating an even variable. This we will observe in the next example.

\begin{example}[Super cluster structure of $G_2(\R{5|1})$]
\label{ex.supcl25}
Now we have even variables $T^{ab}=-T^{ba}$ (essentially ten even variables) and odd variables $\t^a$, $a,b=1,\ldots,5$. The ``super cluster'' picture is given by the following exchange graph (explained below):
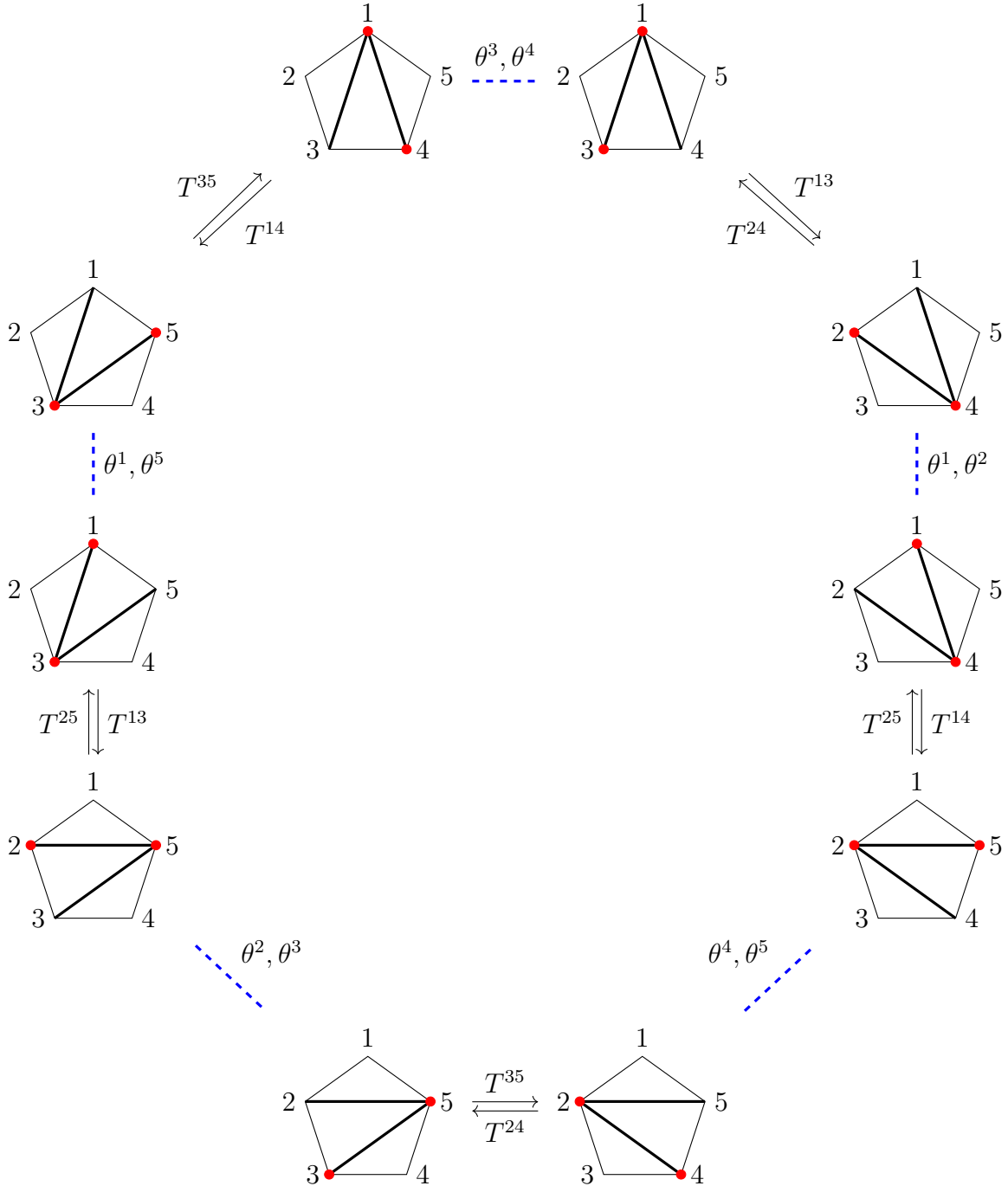
\begin{center}
\pentapic
\end{center}
Recall that in the classical case, i.e. for $G_2(\R{5|1})$,  the clusters correspond to the triangulations of the pentagon by proper diagonals and there are five such clusters
(see~\cite[\S2.1.3]{gekhshava:book2010}). To obtain super clusters, we need to include odd variables. Arguing as above, we arrive at   ten \emph{super clusters} naturally grouped into pairs (corresponding to the classical clusters).
\begin{align*}
    (T^{13}, T^{14}\,|\,\t^1,\t^3), \
    (T^{13}, T^{35}\,|\,\t^3,\t^5), \
    (T^{25}, T^{35}\,|\,\t^2,\t^5), \
    (T^{24}, T^{25}\,|\,\t^2,\t^4), \
    (T^{14}, T^{24}\,|\,\t^1,\t^4), \\
    (T^{13}, T^{14}\,|\,\t^1,\t^4), \
    (T^{13}, T^{35}\,|\,\t^1,\t^3), \
    (T^{25}, T^{35}\,|\,\t^3,\t^5), \
    (T^{24}, T^{25}\,|\,\t^2,\t^5), \
    (T^{14}, T^{24}\,|\,\t^2,\t^4)\,.
\end{align*}
From each cluster there are two possible \emph{(super) mutations}: an \emph{even super mutation} (like in the previous example) that exchanges one even variable and two accompanying odd variables, and an \emph{odd super mutation} that exchanges one odd variable (leaving other variables, even and odd, unchanged). The graph  
shows all super clusters and (super)mutations.  Even mutations are depicted by arrows, odd mutations are depicted by blue dashed lines. Arrows   are marked by even variables that are  mutated, dashed lines are marked by pairs of odd variables that are mutated.
Write, for example, formulas for mutations for the cluster $(T^{13}, T^{14}\,|\,\t^1,\t^4)$, the top left on the picture. The \textbf{even mutation} (on the picture, the arrow marked with $T^{14}$):
\begin{align}
    \boxed{\underline{T^{14}}}\,\boxed{\underline{T^{35}}}&= \boxed{T^{13}}\,T^{45} + T^{15}\,T^{34}\,,\label{eq.t14t35}\\
    \boxed{\underline{T^{14}}}\,\boxed{\underline{\t^{3}}}&= \boxed{T^{13}}\,\boxed{\underline{\t^{4}}} + \boxed{\underline{\t^{1}}}\,T^{34}\,,\label{eq.t14th3ev}\\
    \boxed{\underline{T^{14}}}\,\boxed{\underline{\t^{5}}}&= T^{15}\,\boxed{\underline{\t^{4}}} + \boxed{\underline{\t^{1}}}\,T^{54}\,.\label{eq.t14th5ev}
\end{align}
By this mutation we obtain the cluster $(T^{13}, T^{35}\,|\,\t^3,\t^5)$. The \textbf{odd mutation} (on the picture, the dashed line marked with $\t^3,\t^4$):
\begin{equation}\label{eq.t14th3od}
    \boxed{T^{14}}\,\boxed{\underline{\t^{3}}} = \boxed{T^{13}}\,\boxed{\underline{\t^{4}}} + \boxed{\t^{1}}\,T^{34}\,.
\end{equation}
By this mutation we obtain the cluster $(T^{13}, T^{14}\,|\,\t^1,\t^3)$.
Boxed are the cluster variables, not boxed are the ground variables (which are the same as in the classical case). We have also underlined for clarity the variables that are exchanged by a mutation. Formulas~\eqref{eq.t14th3ev} and~\eqref{eq.t14th3od} are the same, but play different roles: for the latter, the even cluster variable $T^{14}$ is not itself mutated. While it is obvious that the formula equivalent to~\eqref{eq.t14th3od},
\begin{equation}\label{eq.t13th4od}
    \boxed{T^{13}}\,\boxed{\underline{\t^{4}}} = \boxed{T^{14}}\,\boxed{\underline{\t^{3}}} + \boxed{\t^{1}}\,T^{43}\,,
\end{equation}
gives the inverse odd mutation (from $(T^{13}, T^{14}\,|\,\t^1,\t^3)$ to $(T^{13}, T^{14}\,|\,\t^1,\t^4)$), it is a little exercise to show that the even mutation $(T^{13}, T^{35}\,|\,\t^3,\t^5)$ to  $(T^{13}, T^{14}\,|\,\t^1,\t^3)$ given by
\begin{align}
    \boxed{\underline{T^{35}}}\,\boxed{\underline{T^{14}}} &= \boxed{T^{31}}\,T^{54} + T^{34}\,T^{15}\,,\\
    \boxed{\underline{T^{35}}}\,\boxed{\underline{\t^{1}}}&=
    \boxed{T^{31}}\,\boxed{\underline{\t^{5}}} + \boxed{\underline{\t^{3}}}\,T^{15}\,, \\
    \boxed{\underline{T^{35}}}\,\boxed{\underline{\t^{4}}}&=
    T^{34}\,\boxed{\underline{\t^{5}}} + \boxed{\underline{\t^{3}}}\,T^{45}
\end{align}
and the even mutation given by \eqref{eq.t14t35}, \eqref{eq.t14th3ev}, \eqref{eq.t14th5ev} are mutually inverse. The part involving only $T^{14}$ and $T^{35}$ is of course obvious. One needs to check for the odd variables. It works as in Example~\ref{ex.supcl24} and we skip the calculation.
\end{example}

As mentioned, we hope such a picture extends to $G_2(\R{n|m})$ with general $n|m$ (and further to arbitrary $G_r(\R{n|m})$). For example, for $G_2(\R{n|1})$, the super clusters look as in Examples~\ref{ex.supcl24} and~\ref{ex.supcl25}: they contain the same even variables as for the underlying ordinary Grassmannian $G_2(\R{n})$ labeled by   proper diagonals giving a triangulation of the $n$-gon and additionally 
$2$ odd variables labeled by the vertices of one of the proper diagonals which is marked (and the frozen variables are the same as in the classical case, corresponding to the sides).

It is intriguing if the more complicated super Pl\"{u}cker relations such as in Theorem~\ref{thm.splueck11} and Theorem~\ref{thm.splueckrs} can be also   related with some sort of a (super) cluster structure.



\def\cprime{$'$}

\end{document}